\documentclass[12 pt,reqno]{amsart}
\usepackage{pdfpages}
\usepackage{amsmath,amsthm,amssymb,amscd}
\usepackage{setspace}
\usepackage{upgreek}
\usepackage{cite}
\usepackage{url}
\usepackage{mathtools}
\usepackage{fullpage}
\usepackage{float}
\usepackage{comment}
\usepackage[bookmarksopen,bookmarksdepth=3]{hyperref}

\input xy
\xyoption{all}

\usepackage{microtype}

\allowdisplaybreaks[1]

\newtheorem{thm}{Theorem}
\newtheorem{prop}{Proposition}[section]
\newtheorem{lm}[prop]{Lemma}
\newtheorem{cor}[prop]{Corollary}

\newtheorem{cl}[prop]{Proposition}

\theoremstyle{definition}
\newtheorem{dfn}[prop]{Definition}

\theoremstyle{remark}
\newtheorem{rem}[prop]{Remark}
\newtheorem{ex}[prop]{Example}

\DeclareMathOperator{\Coker}{Coker}

\DeclareMathOperator{\Id}{Id}

\DeclareMathOperator{\im}{Im}

\DeclareMathOperator{\ogw}{OGW}

\newcommand{\fix}{\textit{fix}}

\newcommand{\rarr}{\rightarrow}
\newcommand{\lrarr}{\longrightarrow}
\newcommand{\Rarr}{\Rightarrow}

\newcommand{\R}{\mathbb{R}}
\newcommand{\C}{\mathbb{C}}
\newcommand{\Z}{\mathbb{Z}}
\newcommand{\Q}{\mathbb{Q}}

\newcommand{\M}{\mathcal{M}}

\renewcommand{\P}{\mathbb{C}P}
\renewcommand{\L}{\Lambda}
\renewcommand{\l}{\lambda}
\newcommand{\mI}{\mathcal{I}}
\newcommand{\g}{\Gamma}

\newcommand{\Hh}{\widetilde{H}}
\newcommand{\J}{\mathcal{J}}

\newcommand{\qkl}{\mathfrak{q}_{\,k\!,\:l}}
\newcommand{\q}{\mathfrak{q}}

\newcommand{\m}{\mathfrak{m}}
\renewcommand{\d}{\partial}

\newcommand{\at}{\tilde{\alpha}}

\newcommand{\xit}{\tilde{\xi}}
\newcommand{\etat}{\tilde{\eta}}

\newcommand{\bt}{\tilde b}
\newcommand{\gt}{\tilde\gamma}
\newcommand{\mt}{\tilde\m}
\newcommand{\qt}{\tilde\q}
\newcommand{\ct}{\tilde{c}}
\newcommand{\ot}{\tilde{o}}
\newcommand{\mC}{\mathfrak{C}}
\newcommand{\mD}{\mathfrak{D}}
\newcommand{\mR}{\mathfrak{R}}
\newcommand{\evbt}{\widetilde{evb}}
\newcommand{\evit}{\widetilde{evi}}
\newcommand{\evt}{\widetilde{ev}}
\newcommand{\bh}{\hat{b}}
\newcommand{\oh}{\hat{o}}

\renewcommand{\u}{\upsilon}

\newcommand{\mg}{\m^{\gamma}}
\newcommand{\mgp}{\m^{\gamma'}}
\newcommand{\mgt}{\mt^{\gt}}

\renewcommand{\ll}{\langle\!\langle}
\renewcommand{\gg}{\rangle\!\rangle}

\newcommand{\T}{\mathcal{T}}

\renewcommand{\Im}{\im}

\newcommand{\Oh}{\widehat{\Omega}}

\newcommand{\RP}{\mathbb{R}P}

\newcommand{\sly}{\Pi}
\newcommand{\pr}{\varpi}

\newcommand{\D}{\mathcal{D}}
\newcommand{\bb}{\bar{b}}
\newcommand{\sababa}{sababa}

\newcommand{\Ups}{\Upsilon}

\newcommand{\Mt}{\widetilde{\M}}

\newcommand{\s}{\mathfrak{s}}

\renewcommand{\a}{\alpha}
\newcommand{\Ah}{\widehat{A}}
\newcommand{\Hhh}{\widehat{H}}
\newcommand{\tg}{\hat t}
\newcommand{\Tg}{{\hat T}}
\newcommand{\sg}{{\hat s}}
\newcommand{\Gg}{{\hat \g}}
\newcommand{\gag}{{\hat \gamma}}
\newcommand{\ssly}{{S}}

\author[J. Solomon]{Jake P. Solomon}
\address{Institute of Mathematics\\ Hebrew University, Givat Ram\\Jerusalem, 9190401, Israel } \email{jake@math.huji.ac.il}
\author[S. Tukachinsky]{Sara B. Tukachinsky}
\address{School of Mathematical Sciences\\ Tel Aviv University\\Tel Aviv, 6997801, Israel }\email{sarabt1@gmail.com}

\begin{document}
\title{Point-like bounding chains in open Gromov-Witten theory}

\date{Sept. 2021}

\begin{abstract}
We present a solution to the problem of defining genus zero open Gromov-Witten invariants with boundary constraints for a Lagrangian submanifold of arbitrary dimension. Previously, such invariants were known only in dimensions $2$ and $3$ from the work of Welschinger. Our approach does not require the Lagrangian to be fixed by an anti-symplectic involution, but can use such an involution, if present, to obtain stronger results. Also, non-trivial invariants are defined for broader classes of interior constraints and Lagrangian submanifolds than previously possible even in the presence of an anti-symplectic involution. The invariants of the present work specialize to invariants of Welschinger, Fukaya, and Georgieva in many instances.

The main obstacle to defining open Gromov-Witten invariants with boundary constraints in arbitrary dimension is the bubbling of $J$-holomorphic disks. Unlike in low dimensions or for interior constraints, disk bubbles do not cancel in pairs by anti-symplectic involution symmetry. Rather, we use the technique of bounding chains introduced in Fukaya-Oh-Ohta-Ono's work on Lagrangian Floer theory to cancel disk bubbling. At the same time and independently, gauge equivalence classes of bounding chains play the role of boundary constraints, in place of the cohomology classes that usually serve as constraints in Gromov-Witten theory. A crucial step in our construction is to identify a canonical up to gauge equivalence family of ``point-like'' bounding chains, which specialize in dimensions $2$ and $3$ to the point constraints considered by Welschinger.
\end{abstract}

\maketitle
\vspace{-2.5em}

\tableofcontents

\section{Introduction}\label{sec:intro}
\subsection{Overview}\label{ssec:intro^2}
It is a basic problem in symplectic geometry to count $J$-holomorphic curves in a symplectic manifold $X$ with boundary in a Lagrangian submanifold $L \subset X$ representing a fixed homology class $d \in H_2(X,L).$ When the dimension of the moduli space of $J$-holomorphic curves is positive, one obtains a finite number by restricting attention to $J$-holomorphic curves passing through a collection of geometric constraints. The count should only depend on the symplectic manifold $X$, the Lagrangian submanifold $L$ and a collection of cohomology classes or similar topological information that classifies the geometric constraints. The count should not depend on the choice of almost complex structure $J$ so long as it is tamed by the symplectic structure nor on the choice of geometric constraints so long as the topological information is unchanged. For example, one would like to count $J$-holomorphic disks $u: (D^2,\partial D^2) \to (X,L)$ representing $d \in H_2(X,L)$ with boundary constrained to pass through a generic collection of $k$ points on $L$ in such a way that the count does not depend on which $k$ points are chosen. Such a count of $J$-holomorphic curves, when it can be defined, is called an open Gromov-Witten invariant. A major obstacle to defining open Gromov-Witten invariants is that as one varies relevant choices in a one parameter family, $J$-holomorphic curves with boundary can degenerate and disappear through the process of a disk bubbling off at the boundary. See Figure~\ref{fig:bubble}.

\begin{figure}[!ht]
\centering
\includegraphics[width=12cm]{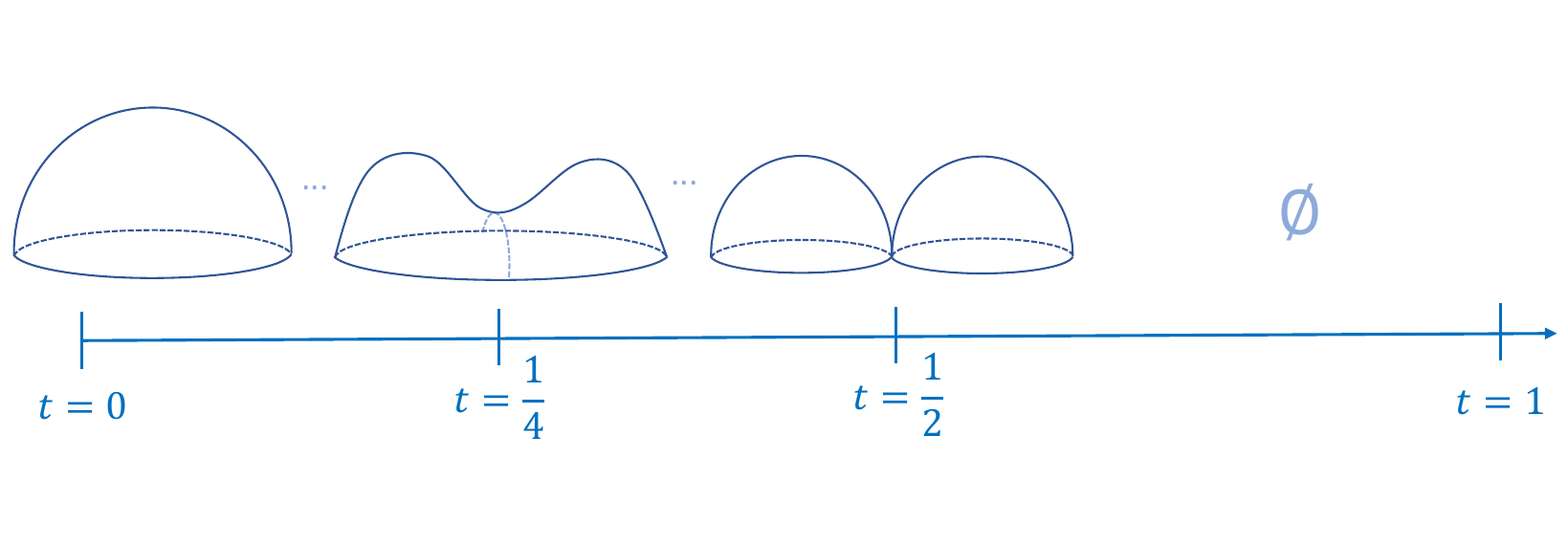}
\caption{Bubbling at the boundary}
\label{fig:bubble}
\end{figure}

An important example of open Gromov-Witten invariants arises from a reformulation~\cite{SolomonThesis} of Welschinger's invariant counts of real rational curves~\cite{Welschinger1,Welschinger2}. In the language of open Gromov-Witten invariants, one can explain the idea of Welschinger's invariants as follows: When the Lagrangian submanifold $L$ is fixed by an anti-symplectic involution, disk bubbles come in pairs that sometimes are counted with opposite signs and thus cancel out. When $\dim L = 2$ or $3,$ this idea indeed gives rise to open Gromov-Witten invariants that count $J$-holomorphic disks with boundary constrained to pass through a given collection of $k$ points on $L$ and interior constrained to pass through a give collection of $l$ invariant pairs of points on $X.$ In higher dimensions, the idea of cancelling disk bubbles in pairs by anti-symplectic involution only works for certain types of interior constraints under certain topological assumptions on $(X,L),$ as shown by Georgieva~\cite{Georgieva}.

Since Welschinger's initial work in 2003, it has remained an open problem to define invariant counts of $J$-holomorphic disks with boundary constraints in dimensions greater than $3.$ We present a solution to this problem using the theory of bounding chains introduced by Fukaya-Oh-Ohta-Ono in their seminal work on Lagrangian Floer cohomology~\cite{FOOO}. We define invariants in arbitrary dimension with both boundary and interior constraints, which specialize to the invariants of Welschinger and Georgieva~\cite{Welschinger1,Welschinger2,Georgieva} when the latter are defined. Moreover, our definition does not require the Lagrangian submanifold $L$ to be fixed by an anti-symplectic involution or real structure. Nonetheless, our method can use an anti-symplectic involution to obtain stronger results. Indeed, such an involution forces some of the obstruction classes of Fukaya-Oh-Ohta-Ono to vanish. In fact, when $\dim L = 2$ or~$3,$ all the obstruction classes are forced to vanish, which explains why obstruction classes did not appear in Welschinger's original work. Even in the case the Lagrangian is fixed by an anti-symplectic involution, our definition of open Gromov-Witten invariants applies to Lagrangian submanifolds for which no invariants could be defined previously, even with only interior constraints.

Our definition of open Gromov-Witten invariants uses the Fukaya-Oh-Ohta-Ono theory of bounding chains somewhat differently from the way it is usually used in Lagrangian Floer theory. In the context of Lagrangian Floer theory, bounding chains are introduced to cancel disk bubbling. They are not necessary in cases where disk bubbling does not obstruct the definition of Lagrangian Floer theory. On the other hand, in the context of open Gromov-Witten invariants, even when disk bubbles do not spoil invariance, it is still useful to introduce bounding chains. Indeed, we show that bounding chains are the natural geometric constraints for the boundary of $J$-holomorphic disks, and thus they should be introduced any time we wish to count $J$-holomorphic disks with boundary constraints. When disk bubbles do not spoil invariance, for example in the situations considered by Welschinger~\cite{Welschinger1,Welschinger2}, bounding chains take a particularly simple form. Namely, they turn out to be cycles.

In Lagrangian Floer theory, for many purposes it suffices to prove the existence of bounding chains, and if there are many, the choice is not important. On the other hand, in the definition of open Gromov-Witten invariants, bounding chains play the role of boundary constraints, so the invariants depend on the choice. Thus, one of our main results is a canonical parameterization of the space of bounding chains up to gauge equivalence. In particular, we identify a family of bounding chains, called point-like, that generalize the point constraints familiar from Welschinger's invariants to arbitrary dimension.  Point-like bounding chains are defined uniquely up to gauge equivalence. We view gauge equivalence as a non-linear analog of cohomology.

The present paper builds on Fukaya's definition of open Gromov-Witten disk invariants for Maslov zero Lagrangians in Calabi-Yau threefolds~\cite{Fukaya2}. However, for $3$-dimensional Maslov zero Lagrangians, the space of $J$-holomorphic disks representing any $d \in H_2(X,L)$ is zero dimensional. So, boundary constraints of positive codimension, which are a major source of complexity in the present paper, would only yield vanishing invariants and are therefore not treated in~\cite{Fukaya2}.

In a companion paper~\cite{ST3}, we show that the generating function of our invariants satisfies a system of non-linear PDE called the open WDVV equations, which allows extensive computations. We also prove a wall crossing formula that relates interior and boundary constraints. In~\cite{ST3}, it is shown that the open WDVV equations combined with the wall crossing formula and the axioms proved in the present paper determine all invariants for $(X,L) = (\C P^n,\R P^n).$ In~\cite{HST}, the same technique is used to determine all invariants for $X$ a quadric hypersurface and $L = S^n$ its real locus as in Example~\ref{ex:XL} below.
This demonstrates that enumerative invariants defined using the Fukaya-Oh-Ohta-Ono theory of bounding chains can be effectively calculated.

The calculations of~\cite{ST3} show that our invariants are non-vanishing in dimensions greater than $3$ for both conjugation invariant and conjugation anti-invariant interior constraints as well as for boundary constraints. On the other hand, the invariants of~\cite{Georgieva} vanish if one or more of the interior constraints is conjugation invariant.

The definition of open Gromov-Witten invariants in the present paper is valid for an arbitrary Lagrangian submanifold given the virtual fundamental class technique of Kuranishi structures~\cite{Fukaya,Fukaya2,FOOO,FOOOtoricI,FOOOtoricII,FOOO1,FOOOinv,FOOOKSII,FOOOspec,FOOOKSVFC,FO}. The same is true for the proofs of four of the six axioms. The unit and divisor axioms require compatibility of the virtual fundamental class with the forgetful map of interior marked points, which has not yet been worked out in the Kuranishi structure formalism in the context of differential forms. Alternatively, it should be possible to use the polyfold theory of~\cite{HoferWysockiZehnder,HoferWysockiZehnder1,HoferWysockiZehnder2,HoferWysockiZehnder3,LiWehrheim}. To make the present paper more accessible, in a companion paper~\cite{ST1}, we give an essentially self-contained treatment of Fukaya $A_\infty$ algebras for Lagrangians which satisfy an analog of the convex condition familiar from algebraic geometry~\cite{FultonPandharipande,Welschinger2}. This treatment covers all properties of Fukaya $A_\infty$ algebras used in the present paper, and also suffices for the calculations of~\cite{ST3}. The treatment is sufficiently general to cover an infinite family of Lagrangian submanifolds for which no open Gromov-Witten invariants were defined previously, as mentioned above. The techniques introduced in this paper are no simpler for the Lagrangians covered by the treatment of~\cite{ST1} than in the general case.

\subsection{Statement of results}
\subsubsection{\texorpdfstring{$A_{\infty}$}{A-infinity} algebras}\label{ssec:iainf}
To formulate our results, we recall relevant notation from~\cite{ST1}. Consider a symplectic manifold $(X,\omega)$ of complex dimension $n$, and a connected, Lagrangian submanifold $L$ with relative spin structure~$\s =\s_L.$ Let $J$ be an $\omega$-tame almost complex structure on~$X.$ Denote by $\mu:H_2(X,L;\Z) \to \Z$ the Maslov index. Denote by $A^*(L)$ the ring of differential forms on $L$ with coefficients in $\R$.
Let $\sly = \sly_L$ be the quotient of $H_2(X,L;\Z)$ by a possibly trivial subgroup $\ssly_L$ contained in the kernel of the homomorphism $\omega \oplus \mu : H_2(X,L;\Z) \to \R \oplus \Z.$ Thus the homomorphisms $\omega,\mu,$ descend to $\sly.$ Denote by $\beta_0$ the zero element of $\sly.$ We use a Novikov ring $\L$ which is a completion of a subring of the group ring of $\sly$. The precise definition follows. Denote by $T^\beta$ the element of the group ring corresponding to $\beta \in \sly$, so $T^{\beta_1}T^{\beta_2} = T^{\beta_1 + \beta_2}.$ Then,
\begin{equation}\label{eq:Nov}
\L=\L_\omega =\left\{\sum_{i=0}^\infty a_iT^{\beta_i}\bigg|a_i\in\R,\beta_i\in \sly,\omega(\beta_i)\ge 0,\; \lim_{i\to \infty}\omega(\beta_i)=\infty\right\}.
\end{equation}
A grading on $\L$ is defined by declaring $T^\beta$ to be of degree $\mu(\beta).$
Denote also
\[
\L^+=\left\{\sum_{i=0}^\infty a_iT^{\beta_i} \in \L \bigg|\;\omega(\beta_i)> 0 \quad\forall i\right\}.
\]
In the case $S_L = \ker(\omega \oplus \mu),$ the above ring $\Lambda$ is a subring of the Novikov ring $\Lambda_{0,nov}(\R)$ from~\cite[Section 1.7]{FOOO}. For the present work, it is desirable to allow some flexibility in the choice of $S_L.$ Indeed, the smaller we take $S_L,$ the more control we have over the class in $H_2(X,L;\Z)$ of the disks being counted in the open Gromov-Witten invariants defined below. On the other hand, to make use of an anti-symplectic involution, it is generally necessary to consider $S_L$ non-trivial.

We use a family of $A_\infty$ structures on $A^*(L)\otimes \L$ following~\cite{Fukaya, FOOO}, based on the results of~\cite{ST1}. Let $\M_{k+1,l}(\beta)$ be the moduli space of genus zero $J$-holomorphic open stable maps $u:(\Sigma,\d \Sigma) \to (X,L)$ of degree $[u_*([\Sigma,\d \Sigma])] = \beta \in \sly$ with one boundary component, $k+1$ boundary marked points, and $l$ interior marked points. The boundary points are labeled according to their cyclic order. The space $\M_{k+1,l}(\beta)$ carries evaluation maps associated to boundary marked points $evb_j^\beta:\M_{k+1,l}(\beta)\to L$, $j=0,\ldots,k$, and evaluation maps associated to interior marked points $evi_j^\beta:\M_{k+1,l}(\beta)\to X$, $j=1,\ldots,l$.

We assume that all $J$-holomorphic genus zero open stable maps with one boundary component are regular, the moduli spaces $\M_{k+1,l}(\beta;J)$ are smooth orbifolds with corners, and the evaluation maps $evb_0^\beta$ are proper submersions.
Examples include $(\P^n,\RP^n)$ with the standard symplectic and complex structures or, more generally, flag varieties, Grassmannians, and products thereof. See~\cite[Example 1.4, Remark 1.5]{ST1}.
Throughout the paper we fix a connected component $\mathcal{J}$ of the space of $\omega$-tame almost complex structures satisfying our assumptions.  All almost complex structures are taken from $\J.$ The definition of open Gromov-Witten invariants in the present paper extends to arbitrary targets $(X,\omega,L)$ and $\mathcal{J}$ the space of all $\omega$-tame almost complex structures given the virtual fundamental class technique of Kuranishi structures~\cite{Fukaya,Fukaya2,FOOO,FOOOtoricI,FOOOtoricII,FOOO1,FOOOinv,FOOOKSII,FOOOspec,FOOOKSVFC,FO}. The same is true for the proofs of four of the six axioms. The unit and divisor axioms require compatibility of the virtual fundamental class with the forgetful map of interior marked points, which has not yet been worked out in the Kuranishi structure formalism in the context of differential forms. See Remark~\ref{rem:intforget}.
Alternatively, it should be possible to use the polyfold theory of~\cite{HoferWysockiZehnder,HoferWysockiZehnder1,HoferWysockiZehnder2,HoferWysockiZehnder3,LiWehrheim}.
The relative spin structure $\s$ determines an orientation on the moduli spaces $\M_{k+1,l}(\beta)$ as in~\cite[Chapter 8]{FOOO}.

Let $s, t_0,\ldots,t_N,$ be formal variables with $\deg s=1-n$.
In our definition of open Gromov-Witten invariants, the $s$ variable will be used to keep track of boundary constraints and the $t_j$ variables will be used to keep track of interior constraints.
For $m>0$ denote by
$A^m(X,L)$ differential $m$-forms on $X$ that vanish on $L$, and denote by $A^0(X,L)$ functions on $X$ that are constant on $L$. The exterior differential $d$ makes $A^*(X,L)$ into a complex. Alternatively, one can consider the complex $\Ah^*(X,L)$ of forms on $X$ with vanishing integral over $L$ as explained in  Section~\ref{ssec:relax}.
Set
\begin{gather*}
R:=\L[[s,t_0,\ldots,t_N]],\quad Q:=\R[t_0,\ldots,t_N] ,\\
C:=A^*(L)\otimes R,\quad\text{and}\quad D:= A^*(X,L)\otimes Q ,
\end{gather*}
where $\otimes$ is understood as the completed tensor product.
For an $\R$-algebra $\Upsilon$, write
\[
 \Hh^*(X,L;\Upsilon):=H^*(A^*(X,L)\otimes\Upsilon, d).
\]
Observe that
\[
\Hh^*(X,L;\Upsilon) \simeq \left(H^0(L;\R)\oplus H^{>0}(X,L;\R)\right)\otimes  \Upsilon,\quad \Hh^*(X,L;Q)=H^*(D).
\]
The gradings on $C,D,$ and $\Hh^*(X,L;Q)$, take into account the degrees of $s,t_j,T^\beta,$ and the degree of differential forms. Given a graded module $M$, we write $M_j$ or $(M)_j$ for the degree~$j$ part.
Let
\[
R^+:=R\L^+\triangleleft R,\quad\mathcal{I}_R:=\left<s,t_0,\ldots,t_N\right>+R^+ \triangleleft R,
\quad \mathcal{I}_Q:=\left<t_0,\ldots,t_N\right>\triangleleft Q,
\]
be the ideals generated by the formal variables.

Let $\gamma\in \mI_QD$ be a closed form with $\deg_D\gamma=2$. For example, given closed differential forms $\gamma_j\in A^*(X,L)$ for $j=0,\ldots,N,$ take $t_j$ of degree $2-\deg \gamma_j$ and $\gamma:=\sum_{j=0}^Nt_j\gamma_j$.
Define structure maps
\[
\mg_k:C^{\otimes k}\lrarr C
\]
by
\begin{multline*}
\m^{\gamma}_k(\alpha_1,\ldots,\alpha_k):=\\
=\delta_{k,1}\cdot d\alpha_1+(-1)^{\sum_{j=1}^kj(\deg \alpha_j+1)+1}
\sum_{\substack{\beta\in\sly\\l\ge0}}T^{\beta}\frac{1}{l!}{evb_0^\beta}_* (\bigwedge_{j=1}^l(evi_j^\beta)^*\gamma\wedge
\bigwedge_{j=1}^k (evb_j^\beta)^*\alpha_j
),
\end{multline*}
where $\delta_{k,1}$ is the Kronecker delta.
The push-forward $(evb_0^\beta)_*$ is defined by integration over the fiber; it is well-defined because $evb_0^\beta$ is a proper submersion. The condition $\gamma\in \mI_QD$ ensures that the infinite sum converges.
Intuitively, $\gamma$ should be thought of as interior constraints, while $\alpha_j$ are boundary constraints. Then the output is a cochain on $L$ that is ``Poincar\'e dual'' to the image of the boundaries of disks that satisfy the given constraints. In~\cite{FOOO,Fukaya,ST1}, as summarized in Proposition~\ref{cl:a_infty_m} below, it is shown that $(C,\{\mg_k\}_{k\ge 0})$ is an $A_\infty$ algebra.
Furthermore, following~\cite{Fukaya2}, define
\[
\m_{-1}^\gamma:=\sum_{\substack{\beta\in\sly\\l\ge 0}}\frac{1}{l!}T^{\beta}\int_{\M_{0,l}(\beta)}\bigwedge_{j=1}^l (evi_j^\beta)^*\gamma.
\]
\subsubsection{Bounding pairs and the superpotential}\label{sssec:bno}
Our strategy is to extract OGW invariants from the superpotential. For us, the superpotential is a function on the space of (weak) bounding pairs:
\begin{dfn}\label{dfn_bd_pair}
A \textbf{bounding pair} with respect to $J$ is a pair $(\gamma,b)$ where $\gamma\in \mathcal{I}_Q D$ is closed with $\deg_D\gamma=2$ and $b\in \mathcal{I}_R C$ with
$
\deg_Cb=1,
$
such that
\begin{equation}\label{eq:bc}
\sum_{k\ge 0}\m_k^\gamma(b^{\otimes k})=c\cdot 1, \qquad c\in \mI_R,\;\deg_Rc=2.
\end{equation}
In this situation, $b$ is called a \textbf{bounding chain} for $\mg$.
\end{dfn}

\begin{rem}
Bounding chains were introduced in~\cite{FOOO}. A bounding pair $(\gamma,b)$ in our terminology translates in the terminology of~\cite{FOOO} to $b$ being a weak bounding cochain for the $A_\infty$ algebra of $L$ bulk-deformed by $\gamma$.
The constant $c$ can be thought of as a function of $b$ and as such is often denoted by~$\mathfrak{PO}$ in~\cite{FOOO}.
\end{rem}

Following~\cite{Fukaya2}, the standard superpotential is given by
\[
\Oh(\gamma,b):=\Oh_J(\gamma,b)
:=(-1)^{n}\big(\sum_{k\ge0}\frac{1}{(k+1)}\langle\m_k^\gamma (b^{\otimes k}),b\rangle
+\m_{-1}^{\gamma}\big).
\]
Intuitively, $\Oh$ counts $J$-holomorphic disks with constraints $\gamma$ in the interior and $b$ on the boundary.
Modification is necessary in order to avoid $J$-holomorphic disks the boundary of which can degenerate to a point, forming a $J$-holomorphic sphere.
We say that a monomial element of $R$ is \textbf{of type $\mathcal{D}$} if it has the form $a\, T^{\beta}s^0t_0^{j_0}\cdots t_N^{j_N}$ with $a\in \R$ and $\beta\in \Im(H_2(X;\Z)\to \sly)$.
In the present paper, the superpotential is defined by
\[
\Omega(\gamma,b):=\Omega_J(\gamma,b) : = \Oh_J(\gamma,b)-\text{ all monomials of type }\mathcal{D}\text{ in }\Oh_J.
\]
Intuitively, $\Omega$ counts $J$-holomorphic disks with an arbitrary number of constraints $\gamma$ in the interior and either a positive number of constraints on the boundary or of degree $\beta \in \sly$ that does not arise from an absolute homology class. So, the boundary of a $J$-holomorphic disk counted in $\Omega$ cannot collapse to a point to form a sphere in a generic one parameter family.

Definition~\ref{dfn_g_equiv} gives a notion of gauge equivalence between a bounding pair $(\gamma,b)$ with respect to $J$ and a bounding pair $(\gamma',b')$ with respect to another almost complex structure~$J'.$ Let $\sim$ denote the resulting equivalence relation.
\begin{thm}[Invariance of the super-potential]\label{thm_inv}
If $(\gamma,b)\sim(\gamma',b')$, then $\Omega_J(\gamma,b)=\Omega_{J'}(\gamma',b')$.
\end{thm}

To obtain invariants from $\Omega,$ we must understand the space of gauge equivalence classes of bounding pairs.
Assume $n>0$.
Define a map
\[
\varrho:\{\text{bounding pairs}\}/\sim\;\;\lrarr \;(\mI_Q\Hh^*(X,L;Q))_2\oplus (\mI_R)_{1-n}
\]
by
\begin{equation}\label{eqn_rho}
\varrho([\gamma,b]):=\left([\gamma]\,,\int_Lb\right).
\end{equation}
We prove in Lemma~\ref{lm_rho} that $\varrho$ is well defined.
\begin{thm}[Classification of bounding pairs -- rational cohomology spheres]\label{thm1}
Assume $H^*(L;\R)=H^*(S^n;\R)$. Then
$\varrho$ is bijective.
\end{thm}

In the presence of an anti-symplectic involution, the assumptions on $L$ can be relaxed. A \textbf{real setting} is a quadruple $(X,L,\omega,\phi)$ where $\phi:X\to X$ is an anti-symplectic involution such that $L\subset \fix(\phi)$. Throughout the paper, whenever we discuss a real setting, we fix a connected subset $\J_\phi \subset \J$ consisting of $J \in \J$ such that $\phi^*J = -J.$ All almost complex structures of a real setting are taken from $\J_\phi.$ If we use virtual fundamental class techniques, we can treat any $\omega$-tame almost complex structure $J$ satisfying $\phi^*J = -J.$ Whenever we discuss a real setting, we take $\ssly_L \subset  H_2(X,L;\Z)$ with $\Im(\Id+\phi_*) \subset \ssly_L,$ so $\phi_*$ acts on $\sly_L = H_2(X,L;Z)/\ssly_L$ as $-\Id.$ Also, the formal variables $t_i$ have even degree. We denote by $\Hh_\phi^{even}(X,L;\R)$ (resp. $H_\phi^{even}(X;\R)$) the direct sum over $k$ of the $(-1)^{k}$-eigenspace of $\phi^*$ acting on $\Hh^{2k}(X,L;\R)$ (resp. $H^{2k}(X;\R)$). Note that the Poincar\'e duals of $\phi$-invariant almost complex submanifolds of $X$ disjoint from $L$ belong to $\Hh_\phi^{even}(X,L;\R).$
Extend the action of $\phi^*$ to $\L,Q,R,C,$ and $D,$ by taking
\begin{equation}\label{eq:phi*ext}
\phi^*T^\beta = (-1)^{\mu(\beta)/2}T^\beta, \qquad \phi^* t_i = (-1)^{\deg t_i/2}t_i, \qquad \phi^* s = -s.
\end{equation}
Elements $a \in \L,Q,R,C,D,$ and pairs thereof are called \textbf{real} if
\begin{equation}\label{eq:relt}
\phi^* a = -a.
\end{equation}
For a group $Z$ on which $\phi^*$ acts, let $Z^{-\phi^*}\subset Z$ denote the elements fixed by $-\phi^*.$  Let
\[
\varrho_\phi:\{\text{real bounding pairs}\}/\sim\;\;\lrarr\;(\mI_Q\Hh^*(X,L;Q))_2^{-\phi^*} \oplus(\mI_R)_{1-n}
\]
be given by the same formula as $\varrho$.
Then, we obtain the following variant of Theorem~\ref{thm1}.
\begin{thm}[Classification of bounding pairs -- real spin case]\label{thm2}
Suppose $(X,L,\omega,\phi)$ is a real setting, $\s$ is a spin structure,
and $n \not \equiv 1 \pmod 4.$ Moreover,
\begin{itemize}
\item
if $n \equiv 3 \pmod 4,$ assume $H^i(L;\R) \simeq H^i(S^n;\R)$ for $i \equiv 0,3 \pmod 4$;
\item
if $n \equiv 2 \pmod 4,$ assume $H^i(L;\R) \simeq H^i(S^n;\R)$ for $i \not \equiv 1 \pmod 4$;
\item
if $n \equiv 0 \pmod 4,$ assume $H^i(L;\R) \simeq H^i(S^n;\R)$ for $i \not \equiv 2 \pmod 4$.
\end{itemize}
Then $\varrho_\phi$ is bijective.
\end{thm}
\begin{rem}\label{rem:dim23}
In the special cases when $n=2,3,$ the cohomological assumption is always satisfied. This explains the significance of dimensions $2$ and $3$ in Welschinger's work~\cite{Welschinger1,Welschinger2}. See Theorem~\ref{Welsch} for a comparison of Welschinger's invariants with the invariants of the present work.
\end{rem}
\begin{rem}
In an earlier version of this paper in the case $n \equiv 3 \pmod 4,$ we used three-typical bounding chains instead of real bounding chains. Lemma~\ref{lm:reC} shows the two notions are equivalent.
\end{rem}

\subsubsection{Open Gromov-Witten invariants and axioms}\label{sssec:intro_OGW}
When the hypothesis of either Theorem~\ref{thm1} or~\ref{thm2} is satisfied, we define open Gromov-Witten invariants as follows. In the case of Theorem~\ref{thm2}, take $\ssly_L$ containing $\Im(\Id+\phi_*).$ In the case of Theorem~\ref{thm1} (resp. Theorem~\ref{thm2}) let $W_L = \Hh^*(X,L;\R)$
(resp. $W_L = \Hh^{even}_\phi(X,L;\R)$). Fix
$\g_0,\ldots,\g_N,$ a basis of $W_L$, set $\deg t_j=2-\deg(\g_j)$, and take
\[
\g:=\sum_{j=0}^Nt_j\g_j \in (\mI_Q\Hh^*(X,L;Q))_2.
\]
By Theorem~\ref{thm1} (resp. Theorem~\ref{thm2}), choose a bounding pair $(\gamma,b)$ such that
\begin{equation}\label{eq:cbp}
\varrho([\gamma,b])=(\g,s) \qquad  \text{(resp. $\varrho_\phi([\gamma,b])=(\g,s)).$}
\end{equation}
By Theorem~\ref{thm_inv}, the superpotential $\Omega = \Omega(\gamma,b)$ is independent of the choice of $(\gamma,b).$
\begin{dfn}\label{def:ogw}
The \textbf{open Gromov-Witten invariants} of $(X,L),$
\[
\ogw_{\beta,k}=\ogw_{\beta,k}^L : W_L^{\otimes l} \to \R,
\]
are defined by setting
\begin{equation*}
\ogw_{\beta,k}(\g_{i_1},\ldots,\g_{i_l}):= \text{ the coefficient of }T^{\beta}\text{ in }
\d_{t_{i_1}}\cdots\d_{t_{i_l}}\d_s^k\Omega|_{s=0,t_j=0}
\end{equation*}
and extending linearly to general input.
\end{dfn}

In Corollary~\ref{cor:indep} we show that the invariants $\ogw_{\beta,k}$ are independent of the choice of basis $\g_j.$ Also, it is not hard to see that the invariants arising from Theorem~\ref{thm1} and Theorem~\ref{thm2} coincide when both hypotheses are satisfied.

\begin{rem}\label{rem1}
Theorems~\ref{thm1} and~\ref{thm2} say that a bounding chain is determined up to equivalence by the cohomology class of its part that has degree $n$ in $A^*(L)$. In general, the degree $n$ part of $b$ must be ``corrected'' by non-closed forms of lower degrees in order to solve equation~\eqref{eq:bc}. The degree $n$ part represents a multiple of the Poincar\'e dual of a point. Equation~\eqref{eq:cbp} says that the degree $n$ part of the bounding chain parameterizes multiples of the point as the formal variable $s$ varies. Thus, we call such a bounding chain \emph{point-like}.
\end{rem}

\begin{ex}\label{ex:XL}
Consider the hypersurfaces
\[
X_{d,n} = \left \{\sum_{i = 0}^n z_i^d - z_{n+1}^d = 0 \right\} \subset \C P^{n+1}
\]
equipped with the symplectic form obtained by restricting the Fubini-Study form. Let $\phi : X_{d,n} \to X_{d,n}$ be the anti-symplectic involution given by complex conjugation of all coordinates and let $L_{d,n} = \fix(\phi).$ We claim that the Lagrangian submanifold $L_{d,n}\subset X_{d,n}$ satisfies the hypothesis of Theorem~\ref{thm1} for $d$ even and $n$ arbitrary or $d$ odd and $n$ odd. Indeed, for $d$ odd, we have $L_{d,n} \simeq \R P^n$ and for $d$ even we have $L_{d,n} \simeq S^n.$ So $L_{d,n}$ is a real cohomology sphere. Since $S^n$ is spin, so is $L_{d,n}$ for $d$ even. When $d$ is odd, we need $n$ odd for $L_{d,n}$ to be orientable. When $n \geq 3,$ it follows from the Lefschetz hyperplane theorem that the map $H^2(X) \to H^2(L)$ is surjective, so $L_{d,n}$ is relatively spin. For $n = 1,$ all manifolds are spin.

Furthermore, consider the Lagrangian submanifold $L = \prod_{j = 1}^k L_{d_j,n_j} \subset \prod_{j = 1}^k X_{d_j,n_j}.$ Supppose $d_j$ is even for $j = 1,\ldots,k.$ If $k = 3$ or $k = 2$ and $n_j \equiv 1 \pmod 4,$ for $j = 1,\ldots,k,$ then $L$ satisfies the hypothesis of Theorem~\ref{thm2}. Also, if $k = 2$ and $n_j \equiv 2 \pmod 4,$ for $j = 1,2,$ then $L$ satisfies the hypothesis of Theorem~\ref{thm2}.

So, in all these cases, the invariants $\ogw_{\beta,k}$ are defined. For $d = 1$ or $d = 2$ and for $n$ arbitrary, the pairs $(X_{d,n},L_{d,n})$ and products thereof satisfy the hypothesis of~\cite[Example~1.4]{ST1}, so the virtual fundamental class is not needed. On the other hand, for $n \geq 3$ the $\phi$-oriented condition of~\cite{Georgieva} is only satisfied for $X_{d,n}$ when either $d$ is even and $n$ is even or $d \equiv 1 \pmod 4$ and $n$ is odd.
Using the techniques of~\cite{ST3}, it is shown in~\cite{HST} that the invariants $\ogw_{\beta,k}$ of $(X_{d,n},L_{d,n})$ are non-trivial for $d = 2$ and $n$ arbitrary.
\end{ex}

\begin{rem}
When $\beta$ lies in the image of the natural map $H_2(X;\Z) \to \sly$ and $k = 0,$ the invariants $\ogw_{\beta,k}$ vanish. This is because monomials of type $\mathcal{D}$ were removed in the definition of the superpotential $\Omega.$ In particular, when $H_1(L;\Z) \simeq 0$ and thus the map $H_2(X;\Z) \to \sly$ is surjective, only the invariants $\ogw_{\beta,k}$ with $k > 0$ can be non-trivial. On the other hand, the calculations of~\cite{HST} mentioned above in Example~\ref{ex:XL} show that even when $L \simeq S^n,$ the invariants $\ogw_{\beta,k}$ are indeed non-trivial when $k > 0.$
\end{rem}

\begin{rem}\label{rem:neven}
When $n$ is even, it follows that $\deg s = 1-n$ is odd, so $s^2=0.$ Thus, for $k > 1,$ we have $\ogw_{\beta,k} = 0.$ In work in progress, the definition of the invariants $\ogw_{\beta,k}$ for $n$ even is modified so that they take non-zero values for $k > 1.$
\end{rem}

We next formulate properties of open Gromov-Witten invariants analogous to several standard properties of closed Gromov-Witten invariants as given, for example, in~\cite{KontsevichManin, MS, RuanTian}.

\begin{thm}[Axioms of the $\ogw$ invariants]\label{axioms}
The invariants $\ogw_{\beta,k}$ defined above have the following properties.
Let $A_j \in W_L$ for $j = 1,\ldots,l$.
\begin{enumerate}
	\item\label{it:deg} (Degree)
		$\ogw_{\beta,k}(A_1,\ldots,A_l)=0$ unless
		\begin{equation}\label{ax_deg}
		n-3+\mu(\beta)+k+2l = kn+\sum_{j=1}^l\deg A_j.
		\end{equation}
    \item (Symmetry)
        For any permutation $\sigma\in S_l$,
        \begin{equation}
        \ogw_{\beta,k}(A_1,\ldots,A_l)= (-1)^{s_\sigma(A)}\ogw_{\beta,k}(A_{\sigma(1)}\ldots,A_{\sigma(l)}),
        \end{equation}
      where $s_\sigma(A):
=
\sum_{\substack{i>j\\ \sigma(i)<\sigma(j)}} \deg A_{\sigma(i)}\cdot\deg A_{\sigma(j)}.
$
	\item (Unit / Fundamental class)
		\begin{equation}\label{ax_unit}
		\ogw_{\beta,k}(1,A_{1},\ldots,A_{l-1})=
		\begin{cases}
			-1, & (\beta,k,l)=(\beta_0,1,1),\\
			0, & \text{otherwise}.
		\end{cases}
		\end{equation}
	\item (Zero)
		\begin{equation}
		\ogw_{\beta_0,k}(A_1,\ldots,A_l)=
		\begin{cases}
		-1, & (k,l)=(1,1)\text{ and } A_1=1,\\
		0, & \text{otherwise}.
		\end{cases}
		\end{equation}
	\item (Divisor)
		If $\deg A_l=2,$ then
		\begin{equation}\label{ax_divisor}
		\ogw_{\beta,k}(A_1,\ldots,A_{l})=\int_\beta A_l \cdot\ogw_{\beta,k}(A_1,\ldots,A_{l-1}).
		\end{equation}
    \item (Deformation invariance) \label{ax:def}
        The invariants $\ogw_{\beta,k}$ remain constant under deformations of the symplectic form $\omega$ for which $L$ remains Lagrangian, the fixed subgroup $S_L \subset H_2(X,L;\Z)$ remains in the kernel of $\omega,$ and in the case of Theorem~\ref{thm2}, the involution $\phi$ remains anti-symplectic.
\end{enumerate}
\end{thm}

\subsubsection{Intuition and motivation}\label{sssec:intuition}

The invariant $\ogw_{\beta,k}(\g_{i_1},\ldots,\g_{i_l})$ counts configurations of $J$-holomorphic disks that collectively have degree $\beta$, $k$ boundary point-constraints, and interior constraints $\gamma_{i_1},\ldots,\gamma_{i_l}$.
Figure~\ref{fig:chain} illustrates a configuration that may contribute to $\ogw_{\beta,3}(\g_1,\g_2,\g_3,\g_4,\g_5,\g_6)$. The notation $b_j,\beta_j,$ comes from a decomposition $b=\sum_jT^{\beta_j}b_j$ with $b_j \in A^*(L)[[s,t_0,\ldots,t_N]].$ The illustration shows chains that intuitively represent the ``Poincar\'e duals" of the non-closed differential forms $b_j.$
When a disk degenerates to a two component stable map, instead of disappearing and spoiling the invariance of the count as in Figure~\ref{fig:bubble}, the two components are allowed to separate while being linked by part of the bounding chain. Thus, the count remains invariant.

\begin{figure}[!ht]
\centering
\includegraphics[width=12cm]{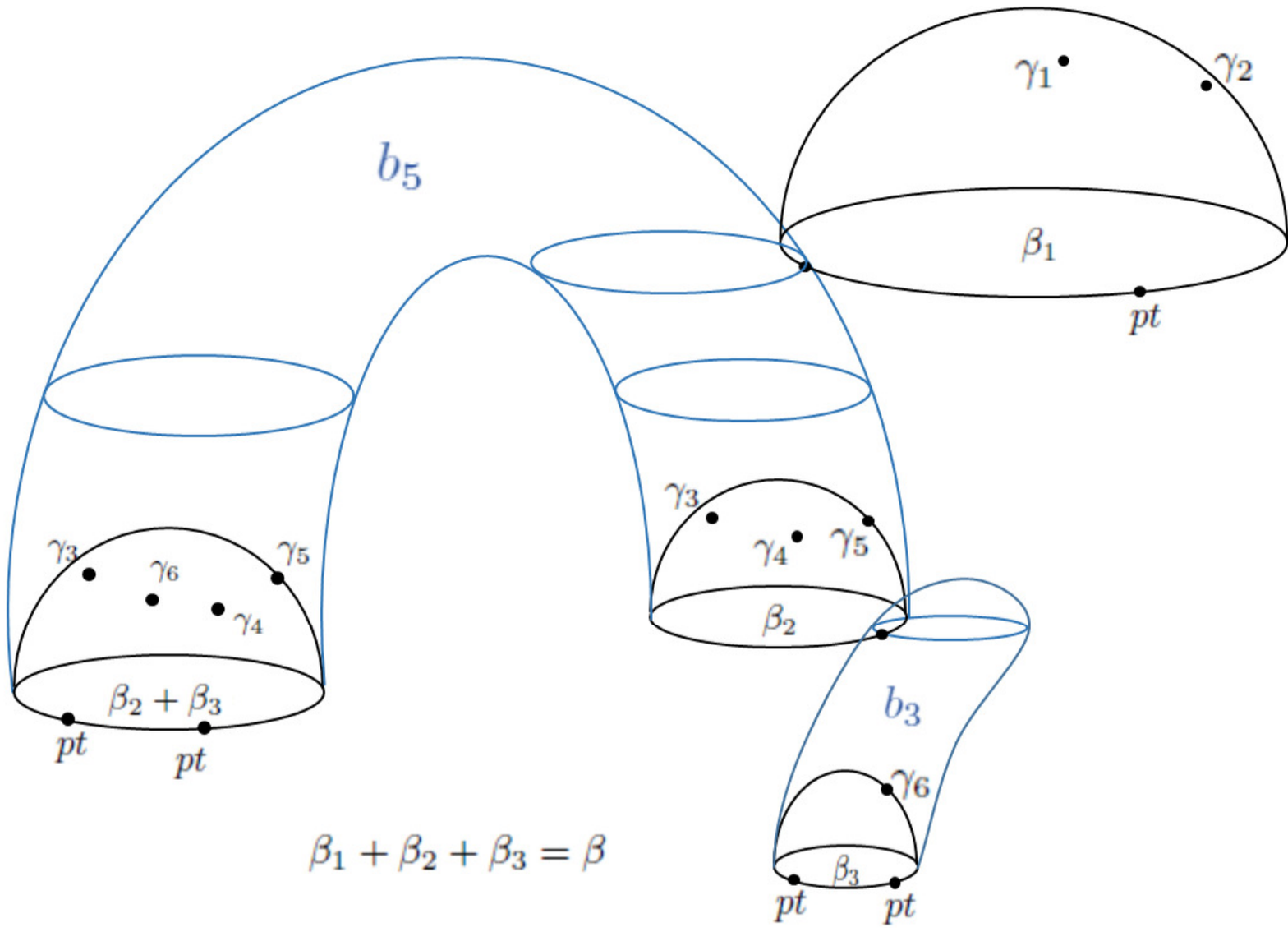}
\caption{A schematic description of one contribution to $\ogw_{\beta,3}(\g_1,\g_2,\g_3,\g_4,\g_5,\g_6)$}
\label{fig:chain}
\end{figure}

A major motivation for using bounding chains in the present work is that they allow the definition of open Gromov-Witten invariants that depend only on the Lagrangian submanifold without any additional data. It may be necessary to impose cohomological constraints on the Lagrangian as in Theorem~\ref{thm1}, but the resulting invariants depend only on the Lagrangian. They do not depend on the existence of an auxiliary geometric structure such as an anti-symplectic involution, or the choice of one if there are several distinct deformation equivalence classes.

As mentioned in Section~\ref{ssec:intro^2}, previous works on open Gromov-Witten invariants for Lagrangian submanifolds fixed by an anti-symplectic involution~\cite{Cho,Georgieva,SolomonThesis,Welschinger1,Welschinger2} use the involution to cancel disk bubbles in pairs.  Examining the proof of Theorem~\ref{thm2} for $n > 3,$ one sees that the involution does not control obstruction chains in degrees $0,n+1 \pmod 4,$ which means that disk bubbles do not cancel in pairs. This motivates the introduction of bounding chains in the presence of boundary constraints or involution invariant interior constraints when $n > 3.$

Even when only involution anti-invariant interior constraints are present, bounding chains appear to be useful. Indeed,
Proposition~\ref{prop:altrealspin} and the associated definition of superpotential invariants in Section~\ref{sssec:geosi} give an involution based disk cancelling argument similar to that of Georgieva~\cite{Georgieva} in the language of the present paper. The sign calculation in the proof of Proposition~\ref{prop:altrealspin} fails for the Lagrangian submanifold $L_{d,n} \subset X_{d,n}$ when $d \equiv 3 \pmod 4$  and $n$ is odd as in Example~\ref{ex:XL}.
Nonetheless, one can define invariants using bounding chains.

\subsubsection{Comparison with other invariants}\label{sssec:comps}
When $X$ is a Calabi-Yau threefold and $L$ has vanishing Maslov class, our invariants are essentially the same as those of Fukaya~\cite{Fukaya2}. Indeed, it follows from the degree axiom of Theorem~\ref{axioms} that $\ogw_{\beta,k}(A_1,\ldots,A_l) = 0$ unless $k = 0$ and $\deg A_j = 2$ for $j = 1,\ldots,l.$ We can further reduce to the case $l = 0$ by the divisor axiom. So, we can set the formal variables $s$ and $t_j$ to zero without losing any information, and the superpotential $\Omega$ is the same as Fukaya's modulo terms of type $\mathcal{D}.$ Terms of type $\mathcal{D}$ are removed so that invariance is not spoiled by the wall-crossing formula of~\cite[Theorem 1.5]{Fukaya2}.

For a symplectic manifold $(X,\omega)$ with an anti-symplectic involution $\phi: X \to X,$ we proceed to compare the invariants defined above with other invariants in the literature. We begin with some relevant notation. Let $Y \subset \fix(\phi)$ be closed and open. Define the doubling map
\begin{equation}\label{eq:chi}
\chi_Y: H_2(X,Y;\Z) \to H_2(X;\Z)
\end{equation}
as follows. For $\beta \in H_2(X,Y;\Z)$ represent $\beta$ by a singular chain $\sigma \in C_2(X,Y;\Z).$ Then $\chi_Y(\beta) = [\sigma - \phi_\#\sigma].$ Let
\[
\hat \chi_Y : H_2(X;\Z) \to H_2(X;\Z)
\]
denote the composition of $\chi_{Y}$ and the natural homomorphism $H_2(X;\Z) \to H_2(X,Y;\Z).$ Observe that $\chi_{Y}$ descends to the quotient of $H_2(X,Y;\Z)$ by any subgroup contained in $\Im(\Id+\phi_*).$ We denote the map on the quotient by $\chi_Y$ as well. In particular, as long as $\ssly_L \subset \Im(\Id+\phi_*),$ we have $\chi_L : \sly_L \to H_2(X;\Z).$ Thus, for the following two theorems, we always take $\ssly_L \subset \Im(\Id+\phi_*)$. When we consider invariants arising from Theorem~\ref{thm2}, the opposite inclusion holds as well, so $\ssly_L = \Im(\Id+\phi_*).$

The following theorem relates Welschinger's invariants~\cite{Welschinger1,Welschinger2} to the invariants $\ogw_{\beta,k}$ associated with a real setting~$(X,L,\omega,\phi)$.
For $d \in H_2(X;\Z), l \in \Z_{\geq 0},$ write
\[
k_{d,l} = \frac{c_1(X)(d) - 2(n-1)l+n-3}2.
\]
For $n = 2,3,$ and either $k_{d,l} \geq 1$ or $d \notin \Im \hat\chi_L,$ denote by $\mathcal{W}_{d,l}$ Welschinger's invariant counting real rational $J$-holomorphic curves in $(X,\phi)$ of degree $d$ with real locus in $L$ passing through $k_{d,l}$ real points and $l$ pairs of $\phi$-conjugate points. Although~\cite{Welschinger2} does not allow the case $k_{d,l} = 0$ and $d \notin \Im\hat \chi_L,$ the definition extends without modification. Recall from Remark~\ref{rem:dim23} that when $n = 2,3,$ the hypothesis of Theorem~\ref{thm2} is always satisfied, so the invariants $\ogw_{\beta,k}$ are defined.
\begin{thm}[Comparison with Welschinger's invariants]\label{Welsch}
Suppose $n=2$ or $3$. Let $A \in \Hh^{2n}_\phi(X,L;\R)$ be the Poincar\'e dual of a point. Let $d \in H_2(X;\Z), l \in \Z_{\geq 0},$ such that if $n = 3$, then either $k_{d,l} \geq 1$ or $d \notin \im \hat \chi_L$, and if $n=2,$ then $k_{d,l} = 1.$ Then,
\[
\sum_{\chi_L(\beta)=d}
\ogw_{\beta,\,k_{d,l}}(A^{\otimes l})
=\pm 2^{1-l}\cdot \mathcal{W}_{d,\,l}.
\]
\end{thm}

\begin{rem}\label{rem:noddlim}
When $n = 2,$ the hypothesis $k_{d,l} = 1$ is necessary. Indeed, by Remark~\ref{rem:neven}, the open Gromov-Witten invariants $\ogw_{\beta,k}$ vanish whenever $k > 1,$ but it is shown in~\cite{ItenbergKharlamovShustin2,ItenbergKharlamovShustin4} that Welschinger's invariants $\mathcal{W}_{d,l}$ do not generally vanish when $k_{d,l} > 1.$ On the other hand, the modified invariants mentioned in Remark~\ref{rem:neven} coincide with Welschinger's invariants also for $k_{d,l} > 1$ when $n = 2.$
\end{rem}

The following theorem relates the invariants of Georgieva~\cite{Georgieva} for $(X,\omega,\phi)$ to the invariants $\ogw_{\beta,k}^L:W_L^{\otimes l} \to \R$ associated to the components  $L\subset fix(\phi).$
The invariants $\ogw_{\beta,k}^L$ may arise from either Theorem~\ref{thm1} or Theorem~\ref{thm2} depending on which hypothesis is satisfied by $L$. The group $\sly_L$ and the vector space $W_L$ are taken accordingly. Assume $(X,\omega,\phi)$ is admissible in the sense of~\cite{Georgieva}, choose a $\phi$-orienting structure as in~\cite{Georgieva}, and let $\s_L$ be the associated relative spin structure for the component $L \subset fix(\phi).$
Recall that $\s_L$ determines a class $w_{\s_L} \in H^2(X;\Z/2\Z)$ such that $w_2(TL) = i^* w_{\s_L}.$
Let $\mathcal A \subset H_2(X;\Q)$ denote the complement of the image of $\hat \chi_{fix(\phi)}$ composed with the natural map $H_2(X;\Z) \to H_2(X;\Q).$ For $d \in \mathcal A$ and $A_j \in H^*(X;\R),$ let $\ogw^{\text{Georgieva}}_{d,0,l}(A_1,\ldots,A_l)$ denote the invariant of~\cite{Georgieva}. Let $W_\phi$ denote $H^*(X;\R)^{-\phi^*}$ if the hypothesis of Theorem~\ref{thm1} holds for each component $L \subset fix(\phi).$ Otherwise, let $W_\phi$ denote $H^{even}_\phi(X;\R)^{-\phi^*} = \oplus_{m \text{ odd}} H^{2m}(X;\R)^{-\phi^*}.$
The natural map $H^*(X,L;\R)^{-\phi^*} \to H^*(X;\R)^{-\phi^*}$ is an isomorphism by the long exact sequence of the pair $(X,L).$ So, we identify $W_\phi$ with a subspace of $W_L$ for each component $L \subset fix(\phi).$

\begin{thm}[Comparison with Georgieva's invariants]\label{thm:penka}
Suppose $(X,\omega,\phi)$ is admissible in the sense of~\cite{Georgieva}, and the hypothesis of either Theorem~\ref{thm1} or Theorem~\ref{thm2} holds for each component $L \subset fix(\phi).$  Assume that $\frac{\mu(\beta)}{2} + w_{\s_L}(\chi_L(\beta)) \equiv 0 \pmod 2$ for all components $L \subset fix(\phi)$ and all $\beta\in \sly_L$. Then, for $d \in \mathcal{A}$ and $A_j \in W_\phi$, we have
\[
\sum_{\substack{\chi_L(\beta) = d, \\ L \subset fix(\phi)}}\ogw^L_{\beta,0}(A_1,\ldots,A_l)=
2^{1-l}\ogw^{\text{Georgieva}}_{d,0,l}(A_1,\ldots,A_l).
\]
\end{thm}

\begin{rem}
Theorem~\ref{thm:penka} does not generalize if we replace $W_\phi$ with $H^*(X;\R)$ or $H^{even}_\phi(X;\R)$. Indeed, the invariants $\ogw^{\text{Georgieva}}_{d,0,l}(A_1,\ldots,A_l)$ are known to vanish if there is a $j$ such that $\phi^* A_j = A_j.$ However, the invariants $\ogw^L_{\beta,0}$ do not generally vanish as shown in~\cite{ST3}.
\end{rem}

\begin{rem}
On the other hand, when $L$ satisfies the hypothesis of Theorem~\ref{thm2} but not of Theorem~\ref{thm1}, the invariants $\ogw_{\beta,k}^L(A_1,\ldots,A_l)$ are defined only for $A_1,\ldots,A_l \subset H_\phi^*(X)$ whereas the invariants $\ogw_{\beta,0,l}^{\text{Georgieva}}(A_1,\ldots,A_l)$ are defined for $A_1,\ldots,A_l \in H^*(X).$ In Section~\ref{sssec:geosuperpot} we show how to use the superpotential to recover the invariants $\ogw_{\beta,0,l}^{\text{Georgieva}}$ in this case as well using a classification result for a modified notion of real bounding pairs given in Section~\ref{ssec:pnksomg}.
\end{rem}

\subsection{Context}
\subsubsection{The symmetry approach}
Several existing approaches utilize symmetries to define open invariants in certain cases.
Katz-Liu \cite{KatzLiu} and Liu \cite{Liu-M} use an $S^1$ action.
Using anti-symplectic involutions, Cho \cite{Cho} and Solomon \cite{SolomonThesis}, both in dimensions 2 and 3, define open invariants that generalize Welschinger's real enumerative invariants~\cite{Welschinger1,Welschinger2}. Georgieva~\cite{Georgieva} uses anti-symplectic involutions to treat higher dimensions in the absence of boundary constraints. A similar result can be deduced from \cite{SolomonThesis}.

\subsubsection{The superpotential}
The problem of defining open Gromov-Witten invariants without using symmetries has a long history. The idea of using the superpotential to define holomorphic disk counting invariants goes back to~\cite{Witten3}. Subsequently, the superpotential has been discussed widely in the physics literature~\cite{HerbstLazariouLerche,Lazaroiu,Tomasiello,Vafa,Walcher}. It is explained in~\cite[Section 0.4]{PandharipandeSolomonWalcher} that the invariants of~\cite{SolomonThesis} in the Calabi-Yau setting can be extracted from a critical value of the superpotential. The use of bounding chains along with the superpotential to define invariants was suggested by Joyce~\cite[Section~6.7]{Joyce}.

\subsubsection{Vanishing Maslov class}
Fukaya~\cite{Fukaya2} uses the superpotential and bounding chains to define open Gromov-Witten invariants for a Lagrangian submanifold $L$ with vanishing Maslov class in a Calabi-Yau threefold. In the Maslov zero setting, the grading on $\Lambda$ is trivial. Moreover, when $\dim L = 3,$ Maslov zero disks have expected dimension zero. So, boundary and interior constraints as well as the corresponding variables $t_i$ and $s$ are irrelevant. Consequently, it is natural to take $R = \Lambda$ and the degree $1$ elements of $C$ are necessarily $1$-forms. Also, bounding chains are necessarily strong, that is, the constant $c$ of equation~\eqref{eq:bc} vanishes. Strong bounding chains are critical points of the superpotential. So, the superpotential is constant on each connected component (in an appropriate non-Archimedean sense) of the space of bounding chains. Thus, if one can identify a connected component of the space of bounding chains, evaluating the superpotential at any point therein will give the same invariant. In favorable situations, there may be only one connected component.

There are several other works in the Maslov zero setting, which should be related to that of Fukaya~\cite{Fukaya2}. Aganagic-Ekholm-Ng-Vafa~\cite{AENV} define open Gromov-Witten invariants for knots, and Ekholm-Shende~\cite{EkholmShende} construct open Gromov-Witten invariants valued in skeins. Cho~\cite{Cho0} defines an invariant function similar to the superpotential on Hochschild cycles.  Iacovino defines a superpotential counting ``multi-curves'' in Calabi-Yau threefolds~\cite{Iacovino1,Iacovino2}. He then rephrases his construction in the language of obstruction theory~\cite{Iacovino3}.

\subsubsection{Non-vanishing Maslov class}
When the Maslov class of the Lagrangian submanifold $L \subset X$ does not vanish, strong bounding chains generally do not exist. Rather, one must consider weak bounding chains~\cite{FOOO} as we do in the present article. Weak bounding chains are not critical points of the superpotential.
So, the value of the superpotential depends on the precise choice of bounding chain, as opposed to its connected component in the Maslov zero case.
On the other hand, the fact that weak bounding chains are not critical points means that the superpotential can give rise to larger families of invariants. Indeed, if one can find a canonically parameterized family of bounding chains on which to evaluate the superpotential, one obtains a parameterized family of invariants. The parameter $k$ in the invariants $\ogw_{\beta,k}(\cdot)$ of the present work arises in this way from the $s$ dependence of $b.$

Fukaya-Oh-Ohta-Ono~\cite{FOOO1} study a potential function on the space of weak bounding chains, which is essentially the derivative of the superpotential discussed here, in the context of Lagrangian tori in compact toric manifolds. They recover quantum cohomology, or equivalently, closed Gromov-Witten theory, as the Jacobian ring of the potential. However, there does not seem to be a canonical parameterization of the space of weak bounding chains in that context, so it is not clear how to extract a numerical invariant by evaluating the superpotential on a bounding chain. Nonetheless, Fukaya~\cite[Section 8.3]{Fukaya2} thinks it likely the superpotential approach can be used to recover the invariants of~\cite{Welschinger1,Welschinger2,SolomonThesis}. The present paper confirms the superpotential approach does in fact recover these invariants and moreover, extends them to arbitrary odd dimension.

\subsubsection{The present work}
Theorems~\ref{thm1} and~\ref{thm2} allow us to choose canonical bounding chains on which the superpotential can be evaluated to obtain invariants. Theorems~\ref{Welsch} and~\ref{thm:penka} show the invariants obtained from the superpotential recover the invariants of Welschinger~\cite{Welschinger1,Welschinger2} and Georgieva~\cite{Georgieva}. In particular, Theorem~\ref{Welsch} shows that bounding chains play the role of boundary point constraints, Poincar\'e dual to $n$-forms, contrary to the intuition from the Calabi-Yau case where bounding chains are $1$-forms.

The proof of Theorem~\ref{thm2} clarifies the importance of anti-symplectic involution symmetry in defining invariants. Namely, such an involution forces part of the obstructions to the existence and uniqueness of bounding chains to vanish. The fact that not all obstructions are forced to vanish explains why the symmetry approach only works in low dimensions~\cite{Welschinger1,Welschinger2} or in the absence of boundary constraints~\cite{Georgieva}. The cohomological hypothesis of Theorem~\ref{thm2} is used to deal with the obstructions that are not forced to vanish by symmetry.

The present work explains why boundary constraints in the open Gromov-Witten invariants of~\cite{Cho,SolomonThesis,Welschinger1,Welschinger2} could only be points and not arbitrary cycles. Indeed, to construct a canonical family of bounding chains, we use the map $\varrho$~\eqref{eqn_rho}, which integrates bounding chains over $L.$ The proof in Lemma~\ref{lm_rho} that $\varrho$ is invariant under gauge equivalence makes crucial use of the top degree property of $A_\infty$ pseudoisotopies, Proposition~\ref{cl:qt_no_top_deg}. In particular, integrating over any cycle on $L$ apart from the fundamental class would not give a well-defined map. As explained in Remark~\ref{rem1}, it follows from the definition of $\varrho$ that our canonical family of bounding chains is point-like. Even if we suppose there exists a canonical family of bounding chains involving more parameters than $\varrho,$ it is shown in Remark 4.16 of~\cite{ST3} that generically the superpotential only depends on the parameters of $\varrho.$

Theorem~\ref{axioms} shows that the superpotential invariants, despite their abstract definition, satisfy simple axioms analogous to those of closed Gromov-Witten theory. To prove Theorem~\ref{axioms}, Definition~\ref{dfn:b_axioms} formulates properties of bounding chains analogous to the fundamental class and divisor axioms. Under the assumptions of Theorem~\ref{thm1} (resp. Theorem~\ref{thm2}), Proposition~\ref{lm:d_t0b} (resp. Proposition~\ref{div_real}) shows that any gauge equivalence class of bounding chains has a representative that satisfies these axioms.

It is known~\cite{BrugalleMikhalkin1,BrugalleMikhalkin2,BrugalleGeorgieva} that Welschinger's invariants for $(\C P^3,\R P^3)$, and thus also the invariants $\ogw_{\beta,k}(\cdot)$ of the present paper, are non-zero for many choices of $\beta,$ interior constraints, and both $k = 0$ and $k > 0.$ It is known~\cite{Tehrani,GZ0} that Georgieva's invariants and thus also $\ogw_{\beta,0}(\cdot)$ are non-zero for $(\C P^n,\R P^n)$ with $n > 3$ odd and various choices of $\beta$ and interior constraints. In~\cite{ST3}, we give recursive formulas that completely determine the invariants $\ogw_{\beta,k}(\cdot)$ in the case $(X,L) = (\C P^n,\R P^n)$ with $n$ odd. In particular, these invariants are shown to be non-zero for many choices of $\beta$ and interior constraints even when both $n > 3$ and $k>0.$
Furthermore, the invariants $\ogw_{\beta,0}(A_1,\ldots,A_l)$ are shown to be non-zero in many cases even when $\phi^*A_j = A_j$ for one or more $j,$ whereas the invariants of Georgieva~\cite{Georgieva} vanish.

\subsubsection{Future plans}
Work in progress proves an analog of Theorem~\ref{thm1} under the weaker cohomological assumption that the restriction map $H^m(X;\R) \to H^m(L;\R)$ is surjective for $0 < m < n.$
As mentioned in Remark~\ref{rem:neven}, work in progress modifies the definition of the invariants $\ogw_{\beta,k}$ in even dimensions so they no longer vanish when $k > 1.$ Theorem~\ref{Welsch} extends accordingly as mentioned in Remark~\ref{rem:noddlim}. Another work in progress extends the definition of the invariants so that $L$ need not be orientable.
Furthermore, we plan to apply the techniques of the present paper to the non-compact Calabi-Yau setting studied in~\cite{GopakumarVafa} and~\cite{AganagicVafa}.

\subsubsection{Related work}
Welschinger \cite{Welschinger4d,Welschinger6d} corrects disk bubbling by taking into account linking numbers, thus obtaining invariants in dimensions $2$ and $3$. The relation with the invariants of the present paper is explained in~\cite{Chen} in dimension $3.$ We expect a similar result in dimension~$2.$ Another related approach is being developed by Tessler~\cite{Tessler}.

Netser Zernik~\cite{Zernik,Zernik1} follows an approach closely related to the present work to define equivariant open Gromov-Witten invariants and give an equivariant localization formula for them.

Biran-Cornea~\cite{BC} define an invariant of Lagrangian $2$-tori counting disks of a given degree through three points. The invariant arises as the discriminant of a quadratic form associated to the Lagrangian quantum product.
Biran-Membrez~\cite{BiranMembrez} define a related invariant for even dimensional Lagrangian spheres.
Cho~\cite{Cho0} gives an example of an open Gromov-Witten invariant for the Clifford torus in $\P^2$. It would be interesting to find a connection between any of these invariants and those of the present paper.

\subsection{Outline}

In Section~\ref{sec:a_infty} we quote some results from~\cite{ST1} that will be useful in the other sections. This includes the construction of the $A_\infty$ structure in Section~\ref{ssec:constr} using $\q$ operators, properties of the $\q$ operators in Section~\ref{ssec:properties}, and the notion of pseudoisotopies between two $A_\infty$ structures in Section~\ref{pseudoisot}.

Section~\ref{sec:bd_chains} contains results concerning bounding pairs, most notably a construction of bounding pairs in Section~\ref{construct_bd_ch} and a construction of gauge equivalences of bounding pairs in Section~\ref{ssec:classification}. Together, these two sections prove Theorem~\ref{thm1}, as detailed in the end of Section~\ref{ssec:classification}.

Section~\ref{sec:real} deals with the real setting. After establishing signs of conjugation in Section~\ref{ssec:prelim}, we move to proving Theorem~\ref{thm2} in Section~\ref{ssec:spin}.
Section~\ref{ssec:pnksomg} proves a classification result for a modified notion of real bounding pairs, which leads to an interpretation of the invariants of Georgieva~\cite{Georgieva} in terms of the superpotential.

Section~\ref{sec:suppot} concerns the superpotential and the open Gromov-Witten invariants $\ogw_{\beta,k}$ derived from it. In Section~\ref{subsec:invar} we show the superpotential is invariant under pseudoisotopy proving Theorem~\ref{thm_inv}. In Section~\ref{ssec:cff}, we give a coordinate free definition of the invariants $\ogw_{\beta,k}$. In Section~\ref{ssec:bd_axioms} we construct bounding chains with properties reminiscent of the Gromov-Witten axioms. Section~\ref{ssec:axioms} proves the axioms of the open Gromov-Witten invariants $\ogw_{\beta,k}$, that is, Theorem~\ref{axioms}. Section~\ref{ssec:relax} shows how to adapt the proofs of Theorems~\ref{thm_inv}-\ref{axioms} to the complex $\Ah^*(X,L)$ in place of $A^*(X,L).$ The complex $\Ah^*(X,L)$ plays an important role in~\cite{ST3}. Section~\ref{ssec:comparison} compares the invariants $\ogw_{\beta,k}$ with those of Welschinger proving Theorem~\ref{Welsch}. Section~\ref{sssec:geosuperpot} applies the superpotential to the modified real bounding pairs of Section~\ref{ssec:pnksomg} to obtain open Gromov-Witten invariants under assumptions similar to those of Georgieva~\cite{Georgieva}. These invariants are related to both Georgieva's invariants and the invariants~$\ogw_{\beta,k}$. The proof of Theorem~\ref{thm:penka} follows.

\subsection{Acknowledgments}
The authors would like to thank M. Abouzaid, D. Auroux, P.~Georgieva, D. Joyce, T. Kimura, E. Kosloff, M. Liu, L. Polterovich, E. Shustin, I. Smith, G.~Tian, and A. Netser Zernik, for helpful conversations.
The authors were partially supported by ERC starting grant 337560 and ISF Grant 1747/13. The first author was partially supported by ISF Grant 569/18.
The second author was partially supported by the Canada Research Chairs Program and NSF grant No. DMS-163852.

\subsection{General notation}\label{ssec:gennote}\leavevmode

We write $I:=[0,1]$ for the closed unit interval.

Use $i$ to denote the inclusion $i:L\hookrightarrow X$. By abuse of notation, we also use $i$ for $\Id\times i:I\times L\to I\times X$. The meaning in each case should be clear from the context.

Denote by $pt$ the map (from any space) to a point.

Define a valuation
\[
\nu:R\lrarr \R_{\ge 0},
\]
by
\[
\nu\left(\sum_{j=0}^\infty a_jT^{\beta_j}s^{k_j}\prod_{a=0}^Nt_a^{l_{aj}}\right)
= \inf_{\substack{j\\a_j\ne 0}} \left(\omega(\beta_j)+k_j+\sum_{a=0}^N l_{aj}\right).
\]

Let $\Upsilon'$ be an $\R$ vector space and let $\Upsilon = \Upsilon' \otimes R.$ Then, equipping $\Upsilon'$ with the trivial valuation, $\nu$ induces a valuation on $\Upsilon$ which we also denote by $\nu.$

Whenever a tensor product (resp. direct sum) of modules with valuation is written, we mean the completed tensor product (resp. direct sum).

Write $A^*(L;R)$ for $A^*(L)\otimes R$. Similarly, $A^*(X;R)$ and $A^*(X,L;R)$ stand for $A^*(X)\otimes R$ and $A^*(X,L)\otimes R$, respectively.

Given $\alpha$, a homogeneous differential form with coefficients in $R$, denote by $|\alpha|$ the degree of the differential form, ignoring the grading of $R$. Denote by $\deg \alpha$ the total grading combining the differential form degree and the grading of $R.$

For a possibly non-homogeneous $\alpha$, denote by $(\alpha)_j$ the form that is the part of degree $j$ in $\alpha$. In particular, $|(\alpha)_j|=j$. Contrariwise, for a graded module $M$, the notation $M_j$ or $(M)_j$ stands for the degree~$j$ part of the module, which in the present context involves degrees of forms as well as degrees of variables.

Let $\Upsilon'$ be an $\R$-vector space, let $\Upsilon''=R,$ $Q,$ or $\L,$ and let $\Upsilon=\Upsilon'\otimes \Upsilon''$. For $x\in \Upsilon$ and a monomial $\lambda\in\Upsilon'',$ denote by $[\lambda](x)\in \Upsilon'$ the coefficient of $\lambda$ in $x$.

\section{\texorpdfstring{$A_{\infty}$ structures}{A-infty structures}}\label{sec:a_infty}
In this section we recall definitions and results from~\cite{ST1}. The notation, as well as sign and orientation conventions, are the same as in~\cite{ST1}, except that we have added the variable $s$ to $R$ (but not to $Q$).
Additionally, we mention relevant earlier references for these results from~\cite{FOOO, Fukaya, Fukaya2, FOOOtoricII, FOOOspec,FOOOKSII,FOOOKSVFC} with slightly different sign conventions but in the full generality offered by Kuranishi structures.

\begin{rem}
In the literature based on Kuranishi structures, it is generally necessary to first construct $A_\infty$ algebras and related structures modulo an arbitrary fixed energy level and then use homological algebra to obtain the full structure. This avoids the problem of the perturbed moduli space running out of the Kuranishi neighborhood as explained in~\cite[Section 7.2.3]{FOOO}. For the results of the present paper, it suffices to construct $A_\infty$ algebras and related structures modulo a fixed energy level. Indeed, it follows from the construction that an open Gromov-Witten invariant $\ogw_{\beta,k}(\cdots)$ depends only on the relevant Fukaya $A_\infty$ algebra truncated to energy level $\omega(\beta) + \epsilon$ for arbitrary $\epsilon > 0.$ Nonetheless, we work with full $A_\infty$ algebras to simplify notation.
\end{rem}

\subsection{Moduli spaces}
Throughout this work $(X,\omega)$ is a symplectic manifold, $L\subset X$ is a connected Lagrangian submanifold with relative spin structure $\s$, and $J$ is an $\omega$-tame almost complex structure. The notion of a relative spin structure appeared in~\cite{FOOO}. See also~\cite{WehrheimWoodward}. In particular, a relative spin structure determines an orientation on $L.$ Let $\dim_\R X=2n$. Write $\beta_0:=0\in \sly$.

A $J$-holomorphic genus-$0$ open stable map to $(X,L)$ of degree $\beta \in \sly$ with one boundary component, $k+1$ boundary marked points, and $l$ interior marked points, is a quadruple $(\Sigma, u,\vec{z},\vec{w})$ as follows. The domain $\Sigma$ is a genus-$0$ nodal Riemann surface with boundary consisting of one connected component. The map
\[
u: (\Sigma,\d\Sigma) \to (X,L)
\]
is continuous, $J$-holomorphic on each irreducible component of $\Sigma,$ and $[u_*([\Sigma,\partial\Sigma])] = \beta.$
The marked points are denoted by
\[
\vec{z} = (z_0,\ldots,z_k), \qquad \vec{w} = (w_1,\ldots,w_l),
\]
with $z_j \in \partial \Sigma, \, w_j \in int(\Sigma),$ distinct. The labeling of the marked points $z_j$ respects the cyclic order given by the orientation of $\partial \Sigma$ induced by the complex orientation of $\Sigma.$
Stability means that
if $\Sigma_i$ is an irreducible component of $\Sigma$, then either $u|_{\Sigma_i}$ is nonconstant or it satisfies the following requirement: If $\Sigma_i$ is a sphere, the number of marked points and nodal points on $\Sigma_i$ is at least 3; if $\Sigma_i$ is a disk, the number of marked and nodal boundary points plus twice the number of marked and nodal interior points is at least $3$.
An isomorphism of open stable maps $(\Sigma,u,\vec{z},\vec{w})$ and $(\Sigma',u',\vec{z}',\vec{w}')$ is a homeomorphism $\theta : \Sigma \to \Sigma'$, biholomorphic on each irreducible component, such that
\[
u = u' \circ \theta, \qquad\qquad  z_j' = \theta(z_j), \quad j = 0,\ldots,k, \qquad w_j' = \theta(w_j), \quad j = 1,\ldots,l.
\]
References for the notion of an open stable map include~\cite[Section 2.1.2]{FOOO} and~\cite[Section 5]{Liu-M}.

Let $\M_{k+1,l}(\beta) = \M_{k+1,l}(\beta;J)$ denote the moduli space of $J$-holomorphic genus zero open stable maps to $(X,L)$ of degree $\beta$ with one boundary component, $k+1$ boundary marked points, and $l$ internal marked points.
Let the evaluation maps
\begin{gather*}
evb_j^\beta:\M_{k+1,l}(\beta)\to L, \qquad  \qquad j=0,\ldots,k, \\
evi_j^\beta:\M_{k+1,l}(\beta) \to X, \qquad \qquad j=1,\ldots,l,
\end{gather*}
be given by $evb_j^\beta((\Sigma,u,\vec{z},\vec{w}))=u(z_j)$ and $evi_j^\beta((\Sigma,u,\vec{z},\vec{w}))= u(w_j).$
We may omit the superscript $\beta$ when the omission does not create ambiguity.

We assume that all $J$-holomorphic genus zero open stable maps with one boundary component are regular, the moduli spaces $\M_{k+1,l}(\beta;J)$ are smooth orbifolds with corners, and the evaluation maps $evb_0^\beta$ are proper submersions.
The definition of open Gromov-Witten invariants in the present paper extends to arbitrary targets $(X,\omega,L)$ and arbitrary $\omega$-tame almost complex structures given the virtual fundamental class technique of Kuranishi structures~\cite{Fukaya,Fukaya2,FOOO,FOOOtoricI,FOOOtoricII,FOOO1,FOOOinv,FOOOKSII,FOOOspec,FOOOKSVFC,FO}. The same is true for the proofs of four of the six axioms. The unit and divisor axioms require compatibility of the virtual fundamental class with the forgetful map of interior marked points, which has not yet been worked out in the Kuranishi structure formalism in the context of differential forms. See Remark~\ref{rem:intforget} below.
Alternatively, it should be possible to use the polyfold theory of~\cite{HoferWysockiZehnder,HoferWysockiZehnder1,HoferWysockiZehnder2,HoferWysockiZehnder3,LiWehrheim}.
The relative spin structure $\s$ determines an orientation on $\M_{k+1,l}(\beta)$ as in~\cite[Chapter 8]{FOOO}.

\subsection{Construction}\label{ssec:constr}

For all $\beta\in\sly$, $k,l\ge 0$,  $(k,l,\beta) \not\in\{ (1,0,\beta_0),(0,0,\beta_0)\}$, define
\[
\qkl^\beta:C^{\otimes k}\otimes A^*(X;Q)^{\otimes l} \lrarr C
\]
by
\begin{align*}
\q^{\beta}_{k,l}(\alpha_1\otimes\cdots\otimes\alpha_k;\gamma_1\otimes\cdots\otimes\gamma_l):=
(-1)^{\varepsilon(\alpha)}
(evb_0^\beta)_* \left(\bigwedge_{j=1}^l(evi_j^\beta)^*\gamma_j\wedge\bigwedge_{j=1}^k (evb_j^\beta)^*\alpha_j\right)
\end{align*}
with
\[
\varepsilon(\alpha)
:=\sum_{j=1}^kj(\deg\alpha_j+1)+1.
\]
The push-forward $(evb_0)_*$ is defined by integration over the fiber; it is well-defined because $evb_0$ is a proper submersion.
The case $\q_{0,0}^{\beta}$ is understood as $-(evb_0^{\beta})_*1.$
Define also $\q_{0,0}^{\beta_0}:=0$ and $\q_{1,0}^{\beta_0}(\alpha):=d\alpha$.
Set
\begin{align*}
\qkl:=\sum_{\beta\in\sly}T^{\beta}\qkl^{\beta}.
\end{align*}
Such operators are defined in~\cite{FOOOspec}, based on ideas from~\cite{FOOO} and~\cite{Fukaya, FOOOtoricII}.
Furthermore, for $l\ge 0$, $(l,\beta)\neq (1,\beta_0),(0,\beta_0),$ define
\[
\q_{-1,l}^\beta:A^*(X;Q)^{\otimes l}\lrarr Q
\]
by
\[
\q_{-1,l}^\beta(\gamma_1\otimes\cdots\otimes\gamma_l)
:=\int_{\M_{0,l}(\beta)} \bigwedge_{j=1}^l (evi_j^\beta)^*\gamma_j,
\]
define $\q_{-1,1}^{\beta_0}:=0, \q_{-1,0}^{\beta_0}:=0,$ and set
\[
\q_{-1,l}(\gamma_1\otimes\cdots\otimes\gamma_l):=\sum_{\beta\in\sly}T^{\beta} \q_{-1,l}^\beta(\gamma_1\otimes\cdots\otimes\gamma_l).
\]
This definition is motivated by ideas from~\cite{Fukaya2}.
Lastly, define similar operations using spheres,
\[
\q_{\emptyset,l}:A^*(X;Q)^{\otimes l}\lrarr A^*(X;Q),
\]
as follows. For $\beta\in H_2(X;\Z)$ let $\M_{l+1}(\beta)$ be the moduli space of genus zero $J$-holomorphic stable maps with $l+1$ marked points indexed from 0 to $l,$ representing the class $\beta$. Denote by $ev_j^\beta:\M_{l+1}(\beta)\to X$ the evaluation map at the $j$-th marked point. Assume that all the moduli spaces $\M_{l+1}(\beta)$ are smooth orbifolds and $ev_0^{\beta}$ is a submersion. Let $\pr: H_2(X;\Z) \to \sly$ denote the projection and let $w_\s\in H^2(X;\Z)$ be the class induced by the relative spin structure on $L$. For $l\ge 0$, $(l,\beta)\ne (1,0),(0,0)$, set
\begin{gather*}
\q_{\emptyset,l}^\beta(\gamma_1,\ldots,\gamma_l):=
(-1)^{w_\s(\beta)}
(ev_0^\beta)_*(\wedge_{j=1}^l(ev_j^\beta)^*\gamma_j),
\qquad
\q_{\emptyset,1}^0:= 0,\qquad \q_{\emptyset,0}^0:= 0,\\
\q_{\emptyset,l}(\gamma_1,\ldots,\gamma_l):=
\sum_{\beta\in H_2(X)}T^{\pr(\beta)} \q_{\emptyset,l}^\beta(\gamma_1,\ldots,\gamma_l).
\end{gather*}

Fix a closed form $\gamma\in \mI_QD$ with $\deg_D\gamma=2$.
Define maps on $C$ by
\[
\m_k^{\beta,\gamma}(\otimes_{j=1}^k\alpha_j)=
\sum_l\frac{1}{l!}T^\beta\qkl^\beta(\otimes_{j=1}^k\alpha_j;\gamma^{\otimes l}),\quad
\m_k^{\gamma}(\otimes_{j=1}^k\alpha_j)=
\sum_l\frac{1}{l!}\qkl(\otimes_{j=1}^k\alpha_j;\gamma^{\otimes l}),
\]
for all $k\ge -1,l\ge 0$. In particular, note that $\m_{-1}^\gamma\in R$.

The following is Proposition~2.6 in~\cite{ST1}. It appears as Lemma~17.10 of~\cite{FOOOspec}. See also Lemma 3.8.39 of~\cite{FOOO} in the singular chains model and~\cite{FOOOtoricII} in the toric context.

\begin{cl}[$A_\infty$ relations]\label{cl:a_infty_m}
The operations $\{\m_k^\gamma\}_{k\ge 0}$ define an $A_\infty$ structure on $C$. That is,
\begin{equation*}\label{eq:a-infty}
\sum_{\substack{k_1+k_2=k+1\\1\le i\le k_1}}(-1)^{\sum_{j=1}^{i-1}(\deg\alpha_j+1)}
\mg_{k_1}(\alpha_1,\ldots,\alpha_{i-1},\mg_{k_2}(\alpha_i,\ldots,\alpha_{i+k_2-1}), \alpha_{i+k_2},\ldots,\alpha_k)=0.
\end{equation*}
\end{cl}

\subsection{Properties}\label{ssec:properties}

Denote by $\langle\;,\;\rangle$ the signed Poincar\'e pairing
\begin{equation}\label{eq:pairing}
\langle\xi,\eta\rangle:=(-1)^{\deg\eta}\int_L\xi\wedge\eta.
\end{equation}

The following is Proposition~3.1 in~\cite{ST1}. It is implicit in Theorem~17.1 of~\cite{FOOOspec}. See also Theorem~3.8.32 of~\cite{FOOO} in the singular chains model.
\begin{cl}[Linearity]\label{lm:qlinear}
The $\q$ operators are multilinear in the sense that for $a \in R$ we have
\begin{multline*}
\qquad\q_{k,l}^\beta(\a_1,\ldots,\a_{i-1},a\cdot\a_i,\ldots,\a_k;\gamma_1,\ldots,\gamma_l)=\\
		=(-1)^{\deg a\cdot\big(i+\sum_{j=1}^{i-1}\deg\a_j +\sum_{j=1}^l\deg\gamma_j\big)}
		a\cdot\q_{k,l}^\beta(\a_1,\ldots,\a_k; \gamma_1,\ldots,\gamma_l),
\end{multline*}
and for $a \in Q$ we have
\[
\q_{k,l}^\beta(\a_1,\ldots,\a_k; \gamma_1,\ldots,a\cdot\gamma_i,\ldots,\gamma_l) =(-1)^{\deg a\cdot\sum_{j=1}^{i-1}\deg\gamma_j}
		a\cdot\q_{k,l}^\beta(\a_1,\ldots,\a_k;\gamma_1,\ldots,\gamma_l),
\]
and
\[
\q^\beta_{\emptyset,l}(\gamma_1,\ldots,a\cdot\gamma_i,\ldots,\gamma_l)=
(-1)^{\deg a\cdot\sum_{j=1}^{i-1}\deg\gamma_j} a\cdot\q^\beta_{\emptyset,l}(\gamma_1,\ldots,\gamma_l).
\]
In addition, the pairing $\langle\;,\,\rangle$ defined by~\eqref{eq:pairing} is $R$-bilinear in the sense that
\[
		\langle a.\alpha_1,\alpha_2\rangle=
a\langle\alpha_1,\alpha_2\rangle,\quad
		\langle\alpha_1,a.\alpha_2\rangle=
(-1)^{\deg a\cdot (1+\deg\alpha_1)}a\langle\alpha_1,\alpha_2\rangle.
\]
\end{cl}

The following is Proposition~3.2 in~\cite{ST1}. The part used in the present work is the case when the function $f \in A^0(L)$ is the constant $1,$ which is~\cite[Theorem~17.1(3)]{FOOOspec}. In the singular chains model, a homotopy version is~\cite[Theorem~3.8.32(3)]{FOOO}.

\begin{cl}[Unit of the algebra]\label{cl:unit}
Fix $f\in A^0(L)\otimes R$, $\alpha_1,\ldots,\alpha_{k}\in C,$ and $\gamma_1,\ldots,\gamma_l \in A^*(X;Q).$ Then
\begin{multline*}
\q_{\,k+1\!,l}^{\beta}(\alpha_1,\ldots,\alpha_{i-1},f,\alpha_{i},\ldots,\alpha_k ;\otimes_{r=1}^l\gamma_r)=\\
=
\begin{cases}
df, & (k+1,\;l,\beta)=(1,0,\beta_0),\\
(-1)^{\deg f}f\cdot\alpha_1, & (k+1,l,\beta)=(2,0,\beta_0),\: i=1,\\
(-1)^{\deg\alpha_1(\deg f+1)}f\cdot\alpha_1, & (k+1,l,\beta)=(2,0,\beta_0),\: i=2,\\
0,& \text{otherwise.}
\end{cases}
\end{multline*}
In particular, $1\in A^0(L)$ is a strong unit for the $A_\infty$ operations $\mg:$
\[
\m_{k+1}^\gamma(\alpha_1,\ldots,\alpha_{i-1},1,\alpha_{i},\ldots,\alpha_{k})=
\begin{cases}
0, & k\ge 2 \mbox{ or } k=0,\\
\alpha_1, & k=1,\: i=1,\\
(-1)^{\deg\alpha_1}\alpha_1, & k=1,\: i=2.
\end{cases}
\]
\end{cl}

The following is Proposition~3.3 in~\cite{ST1}. It appears in~\cite[Theorem 7.1]{Fukaya} when $\gamma = 0,$ and it follows from~\cite[Theorem 17.3(10)]{FOOOspec} for $\gamma$ arbitrary.
\begin{cl}[Cyclic structure]\label{cl:cyclic}
For any $\alpha_1,\ldots,\alpha_{k+1}\in C$ and $\gamma_1,\ldots,\gamma_l\in A^*(X;Q)$,
\begin{align*}
\langle\qkl(\alpha_1,\ldots&,\alpha_k;\gamma_1,\ldots\gamma_l),\alpha_{k+1}\rangle=\\
&(-1)^{(\deg\alpha_{k+1}+1)\sum_{j=1}^{k}(\deg\alpha_j+1)}\cdot
\langle \qkl(\alpha_{k+1},\alpha_1,\ldots,\alpha_{k-1};\gamma_1,\ldots,\gamma_l),\alpha_k\rangle.
\end{align*}
In particular, $(C,\{\m_k^\gamma\}_{k\ge 0})$ is a cyclic $A_\infty$ algebra for any $\gamma$.
\end{cl}

The following is Proposition~3.5 in~\cite{ST1}. It appears as one of the defining properties of the $\q$ operators in~\cite[Theorem 17.1]{FOOOspec}. See also~\cite[Theorem 3.8.32]{FOOO} in the singular chains model.

\begin{cl}[Degree of structure maps]\label{deg_str_map}
For $\gamma_1,\ldots,\gamma_l\in D_2,$ $k\ge 0,$ the map
\[
\qkl(\; ;\gamma_1,\ldots,\gamma_l):C^{\otimes k}\lrarr C
\]
is of degree $2-k$.
\end{cl}

The following is Proposition~3.6 in~\cite{ST1}. It appears as one of the defining properties of the $\q$ operators in~\cite[Theorem 17.1]{FOOOspec}. See also~\cite[Remark 17.7]{FOOOspec}.

\begin{cl}[Symmetry]\label{cl:symmetry}
Let $k\ge -1$. For any permutation $\sigma\in S_l,$
\[
\qkl(\alpha_1,\ldots,\alpha_k;\gamma_1,\ldots,\gamma_l)=
(-1)^{s_\sigma(\gamma)}\qkl(\alpha_1,\ldots,\alpha_k;\gamma_{\sigma(1)},\ldots,\gamma_{\sigma(l)}),
\]
where
$
s_\sigma(\gamma):
=
\sum_{\substack{i>j\\ \sigma(i)<\sigma(j)}} \deg\gamma_{\sigma(i)}\cdot\deg\gamma_{\sigma(j)}\pmod 2
.
$
\end{cl}

The following is Proposition~3.7 in~\cite{ST1}. A homotopy version appears in~\cite[Theorem~3.8.32]{FOOO} in the singular chain model.

\begin{cl}[Fundamental class]\label{q_fund}
For $k\ge 0,$
\[
\qkl^\beta(\alpha_1,\ldots,\alpha_k;1,\gamma_1,\ldots,\gamma_{l-1})=
\begin{cases}
-1, & (k,l,\beta)=(0,1,\beta_0),\\
0, & \text{otherwise}.
\end{cases}
\]
Furthermore,
\[
\q_{-1,l}^\beta(1,\gamma_1,\ldots,\gamma_{l-1})=0.
\]
\end{cl}

The following is Proposition~3.8 in~\cite{ST1}.
The proof holds verbatim in the setting of Kuranishi structures as long as moduli spaces of $J$-holomorphic disks of energy zero are not perturbed, which is possible by Theorem~2.16(V) of~\cite{FOOOKSII}. See also equation (17.10) of~\cite{FOOOspec} and in the singular chain model, equation (3.8.34.1) of~\cite{FOOO}.

\begin{cl}[Energy zero]\label{q_zero}
For $k\ge 0,$
\[
\qkl^{\beta_0}(\alpha_1,\ldots,\alpha_k;\gamma_1,\ldots,\gamma_l)=
\begin{cases}
d\alpha_1, & (k,l)=(1,0),\\
(-1)^{\deg\alpha_1}\alpha_1\wedge\alpha_2, & (k,l)=(2,0),\\
-\gamma_1|_L, & (k,l)=(0,1),\\
0, & \text{otherwise}.
\end{cases}
\]
Furthermore,
\[
\q_{-1,l}^{\beta_0}(\gamma_1,\ldots,\gamma_l)=0.
\]
\end{cl}

The following is Proposition~3.9 in~\cite{ST1}. It appears in~\cite[Lemma~9.2]{FOOOtoricII} in the toric context.
\begin{cl}[Divisors]\label{cl:q_div}
Assume $\gamma_1\in A^2(X,L)\otimes Q$, $d\gamma_1 = 0$,
and the map $H_2(X,L;\Z)\to Q$ given by $\beta\mapsto\int_\beta\gamma_1$ descends to $\sly$.
Then
	\[
	\qkl^{\beta}(\otimes_{j=1}^k\alpha_j;\otimes_{j=1}^{l}\gamma_j)=
	\left(\int_\beta\gamma_1\right) \cdot\q_{k,l-1}^{\beta} (\otimes_{j=1}^k\alpha_j;\otimes_{j=2}^{l}\gamma_j)
	\]
for $k\ge -1$.
	\end{cl}

\begin{rem}\label{rem:intforget}
The proof of Propositions~\ref{q_fund} and~\ref{cl:q_div} using Kuranishi structures would require compatibility of the Kuranishi structures with the forgetful map at interior marked points. We refer the reader to~\cite[Remark 3.1]{Fukaya} for a discussion of what this would entail.

In the present paper, Propositions~\ref{q_fund} and~\ref{cl:q_div} are used only in the proof of the unit and divisor axioms of Theorem~\ref{axioms}. They are not needed for the definition of open Gromov-Witten invariants or the other axioms.
\end{rem}

The following is Proposition~3.12 in~\cite{ST1}. Alternatively, it follows from an argument similar to the proof of~\cite[Theorem 17.1(3)]{FOOOspec}.
\begin{cl}[Top degree]\label{no_top_deg}
Suppose
\[
(k,l,\beta)\not\in\{(1,0,\beta_0),(0,1,\beta_0),(2,0,\beta_0)\}.
\]
Then $(\qkl^\beta(\alpha;\gamma))_n=0$ for all lists $\alpha,\gamma.$
\end{cl}

\begin{rem}
Proposition~\ref{lm:qlinear} shows the operations $\mg$ are multilinear.
Proposition~\ref{cl:cyclic} shows the $A_\infty$ algebra $(C,\mg)$ is cyclic. Proposition~\ref{cl:unit} shows unitality of $(C,\m^\gamma)$ in the usual sense, and Proposition~\ref{no_top_deg} shows an analog of unitality for $\mg_0.$ Thus, $(C,\mg,\langle\;,\,\rangle,1)$ is a cyclic unital $A_\infty$ algebra in the sense of~\cite[Definition~1.1]{ST1}.
\end{rem}

\subsection{Pseudoisotopies}\label{pseudoisot}
Set
\begin{equation}\label{eq:mR}
\mC:=A^*(I\times L;R), \qquad \mD:=A^*(I\times X,I\times L;Q), \qquad \mR := A^*(I;R).
\end{equation}
We construct a family of $A_\infty$ structures on $\mC$. Let $\{J_t\}_{t\in I}$ be a path in $\J$ from $J = J_0$ to $J'=J_1.$
For each $\beta, k, l,$ set
\[
\Mt_{k+1,l}(\beta):=\{(t,u,\vec{z},\vec{w})\,|\, (u,\vec{z},\vec{w})\in\M_{k+1,l}(\beta;J_t)\}.
\]
We have evaluation maps
\begin{gather*}
\evbt_j:\Mt_{k+1,l}(\beta)\lrarr I\times L, \quad j\in\{0,\ldots,k\},\\
\evbt_j(t,u,\vec{z},\vec{w}):=(t,u(z_j)),
\end{gather*}
and
\begin{gather*}
\evit_j:\Mt_{k+1,l}(\beta)\lrarr I\times X, \quad j\in\{1,\ldots,l\}\\
\evit_j(t,u,\vec{z},\vec{w}):=(t,u(w_j)).
\end{gather*}
It follows from the assumption on $\J$ that all $\Mt_{k+1,l}(\beta)$ are smooth orbifolds with corners, and $\evbt_0$ is a proper submersion.
\begin{ex}
In the special case when $J_t=J = J'$ for all $t\in I$, we have $\Mt_{k+1,l}(\beta)=I\times\M_{k+1,l}(\beta;J)$. The evaluation maps in this case are $\evbt_j=\Id \times evb_j$ and $\evit_j=\Id\times evi_j.$
\end{ex}

Let
\[
p:I\times L\lrarr I,\qquad p_\M: \Mt_{k+1,l}(\beta)\lrarr I
\]
denote the projections.

Define
\[
\qt_{k,l}^{\beta}:\mC^{\otimes k}\otimes A^*(I\times X;Q)^{\otimes l}\lrarr \mC,\quad k,l\ge 0,
\]
by
\begin{gather*}
\qt_{1,0}^{\beta_0=0}(\at)=d\at,\quad \qt_{k,l}^{\beta}(\otimes_{j=1}^k\at_j;\otimes_{j=1}^l\gt_j):= (-1)^{\varepsilon(\at)} (\evbt_0)_*( \wedge_{j=1}^l\evit_j^*\gt_j\wedge\wedge_{j=1}^k\evbt_j^*\at_j),\\
\quad \at,\at_j\in A^*(I\times L),\qquad\gt_j\in A^*(I\times X).
\end{gather*}
Define
\[
\qt_{-1,l}^{\beta}:A^*(I\times X;Q)^l\lrarr A^*(I;Q),\quad l\ge 0,
\]
by
\[
\qt_{-1,l}^{\beta}(\otimes_{j=1}^l\gt_j):= (p_\M)_*\wedge_{j=1}^l\evit_j^*\gt_j.
\]
As before, denote the sum over $\beta$ by
\begin{gather*}
\qt_{k,l}(\otimes_{j=1}^k\at_j;\otimes_{j=1}^l\gt_j):=
\sum_{\beta\in \sly}
T^{\beta}\qt_{k,l}^{\beta}(\otimes_{j=1}^k\at_j;\otimes_{j=1}^l\gt_j),\\
\qt_{-1,l}(\otimes_{j=1}^l\gt_j):=\sum_{\beta\in \sly}T^{\beta}\qt_{-1,l}(\gt^l).
\end{gather*}
Lastly, define similar operations using spheres,
\[
\qt_{\emptyset,l}:A^*(I\times X;Q)^{\otimes l}\lrarr A^*(I\times X;R),
\]
as follows. For $\beta\in H_2(X;\Z)$ let
\[
\Mt_{l+1}(\beta):=\{(t,u,\vec{w})\;|\,(u,\vec{w})\in \M_{l+1}(\beta;J_t)\}.
\]
For $j=0,\ldots,l,$ let
\begin{gather*}
\evt_j^\beta:\Mt_{l+1}(\beta)\to I\times X,\\
\evt_j^\beta(t,u,\vec{w}):=(t,u(w_j)),
\end{gather*}
be the evaluation maps. Assume that all the moduli spaces $\Mt_{l+1}(\beta)$ are smooth orbifolds and $\evt_0$ is a submersion. For $l\ge 0$, $(l,\beta)\ne (1,0),(0,0)$, set
\[
\qt_{\emptyset,l}^\beta(\gt_1,\ldots,\gt_l):=
(-1)^{w_\s(\beta)}
(\evt_0^\beta)_*(\wedge_{j=1}^l(\evt_j^\beta)^*\gt_j),
\]
and define $\qt_{\emptyset,1}^0:= 0, \qt_{\emptyset,0}^0:= 0,$ and
\[
\qt_{\emptyset,l}(\gt_1,\ldots,\gt_l):=
\sum_{\beta\in H_2(X)}T^{\pr(\beta)} \qt_{\emptyset,l}^\beta(\gt_1,\ldots,\gt_l).
\]
Write also
\[
\widetilde{GW}:=\sum_{l\ge 0} \frac{1}{l!}p_* i^* \qt_{\emptyset,l}(\gt).
\]
Define a pairing on $\mC$:
\[
\ll\,,\,\gg:\mC\otimes\mC\lrarr \mathfrak{R},\qquad
\ll\tilde{\xi},\tilde{\eta}\gg:=(-1)^{\deg\etat}p_*(\tilde{\xi}\wedge\tilde{\eta}).
\]
Note that
\begin{equation}\label{eq:pseudopair}
\ll\xit,\etat\gg=(-1)^{\deg\etat}p_*(\xit\wedge\etat)
=(-1)^{\deg\etat+\deg\etat\cdot\deg\xit}p_*(\etat\wedge\xit)
=(-1)^{(\deg\etat+1)(\deg\xit+1)+1}\ll\etat,\xit\gg.
\end{equation}

For each closed $\gt\in \mI_Q\mD$ with $\deg_{\mD}\gt=2,$ define structure maps
\[
\mgt_k:\mC^{\otimes k}\lrarr \mC
\]
by
\[
\mgt_k(\otimes_{j=1}^k\at_j):=\sum_{l}
\frac{1}{l!}\;\qt_{k,l}(\otimes_{j=1}^k\at_j;\gt^{\otimes l}),
\]
and define
\[
\mgt_{-1}:=\sum_{l}\frac{1}{l!}\;\qt_{-1,l}(\gt^{\otimes l})\in A^*(I;R).
\]

\begin{rem}
The operators $\mgt_k$ are called $[0,1]$-parameterized $A_\infty$ operations in Definition~21.29 of~\cite{FOOOKSVFC}. The equivalence between such operations and the pseudoisotopy operations that appear in~\cite{Fukaya} is explained in Lemma 21.31 of~\cite{FOOOKSVFC}.
\end{rem}

\begin{rem}
For the proofs of Propositions~\ref{cl:a_infty_m}-\ref{no_top_deg}, we have for the most part referred to~\cite{FOOOspec}, which follows~\cite{Fukaya} making minor adaptations to extend the results to the setting of bulk deformations. However, the results on pseudoisotopy in~\cite{Fukaya} are not needed and thus not treated in~\cite{FOOOspec}. So, for the following propositions concerning pseudoisotopies, we reference~\cite{Fukaya} directly, with the understanding that similar minor adaptations suffice to extend the results to the setting of bulk deformations.
\end{rem}

The following is Proposition~4.5 in~\cite{ST1}.
Alternatively, it follows from the proof of the $A_\infty$ relations in Theorem 11.1 of~\cite{Fukaya}.
\begin{cl}[$A_\infty$ structure]\label{cl:mgt_str}
The maps $\mgt$ define an $A_\infty$ structure on $\mC$. That is,
\[
\sum_{\substack{k_1+k_2=k+1\\k_1,k_2\ge 0}}
(-1)^{\sum_{j=1}^{i-1}(\deg\at_j+1)}
\mgt_{k_1}(\at_1,\ldots,\at_{i-1},\mgt_{k_2}(\at_i,\ldots,\at_{i+k_2-1}),\at_{i+k_2},\ldots,\at_k)=0
\]
for all $\at_j\in \mC$.
\end{cl}

The following is Proposition 4.6 in~\cite{ST1}. It is implicit in Theorem~11.1 of~\cite{Fukaya} and explicit in~\cite[Remark 21.28]{FOOOKSVFC}.

\begin{cl}[Linearity]\label{cl:t_linear}
The operations $\qt$ are $\mR$-multilinear in the sense that for $f\in \mR,$
\begin{multline*}
\qt_{k,l}^\beta(\at_1,\ldots,\at_{i-1},f.\at_i,\ldots,\at_k;\gt_1,\ldots,\gt_l) =\\		=(-1)^{\deg f\cdot \big(i+\sum_{j=1}^{i-1}\deg\at_j+\sum_{j=1}^l\deg\gt_j\big)}	f.\qt_{k,l}^\beta(\at_1,\ldots,\at_k;\gt_1,\ldots,\gt_l)+\delta_{1,k}\cdot df.\at_1,
\end{multline*}
and for $f\in A^*(I;Q),$
\[
\qt_{k,l}^\beta(\at_1,\ldots,\at_k;\gt_1,\ldots,f.\gt_i,\ldots,\gt_l)
=(-1)^{\deg f\cdot\sum_{j=1}^{i-1}\deg\gt_j}		f.\qt_{k,l}^\beta(\at_1,\ldots,\at_k;\gt_1,\ldots,\gt_l),
\]
and
\[
\qt_{\emptyset,l}^\beta(\gt_1,\ldots,f.\gt_i,\ldots,\gt_l)
=(-1)^{\deg f\cdot\sum_{j=1}^{i-1}\deg\gt_j}f.\qt^\beta_{\emptyset,l}(\gt_1,\ldots,\gt_l).
\]
In addition, the pairing $\ll\;,\,\gg$ is $\mR$-bilinear in the sense that
\[
	\ll f.\at_1,\at_2\gg= f\ll\at_1,\at_2\gg,\quad
	\ll\at_1,f.\at_2\gg= (-1)^{\deg f\cdot(1+\deg\at_1)}f\ll\at_1,\at_2\gg.
\]
\end{cl}

The following is Proposition 4.10 in~\cite{ST1}. Alternatively, it follows from an argument similar to the proof of unitality in Theorem~11.1 of~\cite{Fukaya}.
\begin{cl}[Unit of the algebra]\label{cl:qt_unit}
Fix $f\in A^0(I\times L)\otimes R$, $\at_1,\ldots,\at_{k}\in \mC,$ and $\gt_1,\ldots,\gt_l \in A^*(I\times X;Q).$ Then
\begin{multline*}
\qt_{\,k+1\!,l}^{\beta}(\at_1,\ldots,\at_{i-1},f,\at_{i},\ldots,\at_k ;\otimes_{r=1}^l\gt_r)=\\
=
\begin{cases}
df, & (k+1,\;l,\beta)=(1,0,\beta_0),\\
(-1)^{\deg f}\cdot\at_2, & (k+1,l,\beta)=(2,0,\beta_0),\: i=1,\\
(-1)^{\deg\at_1(\deg f+1)}f\cdot\at_1, & (k+1,l,\beta)=(2,0,\beta_0),\: i=2,\\
0,& \text{otherwise.}
\end{cases}
\end{multline*}
In particular, $1\in A^0(I\times L)$ is a strong unit for the $A_\infty$ operations $\mg$:
\[
\mt_{k+1}^\gamma(\at_1,\ldots,\at_{i-1},1,\at_{i},\ldots,\at_{k})=
\begin{cases}
0, & k+1\ge 3 \mbox{ or } k+1=1,\\
\at_1, & k+1=2,\: i=1,\\
(-1)^{\deg\at_1}\at_1, & k+1=2,\: i=2.
\end{cases}
\]
\end{cl}

The following is Proposition 4.11 in~\cite{ST1}.
Alternatively, it follows from an argument similar to the proof of cyclicity in Theorem 11.1 of~\cite{Fukaya}.
\begin{cl}[Cyclic structure]\label{cl:qt_cyclic}
The $\qt$ are cyclic in the sense of~\cite[Definition 1.1(7)]{ST1} with respect to the inner product $\ll\;,\,\gg.$ That is,
\begin{multline*}
\ll\qt_{k,l}(\at_1,\ldots,\at_k;\gt_1,\ldots\gt_l),\at_{k+1}\gg=\\
=(-1)^{(\deg\at_{k+1}+1)\sum_{j=1}^{k}(\deg\at_j+1)}\cdot
\ll \qt_{k,l}(\at_{k+1},\at_1,\ldots,\at_{k-1};\gt_1,\ldots,\gt_l),\at_k\gg
+\delta_{1,k}\cdot d\ll\at_1,\at_2\gg.
\end{multline*}
In particular, $(\mC,\{\mgt_k\}_{k\ge 0})$ is a cyclic $A_\infty$ algebra over $\mR$ for any $\gt$.
\end{cl}

\begin{rem}
The term $\delta_{1,k}\cdot d\ll\at_1,\at_2\gg$ in the preceding proposition appears to be necessary in general when working with an $A_\infty$ algebra over a differential graded algebra.
\end{rem}

The following is Proposition~4.12 in~\cite{ST1}.
Alternatively, it follows from an argument similar to the proof of Theorem~11.1 in~\cite{Fukaya}.

\begin{cl}[Degree of structure maps]\label{qt_deg_str_map}
For $\gt_1,\ldots,\gt_l\in \mD_2$, $k\ge 0,$ the map
\[
\qt_{k,l}(\; ;\gt_1,\ldots,\gt_l):\mC^{\otimes k}\lrarr \mC
\]
is of degree $2-k$ in $\mC$.
\end{cl}

The following is Proposition~4.15 in~\cite{ST1}. Alternatively, it follows from an argument similar to the proof of Theorem~11.1 of~\cite{Fukaya}.

\begin{cl}[Energy zero]\label{cl:qt_zero}
For $k\ge 0,$
\[
\qt_{k,l}^{\beta_0}(\at_1,\ldots,\at_k;\gt_1,\ldots,\gt_l)=
\begin{cases}
d\at_1, & (k,l)=(1,0),\\
(-1)^{\deg\at_1}\at_1\wedge\at_2, & (k,l)=(2,0),\\
-\gt_1|_L, & (k,l)=(0,1),\\
0, & \text{otherwise}.
\end{cases}
\]
Furthermore,
\[
\qt_{-1,l}^{\beta_0}(\gt_1,\ldots,\gt_l)=0.
\]
\end{cl}

The following is Proposition~4.17 in~\cite{ST1}.
Alternatively, it follows from an argument similar to the proof of unitality in Theorem~11.1 of~\cite{Fukaya}.
\begin{cl}[Top degree]\label{cl:qt_no_top_deg}
Suppose
\[
(k,l,\beta)\not\in\{(1,0,\beta_0),(0,1,\beta_0),(2,0,\beta_0)\}.
\]
Then $(\qt_{k,l}^\beta(\at;\gt))_{n+1}=0$ for all lists $\at,\gt$.
\end{cl}

The other properties formulated for the usual $\q$-operators also hold for $\qt$. Namely, there are analogs for Propositions~\ref{cl:symmetry}, \ref{q_fund}, \ref{cl:q_div}.
The next results relate the cyclic $A_\infty$ structure on $\mC$ to the one on $C$. It will be useful in Section~\ref{subsec:invar}.
For any $t\in I$ and $M=L,X,$ denote by $j_t:M\hookrightarrow I\times M$ the inclusion $p\mapsto (t,p)$. Denote by $\qkl^t$ the $\q$-operators associated to the complex structure $J_t$.
Then following is Lemma~4.7 of~\cite{ST1}. Alternatively, it follows from the proof of Theorem~11.1 of~\cite{Fukaya}.

\begin{lm}[Pseudoisotopy]\label{lm:pseudo}\label{rem:pseudo}
For $t\in I$, we have
\[
j_t^*\qt_{k,l}(\at_1,\ldots,\at_k;\gt_1,\ldots,\gt_l)=
\qkl^t(j_t^*\at_1,\ldots,j_t^*\at_k;j_t^*\gt_1,\ldots,j_t^*\gt_l).
\]
\end{lm}

The following is Lemma~4.9 of~\cite{ST1}.
\begin{lm}\label{lm:d_ll_gg}
For any $\xit,\etat\in A^*(I\times L),$
\[
(-1)^{\deg \xit+\deg \etat +n}
\int_I d\ll\xit,\etat\gg=\langle j_1^*\xit,j_1^*\etat\rangle- \langle j_0^*\xit,j_0^*\etat\rangle.
\]
\end{lm}

The following is Proposition~4.20 of~\cite{ST1}. It uses the cyclic structure $\ll\;,\,\gg$ to rephrase
the $A_\infty$ relations so the case $k=-1$ fits more uniformly. For $k \geq 0,$ the proof is a formal computation combining Propositions~\ref{cl:mgt_str} and~\ref{cl:qt_cyclic}. The case $k = -1$ uses also Proposition~4.4 of~\cite{ST1}. Alternatively, it follows from an argument similar to the proof of Proposition~4.4 of~\cite{Fukaya2}.

\begin{prop}[Unified $A_\infty$ relations on an isotopy]\label{prop:mgt_a_infty}
For $k\ge 0$,
\begin{multline*}
d\ll\mgt_k(\at_1,\ldots,\at_k),\at_{k+1}\gg=\\
=\sum_{\substack{k_1+k_2=k+1\\k_1\ge 1,k_2\ge 0\\1\le i\le k_1}}
(-1)^{\nu(\at;k_1,k_2,i)}\ll\mgt_{k_1}(\at_{i+k_2},\ldots,\at_{k+1},\at_1,\ldots,\at_{i-1}), \mgt_{k_2}(\at_i,\ldots,\at_{k_2+i-1})\gg
\end{multline*}
with
\[
\nu(\at;k_1,k_2,i):=\sum_{j=1}^{i-1}(\deg\at_j+1)+
\sum_{j=i+k_2}^{k+1}(\deg\at_j+1)\Big(\sum_{\substack{m\ne j\\1\le m\le k+1}}(\deg\at_m+1)+1\Big)+1
\]
For $k = -1,$
\[
d\mgt_{-1} =-\frac{1}{2}\ll\mgt_0,\mgt_0\gg+\widetilde{GW}.
\]
\end{prop}
Write $\mgp$ for operations defined using the almost complex structure $J'$ and a closed form $\gamma' \in \mI_QD$ with $\deg_D \gamma' = 2.$
\begin{rem}\label{rem:py}
Propositions~\ref{cl:t_linear},~\ref{cl:qt_unit},~\ref{cl:qt_cyclic},~\ref{qt_deg_str_map}, and~\ref{cl:qt_no_top_deg}, show that $(\mC,\mgt,\ll\;,\,\gg,1)$ is a cyclic unital $A_\infty$ algebra.
Set $\gamma=j_0^*\gt$ and $\gamma'=j_1^*\gt$. By Lemma~\ref{lm:pseudo} the algebra $(\mgt,\ll\;,\,\gg,1)$ is a pseudoisotopy from $(\mg,\langle\;,\,\rangle,1)$ to $(\mgp,\langle\;,\,\rangle,1)$ in the sense of~\cite[Definition 1.3]{ST1}.
\end{rem}

\section{Classification of bounding pairs}\label{sec:bd_chains}

\subsection{Additional notation}\label{ssec:add_not}
This section will describe inductive constructions. To enable the induction process, we arrange the elements of $\sly$ that are represented by $J$-holomorphic curves into a countable list as follows. By Gromov compactness, for each fixed value $E$ the set
\[
\sly_E:=\sly_{J,E}:=\{\beta\;|\;\omega(\beta)\le E,\,\M_{3,0}(\beta)\ne \emptyset\}
\]
is finite. Thus we can order $\sly_\infty$ as a list $\beta_0,\beta_1,\ldots,$ where $i<j$ implies $\omega(\beta_i)\le \omega(\beta_j).$ The notation $\beta_0$ is consistent with the one used above.

In addition, let
\[
T(C): = \bigoplus_{k \ge 0} C^{\otimes k}
\]
denote the completed tensor algebra,
and for $x \in \mI_RC,$ abbreviate
\[
e^x = 1 \oplus x \oplus (x\otimes x) \oplus (x \otimes x \otimes x) \oplus \ldots \in T(C).
\]
Moreover, define
\[
\mg : T(C) \to C
\]
by
\[
\mg(\oplus_{k \geq 0} \eta_k) = \sum_{k \geq 0} \mg_k(\eta_k).
\]
Recall the definition of $\nu$ from Section~\ref{ssec:gennote}.
Let
$\Ups =C$ or $\Ups = R$. Decompose $\Ups$ as
$\Ups = \Ups'\otimes_{\R} R$ where $\Ups'=A^*(L)$ or $\Ups'=\R$, respectively.
Any element $\alpha\in \Ups$ can be written as
\begin{equation}\label{eq:n_decomp}
\alpha=\sum_{i=0}^\infty \lambda_i \alpha_i,\qquad \alpha_i\in \Ups',\quad\lambda_i=T^{\beta_i}s^{k_i}\prod_{a=0}^Nt_a^{l_{ai}},\quad \lim_i\nu(\lambda_i)=\infty.
\end{equation}
Note that
\[
\lambda\in \mI_R \iff \nu(\lambda)>0.
\]
Denote by $F^{E}$ the filtration on $R,C,\mC,$ defined by $\nu$. That is,
\[
\lambda\in F^{E}\!R\iff \nu(\lambda)> E
\]
and similarly for $C$ and $\mC.$
Since $\nu$ is non-negative, we have
$F^EC = F^E\!R\cdot C$ and $F^E\mC = F^E\!R\cdot\mC,$ where the product has implicitly been completed.
\begin{dfn}
A multiplicative submonoid $G\subset R$ is \textbf{\sababa{}} if it can be written as a list
\begin{equation}\label{eq:list}
G=\{\pm\lambda_0=\pm T^{\beta_0}, \pm\lambda_1,\pm\lambda_2,\ldots\}
\end{equation}
such that $i<j\Rarr \nu(\lambda_i)\le \nu(\lambda_j).$
\end{dfn}

\begin{rem}
As is detailed below, the constructions in this section follow the ideas of obstruction theory introduced in~\cite{FOOO}.
Our sababa property is an adaptation of the gapped condition of~\cite{FOOO} to the Novikov ring $\Lambda$ used in the present work.
\end{rem}

For $j=1,\ldots,m,$ and elements $\alpha_j=\sum_i\lambda_{ij}\alpha_{ij}\in \Ups$ decomposed as in~\eqref{eq:n_decomp}, denote by $G(\alpha_1,\ldots,\alpha_m)$ the multiplicative monoid generated by $\{\pm T^\beta\,|\,\beta\in \sly_\infty\}$, $\{t_j\}_{j=0}^N$, and  $\{\lambda_{ij}\}_{i,j}$.
\begin{lm}\label{lm:sababa}
For $\alpha_1,\ldots,\alpha_m\in \mI_R$, the monoid $G(\alpha_1,\ldots,\alpha_m)$ is \sababa{}.
\end{lm}
\begin{proof}
It is enough to prove that for any fixed $E\in\R$ there are only finitely many elements $\lambda\in G:=G(\alpha_1,\ldots,\alpha_m)$ with $\nu(\lambda)\le E.$

Decompose $\alpha_j=\sum_i\lambda_{ij}\alpha_{ij}$ as in~\eqref{eq:n_decomp}. By definition of convergence in $\Ups$,
the set
\[
\widehat{G}_E:=\{\lambda_{ij}\,|\,\nu(\lambda_{ij})\le E\}\cup\{T^\beta\,|\,\beta\in \sly_E\}\cup\{t_j\}_{j=0}^N
\]
is finite. Each $\lambda\in \widehat{G}_E$ is either the identity element $T^{\beta_0}$ or has positive valuation. So, the set
\[
G_E:=\left\{\lambda\in G\,|\,\nu(\lambda)\le E\right\}
\]
is finite.
\end{proof}

For $\alpha_1,\ldots,\alpha_m\in\mI_R$ write the image of $G=G(\{\alpha_j\}_j)$ under $\nu$ as the sequence $\nu(G)=\{E^G_0=0,E^G_1,E^G_2,\ldots\}$ with $E^G_i<E^G_{i+1}$.
Let $\kappa^G_i\in\Z_{\ge 0}$ be the largest index such that $\nu(\lambda_{\kappa^G_i})=E^G_i.$ In future we omit $G$ from the notation and simply write $E_i,\kappa_i,$ since $G$ will be fixed in each instance and no confusion should occur.

\subsection{Existence of bounding chains}\label{construct_bd_ch}
\subsubsection{Statement}
Recall the notion of a bounding pair $(\gamma,b)$  given in Definition~\ref{dfn_bd_pair}.
It is our objective to prove the following result.
\begin{prop}\label{prop:exist}
Assume $H^{i}(L;\R)=0$ for $i\ne 0,n$. Then for any closed $\gamma \in (\mI_QD)_2$ and any $a \in (\mI_R)_{1-n}$, there exists a bounding chain $b$ for $\mg$ such that
$
\int_L b = a.
$
\end{prop}

The proof of Proposition~\ref{prop:exist} is given in Section~\ref{sssec:constr}. We start with a candidate bounding chain and analyze the Maurer-Cartan equation~\eqref{eq:bc} modulo finite valuation.
Lemma~\ref{lm:sababa} ensures we only need to consider a discrete set of filtration values. In each step we find the obstruction to~\eqref{eq:bc}, namely, the excess contribution on the left-hand side. Then we use the cohomological assumptions on $L$ to add corrections to the candidate bounding chain in the given filtration level. Proceeding inductively, the limit is a bounding chain. This is a modification of the technique of~\cite{FOOO}.

\subsubsection{Obstruction chains}\label{sssec:oj}
Fix a \sababa{} multiplicative monoid $G=\{\lambda_j\}_{j=0}^\infty\subset R$ ordered as in~\eqref{eq:list}.
Let $l\ge 0.$ Suppose we have $b_{(l)}\in C$ with $\deg_Cb_{(l)}=1,$ $G(b_{(l)})\subset G,$ and
\begin{equation}\label{assumption}
\mg(e^{b_{(l)}})\equiv c_{(l)} \cdot 1\pmod{F^{E_l}C},\quad c_{(l)}\in (\mI_R)_2.
\end{equation}
Define the obstruction chains $o_j\in A^*(L)$ for $j=\kappa_l+1,\ldots,\kappa_{l+1}$ by
\begin{equation}\label{eq:oj_dfn}
o_j:=[\lambda_j](\mg(e^{b_{(l)}})).
\end{equation}

\begin{lm}\label{lm:u_even}
$|o_j|=2-\deg\lambda_j$.
\end{lm}
\begin{proof}
Recall that $\deg_C b_{(l)}=1$ and that by Proposition~\ref{deg_str_map} we have
$\deg_C\mg_k=2-k$. So,
\[
\deg_C \mg(e^{b_{(l)}})=2.
\]
In particular, $\deg\lambda_jo_j=2$ and $|o_j|=2-\deg\lambda_j$.
\end{proof}

\begin{lm}\label{oj_closed}
$do_i=0$.
\end{lm}
\begin{proof}
By the $A_\infty$ relations and assumption~\eqref{assumption},
\begin{align*}
0=&\mg(e^{b_{(l)}}\otimes\mg(e^{b_{(l)}})\otimes e^{b_{(l)}})\\
\equiv&\mg(e^{b_{(l)}}\otimes c_{(l)} \cdot 1\otimes e^{b_{(l)}})+\sum_{i=\kappa_l+1}^{\kappa_{l+1}}\mg(e^{b_{(l)}}\otimes \lambda_i o_i\otimes e^{b_{(l)}}) \pmod{F^{E_{l+1}}C}.
\end{align*}
The first summand in the second row vanishes by the unit property, Proposition~\ref{cl:unit}. Apply the energy zero property, Proposition~\ref{q_zero}, to compute
\[
0\equiv \sum_{i=\kappa_l+1}^{\kappa_{l+1}}\mg(e^{b_{(l)}}\otimes \lambda_io_i\otimes e^{b_{(l)}})
\equiv \sum_{i=\kappa_l+1}^{\kappa_{l+1}}\m_1^{\beta_0,\gamma}(\lambda_io_i) =\sum_{i=\kappa_l+1}^{\kappa_{l+1}}(-1)^{\deg \lambda_i}\lambda_ido_i\pmod{F^{E_{l+1}}C}.
\]
\end{proof}

\begin{lm}\label{lm:oj-cj}
If $\deg\lambda_j=2,$ then $o_j=c_j\cdot 1$ for some $c_j\in\R.$ If $\deg\lambda_j\ne 2,$ then $o_j\in A^{>0}(L)$.
\end{lm}
\begin{proof}
If $\deg\lambda_j=2,$ Lemma~\ref{lm:u_even} implies $|o_j|=0.$ By Lemma~\ref{oj_closed}, in this case $o_j=c_j\cdot 1$ for some $c_j\in\R.$
Otherwise, $o_j\in A^{>0}(L).$
\end{proof}

\begin{lm}\label{lm:td}
If $\deg\lambda_j = 2-n$ and $(db_{(l)})_n = 0,$ then $o_j = 0$.
\end{lm}
\begin{proof}
Since $\deg \lambda_j =2-n$, it follows that $|o_j|=n$.
Thus,
$
o_j=\left([\lambda_j](\mg(e^{b_{(l)}}))\right)_{n}.
$
Since $\deg b_{(l)} = 1,$ $\gamma|_L=0$, Propositions~\ref{no_top_deg} and~\ref{q_zero} give
\[
(\mg(e^{b_{(l)}}))_{n} = (\q_{0,1}^{\beta_0}(\gamma)+\q_{1,0}^{\beta_0}(b_{(l)}) +\q_{2,0}^{\beta_0}(b_{(l)},b_{(l)}))_{n}
=(-i^*\gamma+db_{(l)}- b_{(l)} \wedge b_{(l)})_{n}=0.
\]
Therefore $o_j=0.$
\end{proof}

\subsubsection{Bounding modulo \texorpdfstring{$F^{E_{l+1}}C$}{E}}\label{sssec:bdmd}
\begin{lm}\label{lm:inductive}
Suppose for all $j\in\{\kappa_l+1,\ldots,\kappa_{l+1}\}$ such that $\deg\lambda_j\ne 2$, there exist $b_j\in A^{1-\deg\lambda_j}(L)$ such that $(-1)^{\deg\lambda_j}db_j=-o_j.$ Then
\[
b_{(l+1)}:=b_{(l)}+\sum_{\substack{\kappa_l+1\le j\le \kappa_{l+1}\\ \deg\lambda_j\ne 2}}\lambda_jb_j
\]
satisfies
\[
\mg(e^{b_{(l+1)}})\equiv c_{(l+1)}\cdot 1\pmod{F^{E_{l+1}}C},\qquad c_{(l+1)}\in (\mI_R)_2.
\]
\end{lm}
\begin{proof}
Without loss of generality, assume that $[\lambda_j](c_{(l)})=0$ for all $j=\kappa_l+1,\ldots,\kappa_{l+1}$. Use Proposition~\ref{q_zero} and Lemma~\ref{lm:oj-cj} to deduce
\begin{align*}
\mg(e^{b_{(l+1)}})
\equiv &\sum_{\substack{\kappa_l+1\le j\le \kappa_{l+1}\\ \deg\lambda_j\ne 2}}\m_1^{\gamma,\beta_0}(\lambda_jb_j)+ \mg(e^{b_{(l)}})\\
\equiv&
\sum_{\substack{\kappa_l+1\le j\le \kappa_{l+1}\\ \deg\lambda_j\ne 2}}(-1)^{\deg\lambda_j}\lambda_jdb_j+
\sum_{\kappa_l+1\le j\le \kappa_{l+1}}\lambda_jo_j+c_{(l)} \cdot 1\\
=&\sum_{\substack{\kappa_l+1\le j\le \kappa_{l+1}\\ \deg\lambda_j= 2}}\lambda_jc_j\cdot 1+ c_{(l)} \cdot 1\pmod{F^{E_{l+1}}C}.
\end{align*}
The lemma now follows with
\[
c_{(l+1)}=\sum_{\substack{\kappa_l+1\le j\le \kappa_{l+1}\\ \deg\lambda_j= 2}}\lambda_jc_j+c_{(l)}\;\in (\mI_R)_2.
\]
\end{proof}

\begin{lm}\label{lm:init}
Let $\zeta\in \mI_R C$. Then
$\mg(e^\zeta)\equiv 0\pmod{F^{E_0}C}.$
\end{lm}
\begin{proof}
By the energy zero property, Proposition~\ref{q_zero},
\[
\mg(e^{\zeta})\equiv \m_0^{\beta_0,\gamma}+\m_1^{\beta_0,\gamma}(\zeta)+\m_2^{\beta_0,\gamma}(\zeta,\zeta)
=-\gamma|_L+d\zeta-\zeta\wedge\zeta \equiv 0
\pmod{F^{E_0}C}.
\]
\end{proof}

\subsubsection{Construction}\label{sssec:constr}
\begin{proof}[Proof of Proposition~\ref{prop:exist}]

Fix $a\in (\mI_R)_{1-n}$ and $\gamma \in (\mI_QD)_2$. Write $G(a)$ in the form of a list as in~\eqref{eq:list}.

Take $\bar{b}_0\in A^n(L)$ any representative of the Poincar\'e dual of a point, and let $b_{(0)}:=a\cdot \bar{b}_0$.
By Lemma~\ref{lm:init}, the chain $b_{(0)}$ satisfies
\[
\mg(e^{b_{(0)}})\equiv 0=c_{(0)}\cdot 1\pmod{F^{E_0}C},\qquad c_{(0)}=0.
\]
Moreover,  $\int_L b_{(0)}=a$, $db_{(0)}=0$, and $\deg_{C}b_{(0)}=n+1-n=1$.

Proceed by induction. Suppose we have $b_{(l)}\in C$ with $\deg_Cb_{(l)}=1$, $G(b_{(l)})\subset G(a),$ and
\[
(db_{(l)})_n =0,\qquad\int_L b_{(l)}=a,\qquad\mg(e^{b_{(l)}})\equiv c_{(l)}\cdot 1\pmod{F^{E_l}C},\quad c_{(l)}\in (\mI_R)_2.
\]
Define the obstruction chains $o_j$ by~\eqref{eq:oj_dfn}.
By Lemma~\ref{oj_closed}, we have $do_j=0$, and by Lemma~\ref{lm:u_even} we have $o_j\in A^{2-\deg\lambda_j}(L).$ To apply Lemma~\ref{lm:inductive}, it is necessary to choose forms $b_j\in A^{1-\deg\lambda_j}(L)$ such that $(-1)^{\deg\lambda_j}db_j=-o_j$ for all $j\in\{\kappa_l+1,\ldots,\kappa_{l+1}\}$ such that $\deg\lambda_j\ne 2.$
If $\deg \lambda_j = 2-n,$ Lemma~\ref{lm:td} gives $o_j = 0,$ so we choose $b_j = 0.$ If $2-n < \deg \lambda_j<~2,$ then $0 < |o_j| < n.$ So, the cohomological assumption on $L$ implies $[o_j] = 0 \in H^*(L;\R),$ and we choose $b_j$ such that $(-1)^{\deg\lambda_j}db_j = -o_j.$ For other possible values of $\deg \lambda_j,$ degree considerations imply $o_j = 0,$ so we choose $b_j =0.$

Lemma~\ref{lm:inductive} now guarantees that $b_{(l+1)}:=b_{(l)}+\sum_{\substack{\kappa_l+1\le j\le \kappa_{l+1}\\ \deg\lambda_j\ne 2}}\lambda_jb_j$
satisfies
\[
\mg(e^{b_{(l+1)}})\equiv c_{(l+1)}\cdot 1\pmod{F^{E_{l+1}}C},\qquad c_{(l+1)}\in (\mI_R)_2.
\]
Since $b_j = 0$ when $\deg \lambda_j = 2-n,1-n,$ it follows that $(db_{(l+1)})_n = (db_{(l)})_n=0$ and $\int_L b_{(l+1)}=\int_L b_{(l)}=a.$

Thus, the inductive process gives rise to a convergent sequence $\{b_{(l)}\}_{l=0}^\infty$ where $b_{(l)}$ is bounding modulo $F^{E_l}C$.
Taking the limit as $l$ goes to infinity, we obtain
\[
b=\lim_l b_{(l)},\quad\deg_C b=1, \quad \int_L b = a, \quad \mg(e^b)= c\cdot 1,\quad c=\lim_lc_{(l)}\in (\mI_R)_2.
\]
\end{proof}

\subsubsection{J-dependent Novikov ring}\label{sssec:Jnov}

The content of this section will be useful in proving the symplectic deformation invariance axiom~\eqref{ax:def} of Theorem~\ref{axioms}.

Recall from Section~\ref{ssec:iainf} the group $\sly,$ which is the quotient of $H_2(X,L;\Z)$ by the subgroup~$S_L.$
The Novikov ring $\L=\L_\omega$ is a completion of a subring of the group ring of $\sly.$  For a fixed $\omega$-tame almost complex structure $J$,
let $\sly^J$ be the smallest submonoid of $\sly$ that contains all classes that can be represented by a $J$-holomorphic disk.
In other words, $\sly^J$ is the submonoid generated by $\sly_{J,\infty}$ of Section~\ref{ssec:add_not}.
Denote by $\L^J =\L^J_\omega \subset \L_\omega$ the subring obtained by completing the group ring of $\sly^J$. Our motivation for introducing $\L^J$ is that, while $\L_\omega$ depends on $\omega$, we show in Lemma~\ref{lm:smNov} that $\L^J$ does not depend on $\omega$ so long as $\omega$ tames~$J$.
Denote by $R^J$ and $Q^J$ the rings defined analogously to $R,Q,$ with $\L^J$ in place of $\L$, and denote by $\mI_{R^J}\triangleleft R^J,\mI_{Q^J}\triangleleft Q^J,$ the ideals defined analogously to $\mI_R,\mI_Q$. Set $C^J:=A^*(L;R^J)$ and $D^J:=A^*(X,L;Q^J)$. Note that the evaluation maps used to define the $A_\infty$ structure
make use only of classes $\beta\in \sly_{J,\infty}$. Therefore,
the restricted operators
\[
\m_k^J:(C^J)^{\otimes k}\lrarr C^J
\]
are well defined and give a cyclic unital $A_\infty$ structure on $C^J$. More generally, for $\gamma\in \mI_{Q^J}D^J$, the bulk deformed operators $\m_k^{J,\gamma}$ again define a cyclic unital $A_\infty$ structure on $C^J$. Bounding pairs and bounding chains in this context are defined as in Definition~\ref{dfn_bd_pair}.

The content of Sections~\ref{sssec:oj},~\ref{sssec:bdmd}, and~\ref{sssec:constr}, is valid verbatim with the $A_\infty$ algebra $(C^J,\{\m_k^{J,\gamma}\}_{k\ge 0})$ in place of $(C, \{\mg_k\}_{k\ge 0})$.
Thus, we have the following analog of Proposition~\ref{prop:exist}.

\begin{prop}\label{prop:existJ}
Assume $H^{i}(L;\R)=0$ for $i\ne 0,n$. Then for any closed $\gamma \in (\mI_{Q^J}D^J)_2$ and any $a \in (\mI_{R^J})_{1-n}$, there exists a bounding chain $b$ for $\m^{J,\gamma}$ such that
$
\int_L b = a.
$
\end{prop}

\subsection{Gauge equivalence of bounding pairs}\label{ssec:classification}
In the following we use the notation of Section~\ref{pseudoisot}. As in Remark~\ref{rem:py}, write $\mgp$ for operations defined using the almost complex structure $J'$ and a closed form $\gamma' \in \mI_QD$ with $\deg_D \gamma' = 2.$
Recall also the definition~\eqref{eq:mR} of $\mR=A^*(I;R)$. Denote by $\pi:I\times X\to X$ the projection.

\begin{dfn}\label{dfn_g_equiv}
We say a bounding pair $(\gamma,b)$ with respect to $J$ is \textbf{gauge equivalent} to a bounding pair $(\gamma',b')$ with respect to $J'$, if there exist
a path $\{J_t\}$ in $\mathcal{J}$ from $J$ to $J'$,
$\gt\in (\mI_Q\mD)_2$, and $\bt\in (\mI_R\mC)_1$ such that
\begin{gather}
j_0^*\gt=\gamma,\quad j_1^*\gt=\gamma',\quad  d\gt=0,\notag\\
j_0^*\bt=b,\quad j_1^*\bt=b',\notag \\
\mgt(e^{\bt})=\ct\cdot 1,\qquad \ct\in (\mI_R\mR)_2, \qquad d\ct = 0.\label{eq:mct}
\end{gather}
In this case, we say that $(\mgt,\bt)$ is a pseudoisotopy from $(\mg,b)$ to $(\m^{\gamma'},b')$ and write $(\gamma,b)\sim(\gamma',b')$. In the special case $J_t= J = J',$ $\gamma=\gamma'$, and $\gt=\pi^*\gamma$, we say $b$ is gauge equivalent to $b'$ as a bounding chain for $\mg$.
\end{dfn}
\begin{rem}
Suppose $(\mgt,\bt)$ is a pseudoisotopy from $(\mg,b)$ to $(\m^{\gamma'},b')$ with $\mg(e^b)=c\cdot 1,$ $\m^{\gamma'}(e^{b'})=c'\cdot 1$, and $\mgt(e^{\bt})=\ct\cdot 1$. The assumption $d\ct=0$ implies $\ct=\hat{c}+f(t) \,dt$ with $\hat{c}$ constant.
Lemma~\ref{rem:pseudo} implies
\[
j_0^* \mgt(e^{\bt}) = \mg(e^b),\qquad j_1^*\mgt(e^{\bt})=\m^{\gamma'}(e^{b'}).
\]
Therefore, $c=j_0^*\ct =\hat{c} = j_1^*\ct =c'$.
\end{rem}

\begin{rem}
This notion of gauge-equivalence is somewhat more general than~\cite[Definition~4.3.1]{FOOO}. Namely, we do not require $\ct$ to be constant, so $\bt$ is not a weak bounding cochain for $\mt^{\gt}$ in the sense of~\cite{FOOO}.
However, this definition is sufficient to ensure invariance of the superpotential. In fact, the extra flexibility is necessary for our purposes, as seen in Lemma~\ref{lm:ob1} below.
This observation indicates that~\eqref{eq:mct} is the appropriate version of the Maurer-Cartan equation for an $A_\infty$ algebra over a differential graded algebra (in the sense of~\cite{ST1}).
\end{rem}

We first prove that the map $\varrho$ given by~\eqref{eqn_rho} is well defined.

\begin{lm}\label{lm:homotopy}
Let $M$ be a manifold with $\d M=\emptyset$ and let $\tilde{\xi}\in A^*(I\times M)$ such that $d\tilde{\xi}=0.$ Then
\[
[j_0^*\xit]=[j_1^*\xit]\in H^*(M).
\]
\end{lm}
\begin{proof}
Without loss of generality, suppose $\xi$ is homogeneous.
Apply the generalized Stokes' theorem~\cite[Proposition 2.3]{ST1}
to the projection $p^M:I\times M\to M$ to obtain
\[
d(p^M_*\xit)=(-1)^{\dim M+1+|\xit|} (j_1^*\xit-j_0^*\xit).
\]
\end{proof}

\begin{lm}\label{lm_rho}
Assume $n>0$. If $(\gamma,b)\sim(\gamma',b')$, then $[\gamma]=[\gamma']$ and $\int_Lb=\int_Lb'$.
\end{lm}
\begin{proof}
By definition of gauge equivalence, there exists a form $\gt\in \mI_Q\mD$ with $d\gt = 0$ such that $j_0^*\gt=\gamma$ and $j_1^*\gt=\gamma'$.
Lemma~\ref{lm:homotopy} therefore implies that $[\gamma]=[\gamma'].$

Choose $\bt$ as in the definition of gauge equivalence. Then equation~\eqref{eq:mct} implies
\[
\int_Lb'-\int_Lb=\int_{\d(I\times L)}\bt
=\int_{I\times L}d\bt=\int_{I\times L}\bigg(\ct\cdot 1-\sum_{(k,l,\beta)\ne (1,0,\beta_0)}\qt_{k,l}^{\beta}(\bt^k;\gt^l)\bigg)_{n+1}.
\]
Since $n>0$, we have $(\ct)_{n+1}=0$.
Since $\deg \bt = 1$ and $\gt|_{I\times L}=0$, Proposition~\ref{cl:qt_no_top_deg}, and Proposition~\ref{cl:qt_zero} imply that
the right hand side vanishes, so $\int_Lb'=\int_Lb$.
\end{proof}

\begin{prop}\label{prop_g_equiv}
Assume $H^{i}(L;\R)=0$ for $i\ne 0,n$.
Let $(\gamma,b)$ be a bounding pair with respect to $J$ and let $(\gamma',b')$ be a bounding pair with respect to $J'$ such that $\varrho([\gamma,b])=\varrho([\gamma',b']).$
Then $(\gamma,b)\sim(\gamma',b')$.
\end{prop}
The proof of Proposition~\ref{prop_g_equiv} is given toward the end of the section based on the construction detailed in the following.

\begin{lm}\label{B_l}
$H^{j}(I\times L, \d(I\times L);\R) \simeq H^{j-1}(L;\R).$
\end{lm}
\begin{proof}
Abbreviate $M:=I\times L$.
Observe that $M$ deformation retracts to $L$, and $\d M\simeq L\sqcup L$. Substituting this in the long exact sequence
\[
\cdots \rarr H^{j-1}(\d M)\rarr H^j (M,\d M)\rarr H^j(M)\rarr H^j(\d M)\rarr \cdots
\]
and using the injectivity of $H^j(M)\to H^j(\d M)$, we see that for all $j$,
\begin{align*}
H^j(M,\d M)\simeq&\Coker\left(H^{j-1}(M)\to H^{j-1}(\d M)\right)\\
\simeq& \Coker\left(H^{j-1}(L)\to H^{j-1}(L)^{\oplus 2}\right)
\simeq H^{j-1}(L).
\end{align*}
\end{proof}

\begin{lm}\label{lm:ob1}
Suppose $\alpha \in A^1(I\times L),\,\alpha|_{\d(I\times L)} = 0,$ and $d\alpha = 0.$ Then, there exists $\eta \in A^0(I\times L),\,\eta|_{\d(I\times L)} = 0,$ such that $d \eta = \alpha + r\, dt$ for some $r\in \R.$
\end{lm}
\begin{proof}
By Lemma~\ref{B_l}, we have $H^1(I \times L,\d(I\times L);\R) = \R,$ so it is generated by the class~$[dt].$ Consequently, there exists $r \in \R$ such that $[\alpha + r\,dt] = 0 \in H^1(I \times L,\d(I\times L);\R).$ Thus, there exists $\eta \in A^0(I\times L)$ such that $d\eta = \alpha + r\,dt$ and $\eta|_{\d(I\times L)} = 0.$
\end{proof}

Fix a \sababa{} multiplicative monoid $G\subset R$ ordered as in~\eqref{eq:list}.
Suppose $\gt\in (\mI_Q{\mD})_2$ is closed. Let $l\ge 0,$ and suppose we have $\bt_{(l-1)}\in \mC$ such that $G(\bt_{(l-1)})\subset G$ and $\deg_{\mC}\bt_{(l-1)}=1.$
If $l\ge 1,$ assume in addition that
\begin{gather*}
\mgt(e^{\bt_{(l-1)}})\equiv \ct_{(l-1)}\cdot 1\pmod{F^{E_{l-1}}\mC}, \quad \ct_{(l-1)}\in(\mI_R\mR)_2, \quad d\ct_{(l-1)}=0.
\end{gather*}
Define the obstruction chains $\ot_j \in A^*(I\times L)$ by
\begin{equation}\label{eq:ojt_dfn}
\ot_j:=[\lambda_j](\mgt(e^{\bt_{(l-1)}})),\qquad j=\kappa_{l-1}+1,\ldots,\kappa_l.
\end{equation}

\begin{lm}\label{lm:ojt_closed}
$d\ot_j=0.$
\end{lm}
The proof is by an argument similar to that of Lemma~\ref{oj_closed}.

\begin{lm}\label{lm:ut_even}
$|\ot_j|=2-\deg\lambda_j$.
\end{lm}
The proof is similar to that of Lemma~\ref{lm:u_even}.

\begin{lm}\label{lm:deg_ut}
If $\deg \lambda_j = 1-n$ and $(d\bt_{(l-1)})_{n+1}=0$, then $\ot_j=0.$
\end{lm}
\begin{proof}
By definition,
\[
(\ot_j)_{n+1}=\left([\lambda_j](\mgt(e^{\bt_{(l-1)}}))\right)_{n+1}.
\]
By Propositions~\ref{cl:qt_no_top_deg} and~\ref{cl:qt_zero},
\begin{multline*}
(\mgt(e^{\bt_{(l-1)}}))_{n+1} = (\qt_{0,1}^{\beta_0}(\gt)+\qt_{1,0}^{\beta_0}(\bt_{(l-1)})
+ \qt_{2,0}^{\beta_0}(\bt_{(l-1)},\bt_{(l-1)}))_{n+1}= \\
=(-i^*\gt+d\bt_{(l-1)} - \bt_{(l-1)}\wedge \bt_{(l-1)})_{n+1}=0.
\end{multline*}
Therefore $(\ot_j)_{n+1}=0.$
\end{proof}

\begin{lm}\label{lm:ut_relative}
Let $i=0$ or $i=1.$
Write
\[
b_{(l-1)}=j_i^*\bt_{(l-1)}, \qquad \gamma=j_i^*\gt.
\]
Suppose
\begin{gather*}
\mg(e^{b_{(l-1)}})\equiv c_{(l)}\cdot 1\pmod{F^{E_{l}}C},\quad c_{(l)}\in(\mI_R)_2.
\end{gather*}
If $\deg\lambda_j\ne 2$, then $j_i^*\ot_j=0$. If $\deg\lambda_j=2$, then $\ot_j=\ct_j\cdot 1$ with $\ct_j =  [\lambda_j](c_{(l)}).$
\end{lm}
\begin{proof}
It follows from Lemma~\ref{lm:ut_even} and Lemma~\ref{lm:ojt_closed} that if $\deg\lambda_j=2$, then $\ot_j=\ct_j\cdot 1$ for some constant $\ct_j\in \R.$
Lemma~\ref{rem:pseudo} implies
\[
j_i^*\ot_j =j_i^*[\lambda_j](\mgt(e^{\bt_{(l-1)}}))=
[\lambda_j](\mg(e^{b_{(l-1)}}))=
[\lambda_j](c_{(l)}\cdot 1).
\]
Since $\deg c_{(l)} = 2,$ in the case $\deg \lambda_j = 2,$ we have $\ct_j=[\lambda_j](c_{(l)})$, and otherwise $j_i^*\ot_j=0$.
\end{proof}

\begin{lm}\label{lm:ut_exact}
Suppose for all $j\in\{\kappa_{l-1}+1,\ldots,\kappa_{l}\}$ such that $\deg\lambda_j\ne 2$, there exist $\bt_j\in A^{1-\deg\lambda_j}(I\times L)$ such that $(-1)^{\deg\lambda_j}d\bt_j=-\ot_j + \ct_j \, dt$ with $\ct_j \in \R.$ Then
\[
\bt_{(l)}:=\bt_{(l-1)}+\sum_{\substack{\kappa_{l-1}+1\le j\le\kappa_l \\ \deg\lambda_j\ne 2}}\lambda_j\bt_j
\]
satisfies
\begin{gather*}
\mgt(e^{\bt_{(l)}})\equiv \ct_{(l)}\cdot 1\pmod{F^{E_l}\mC},\quad \ct_{(l)}\in (\mI_R\mR)_2, \quad d\ct_{(l)} = 0.
\end{gather*}
\end{lm}
\begin{proof}
Without loss of generality, assume that $[\lambda_j](\ct_{(l-1)})=0$ for all $j=\kappa_{l-1}+1,\ldots,\kappa_{l}$. Use Proposition~\ref{cl:qt_zero} and Lemma~\ref{lm:ut_relative} to deduce
\begin{align*}
\mgt(e^{\bt_{(l)}})
\equiv &\sum_{\substack{\kappa_{l-1}+1\le j\le \kappa_{l}\\ \deg\lambda_j\ne 2}}\mt_1^{\gamma,\beta_0}(\lambda_j\bt_j)+ \mgt(e^{\bt_{(l-1)}})\\
\equiv&
\sum_{\substack{\kappa_{l-1}+1\le j\le \kappa_{l}\\ \deg\lambda_j\ne 2}}
(-1)^{\deg\lambda_j}
\lambda_jd\bt_j+
\sum_{\kappa_{l-1}+1\le j\le \kappa_{l}}\lambda_j\ot_j+\ct_{(l-1)} \cdot 1\\
=&
\sum_{\substack{\kappa_{l-1}+1\le j\le \kappa_{l}\\ \deg\lambda_j= 1}}\lambda_j \ct_j dt+
\sum_{\substack{\kappa_{l-1}+1\le j\le \kappa_{l}\\ \deg\lambda_j= 2}}\lambda_j\ct_j\cdot 1+ \ct_{(l-1)} \cdot 1\pmod{F^{E_{l}}C}.
\end{align*}
The lemma now follows with
\[
\ct_{(l)}=
\sum_{\substack{\kappa_{l-1}+1\le j\le \kappa_{l}\\ \deg\lambda_j= 1}}\lambda_j \ct_j dt+ \sum_{\substack{\kappa_{l-1}+1\le j\le \kappa_{l}\\ \deg\lambda_j= 2}}\lambda_j \ct_j+\ct_{(l-1)}\;\in (\mI_R\mR)_2.
\]
\end{proof}

\begin{proof}[Proof of Proposition \ref{prop_g_equiv}]
We construct a pseudoisotopy $(\mgt,\bt)$ from $(\mg,b)$ to $(\m^{\gamma'},b').$
Let $\{J_t\}_{t \in I}$ be a path from $J$ to $J'$ in $\J$.
Let $\xi\in (\mI_QD)_1$
be such that
\[
\gamma'-\gamma=d\xi
\]
and define
\[
\gt:=\gamma+t(\gamma'-\gamma)+dt\wedge\xi\in (\mI_Q\mD)_2.
\]
Then
\begin{gather*}
d\gt=dt\wedge\d_t(\gamma+t(\gamma'-\gamma))-dt\wedge d\xi=0,\\
j_0^*\gt=\gamma,\quad j_1^*\gt=\gamma',\quad i^*\gt=0.
\end{gather*}

We now move to constructing $\bt$. Write $G(b,b')$ in the form of a list as in~\eqref{eq:list}.
Since $\int_Lb'=\int_Lb,$ there exists $\eta\in \mI_RC$ such that $G(\eta)\subset G(b,b')$, $|\eta|=n-1$, and $d\eta=(b')_n-(b)_n.$
Write
\[
\bt_{(-1)}:=b+t(b'-b)+dt\wedge\eta\in\mC.
\]
Then
\begin{gather*}
(d\bt_{(-1)})_{n+1}=dt\wedge(b'-b)_n-dt\wedge d\eta=0,\\
j_0^*\bt_{(-1)}=b,\qquad j_1^*\bt_{(-1)}=b'.
\end{gather*}
Let $l\ge 0.$ Assume by induction that we have constructed $\bt_{(l-1)}\in \mC$ with $\deg_{\mC}\bt_{(l-1)}=1,$ $G(\bt_{(l-1)})\subset G(b,b'),$ such that
\[
(d\bt_{(l-1)})_{n+1}=0,\qquad
j_0^*\bt_{(l-1)}=b,\qquad j_1^*\bt_{(l-1)}=b',
\]
and if $l\ge 1$, then
\[
\mgt(e^{\bt_{(l-1)}})\equiv \ct_{(l-1)}\cdot 1\pmod{F^{E_{l-1}}\mC}, \qquad \ct_{(l-1)}\in(\mI_R\mR)_2, \quad d\ct_{(l-1)} = 0.
\]
Define the obstruction chains $\ot_j$ by ~\eqref{eq:ojt_dfn}.
By Lemma~\ref{lm:ut_even} we have $\ot_j\in A^{2-\deg\lambda_j}(I\times L),$ by Lemma~\ref{lm:ojt_closed} we have $d\ot_j=0$, and by Lemma~\ref{lm:ut_relative} we have $\ot_j|_{\d(I\times L)}=0$ whenever $\deg \lambda_j \ne 2$.
To apply Lemma~\ref{lm:ut_exact}, it is necessary to choose forms $\bt_j\in A^{1-\deg\lambda_j}(I\times L)$ such that $(-1)^{\deg\lambda_j}d\bt_j=-\ot_j+ \ct_j \,dt,\, \ct_j \in \R,$ for $j\in\{\kappa_l+1,\ldots,\kappa_{l+1}\}$ such that $\deg\lambda_j\ne 2.$
If $\deg \lambda_j = 1-n,$ Lemma~\ref{lm:deg_ut} implies that $\ot_j = 0,$ so we choose $\bt_j = 0.$
If $\deg \lambda_j = 1,$ Lemma~\ref{lm:ob1} gives $\bt_j$ such that $(-1)^{\deg\lambda_j}d \bt_j = -\ot_j + \ct_j\, dt,\,\ct_j \in \R,$ and $\bt_j|_{\d(I\times L)} = 0.$
If $1-n < \deg \lambda_j< 1,$ then $1 < |\ot_j| < n+1.$ So, the cohomological assumption on $L$ and Lemma~\ref{B_l} imply that $[\ot_j] = 0 \in H^*(I\times L,\d(I\times L);\R).$ Thus, we choose $\bt_j$ such that $(-1)^{\deg\lambda_j}d\bt_j = -\ot_j$ and $\bt_j|_{\d(I\times L)} = 0.$ For other possible values of $\deg \lambda_j,$ degree considerations imply $\ot_j = 0,$ so we choose $\bt_j =0.$
By Lemma~\ref{lm:ut_exact}, the form
\[
\bt_{(l)}:=\bt_{(l-1)}+\sum_{\substack{\kappa_{l-1}+1\le j\le\kappa_l \\ \deg\lambda_j\ne 2}}\lambda_j\bt_j
\]
satisfies
\[
\mgt(e^{\bt_{(l)}})\equiv \ct_{(l)}\cdot 1\pmod{F^{E_{l}}\mC}, \qquad \ct_{(l)}\in(\mI_R\mR)_2, \qquad d\ct_{(l)} = 0.
\]
Since we have chosen $\bt_j =0$ when $\deg \lambda_j = 1-n,$ it follows that
\[
(d\bt_{(l)})_{n+1}= (d\bt_{(l-1)})_{n+1}=0.
\]
Since $\bt_j|_{\d(I\times L)}=0$ for all $j\in\{\kappa_l+1,\ldots,\kappa_{l+1}\}$ such that $\deg\lambda_j\ne 2$, we have
\[
j_0^*\bt_{(l)}=j_0^*\bt_{(l-1)}=b,\qquad j_1^*\bt_{(l)}=j_1^*\bt_{(l-1)}
= b'.
\]

Taking the limit $\bt=\lim_l\bt_{(l)},$ we obtain a bounding chain for $\mgt$ that satisfies $j_0^*\bt=b$ and $j_1^*\bt=b'$. So, $(\mgt,\bt)$ is a pseudoisotopy from $(\mg,b)$ to $(\m^{\gamma'},b').$
\end{proof}

We are now ready to prove the classification result for bounding chains.
\begin{proof}[Proof of Theorem~\ref{thm1}]
Surjectivity of $\varrho$ is guaranteed by Proposition~\ref{prop:exist}, and injectivity by Proposition~\ref{prop_g_equiv}.
\end{proof}

\section{Real bounding pairs}\label{sec:real}

\subsection{Preliminaries}\label{ssec:prelim}
Let $\phi:X\to X$ be an anti-symplectic involution, that is, $\phi^*\omega=-\omega.$ Let $L\subset \fix(\phi)$ and $J\in\J_\phi$. In particular, $\phi^*J=-J$.
For the entire Section~\ref{sec:real}, we take $\ssly_L \subset  H_2(X,L;\Z)$ with $\Im(\Id+\phi_*) \subset \ssly_L,$ so $\phi_*$ acts on $\sly_L = H_2(X,L;Z)/\ssly_L$ as $-\Id.$ Also, we take $\deg t_j\in 2\Z$ for all $j=0,\ldots,N$.
Denote by $\tilde\phi_\beta$ the involution induced on $\M_{k+1,l}(\beta)$, defined as follows. Given a nodal Riemann surface with boundary $\Sigma$ with complex structure $j,$ denote by $\overline \Sigma$ a copy of $\Sigma$ with the complex structure $-j.$ Denote by $\psi_\Sigma: \overline \Sigma \to\Sigma$ the anti-holomorphic map given by the identity map on points. Then
\begin{multline*}
\tilde\phi_\beta\,[u : (\Sigma,\d\Sigma)\to (X,L),\vec{z}=(z_0,z_1,\ldots,z_k),\vec{w}=(w_1,\ldots,w_l)]:=\\
=[\phi\circ u\circ \psi_\Sigma, (\psi^{-1}_\Sigma(z_0),\psi^{-1}_\Sigma(z_k),\ldots,\psi^{-1}_\Sigma(z_1)) ,(\psi^{-1}_\Sigma(w_1),\ldots, \psi^{-1}_\Sigma(w_l))].
\end{multline*}
Thus, for $[u:(\Sigma,\d \Sigma)\to(X,L),\vec{z},\vec{w}]\in\M_{k+1,l}(\beta),$ recalling the choice of $\sly,$ we have
\[
[(\phi\circ u\circ \psi_\Sigma)_*[\overline\Sigma,\d \overline\Sigma]]=[-\phi_*u_*[\Sigma,\d \Sigma]]=-\phi_*\beta=\beta.
\]
Therefore, $\tilde{\phi}_\beta$ indeed acts on $\M_{k+1,l}(\beta)$.
For $\sigma \in S_k$ and $\alpha = (\alpha_1,\ldots,\alpha_k)$ with $\alpha_j \in C,$ let
\[
s_\sigma^{[1]}(\alpha)=\sum_{\substack{i<j\\ \sigma(i)>\sigma(j)}} (\deg\alpha_{\sigma(i)}+1)(\deg\alpha_{\sigma(j)}+1).
\]
Let $\tau\in S_k$ denote the permutation $(k,k-1,\ldots,1).$

\begin{prop}\label{prop:q_sgn_c}
Let $\alpha = (\alpha_1,\ldots,\alpha_k)$ with $\alpha_j \in A^*(L)$ and let $\eta_1,\ldots,\eta_l \in A^*(X).$ Then
\[
\qkl^{\beta} (\alpha_1,\ldots,\alpha_k;\eta_1,\ldots,\eta_l)=
(-1)^{sgn(\tilde{\phi}_\beta)+s_\tau^{[1]}(\alpha) +\frac{k(k-1)}{2}}
\qkl^{\beta}(\alpha_k,\ldots,\alpha_1; \phi^*\eta_1,\ldots,\phi^*\eta_l).
\]
\end{prop}

\begin{proof}
First compute the contribution to the change on the differential form level.
\begin{align*}
(evb_0)_*(\wedge_{j=1}^k&evb_j^*\alpha_{k+1-j}\wedge \wedge_{j=1}^levi_j^*\phi^*\eta_{j})= \\
&=(evb_0)_* (\wedge_{j=1}^k(evb_{k+1-j}\circ\tilde{\phi}_\beta)^*\alpha_{k+1-j} \wedge \wedge_{j=1}^l(evi_j\circ\tilde\phi_\beta)^*\eta_{j})\\
&=(evb_0)_* \tilde\phi_\beta^*(\wedge_{j=1}^kevb_{k+1-j}^*\alpha_{k+1-j} \wedge\wedge_{j=1}^levi_j^*\eta_{j})\\
&=(-1)^{sgn(\tilde{\phi}_\beta) +\sum_{\substack{i<j\\\tau(i)>\tau(j)}} |\alpha_{\tau(i)}||\alpha_{\tau(j)}|}
(evb_0)_*(\wedge_{j=1}^kevb_j^*\alpha_j\wedge \wedge_{j=1}^levi_j^*\eta_{j}).
\end{align*}
Now compute the sign coming from the definition of the $\q$-operators.
\begin{multline*}
\varepsilon(\alpha_k,\ldots,\alpha_1)=
\sum_{j=1}^kj(|\alpha_{k+1-j}|+1)+1
=\sum_{j=1}^k(k+1-j)(|\alpha_j|+1)+1\equiv \\
\equiv (k+1)\sum_{j=1}^k(|\alpha_j|+1) +\varepsilon(\alpha)\equiv(k+1)\sum_{j=1}^k|\a_j| +k(k+1) +\varepsilon(\alpha)\equiv\\
\equiv (k+1)\sum_{j=1}^k|\a_j| +\varepsilon(\alpha) \pmod 2.
\end{multline*}
Since
\begin{align*}
s_\tau^{[1]}(\alpha)
=&
\sum_{\substack{i<j\\\tau(i)>\tau(j)}} |\alpha_{\tau(i)}||\alpha_{\tau(j)}|
+
\sum_{\substack{i<j\\\tau(i)>\tau(j)}} (| \alpha_{\tau(i)}|+|\alpha_{\tau(j)}|)
+
sgn(\tau)\\
=&
\sum_{\substack{i<j\\\tau(i)>\tau(j)}} |\alpha_{\tau(i)}||\alpha_{\tau(j)}|
+
\sum_{i=1}^k(k-1)|\alpha_i|
+
\frac{k(k-1)}{2}\\
\equiv &
\sum_{\substack{i<j\\\tau(i)>\tau(j)}} |\alpha_{\tau(i)}||\alpha_{\tau(j)}|
+
(k+1) \sum_{j=1}^k|\a_j|
+
\frac{k(k-1)}{2} \pmod 2,
\end{align*}
the result follows.
\end{proof}

Recall that the relative spin structure $\s$ on $L$ determines a class $w_\s \in H^2(X;\Z/2\Z)$ such that $w_2(TL) = i^* w_\s.$ Recall also the definition of the doubling map $\chi: \sly \to H_2(X;\Z)$ from~\eqref{eq:chi}.

\begin{prop}[Sign of conjugation on the moduli space]\label{prop:sgn_phit}
The sign of the involution $\tilde\phi_\beta:\M_{k+1,l}(\beta)\to \M_{k+1,l}(\beta)$ is given by
\[
sgn(\tilde\phi_\beta)
\equiv \frac{\mu(\beta)}{2} +  w_\s(\chi(\beta)) +k+l+1+\frac{k(k-1)}{2} \pmod 2.
\]
\end{prop}
\begin{proof}
Denote by $\M_{k+1,l}^S(\beta)$ the moduli space of $J$-holomorphic stable maps $u:(D^2,\d D^2) \to (X,L)$ of degree $\beta \in \sly$ with $k+1$ unordered boundary points and $l$ interior marked points. Thus, $\M_{k+1,l}(\beta)$ is an open and closed subset of $\M^S_{k+1,l}(\beta).$
The space $\M_{k+1,l}^S(\beta)$ carries a natural orientation induced by the spin structure on $L$ as in~\cite[Chapter~8]{FOOO} or~\cite[Chapter~5]{SolomonThesis}. Denote by
\[
\varphi_\sigma:\M_{k+1,l}^S(\beta)\lrarr \M_{k+1,l}^S(\beta)
\]
the diffeomorphism
corresponding to relabeling boundary marked points by a permutation $\sigma \in S_{k+1}$. The orientation of $\M_{k+1,l}^S(\beta)$ changes by $(-1)^{sgn(\sigma)}$ under $\varphi_\sigma$.

Let
\[
\tilde\phi^S_\beta:\M_{k+1,l}^S(\beta)\lrarr\M_{k+1,l}^S(\beta)
\]
be given by
\begin{multline*}
\tilde\phi^S_\beta\,[u:(\Sigma,\d\Sigma)\to(X,L),\vec{z}=(z_0,z_1,\ldots,z_k),\vec{w}=(w_1,\ldots,w_l)]:=\\
=[\phi\circ u\circ \psi_\Sigma, (\psi^{-1}_\Sigma(z_{0}),\psi^{-1}_\Sigma(z_{1}),\ldots,\psi^{-1}_\Sigma(z_{k})) ,(\psi^{-1}_\Sigma(w_1),\ldots, \psi^{-1}_\Sigma(w_l))].
\end{multline*}
By~\cite[Proposition 5.1]{SolomonThesis}, the sign of $\tilde\phi^S$ is given by
\[
sgn(\tilde\phi^S_\beta)
\equiv \frac{\mu(\beta)}{2} + w_\s(\chi(\beta)) +k+l+1
\pmod 2.
\]
The additional summand $(n+1)\frac{k(k-1)}{2}$ in the sign of Proposition 5.1 of~\cite{SolomonThesis} arises from the twisting of the orientation of the moduli space $\M_{k+1,l}^S(\beta)$ that results from Definition 3.2 of~\cite{SolomonThesis}. Thus it is not relevant here. Let $\sigma=(0,k,k-1,\ldots,1)\in S_{k+1}$. Then $sgn(\sigma)=\frac{k(k-1)}{2}$ and
$\tilde\phi_\beta=\varphi_\sigma\circ\tilde\phi^S_\beta|_{\M_{k+1,l}(\beta)}$.
\end{proof}

\begin{rem}
Proposition~\ref{prop:sgn_phit} is almost a special case of Theorem 1.4 in~\cite{FOOOinv}. However, since~\cite{FOOOinv} considers fiber products of the moduli space $\M_{k+1,l}(\beta)$ with singular cochains rather than the moduli space itself, we include a proof of the proposition for the reader's convenience. Alternately, one can deduce the proposition from the proof in~\cite{FOOOinv}. That proof treats the fully general setting of the virtual fundamental class.
\end{rem}

\begin{cor}\label{sgn2}
Let $\alpha = (\alpha_1,\ldots,\alpha_k)$ with $\alpha_j \in C$ and let $\eta_1,\ldots,\eta_l \in A^*(X;Q).$ Then
\begin{multline*}
\phi^*\qkl^\beta(\alpha_1,\ldots,\alpha_k; \eta_1,\ldots,\eta_l) =\\
=(-1)^{\frac{\mu(\beta)}{2} +w_\s(\chi(\beta)) +l+k+1+s_\tau^{[1]}(\alpha)} \qkl^\beta(\phi^*\alpha_k,\ldots,\phi^*\alpha_1; \phi^*\eta_1,\ldots,\phi^*\eta_l).
\end{multline*}
\end{cor}
\begin{proof}
Combine Propositions~\ref{lm:qlinear},~\ref{prop:q_sgn_c}, and~\ref{prop:sgn_phit}.
\end{proof}

\subsection{The real spin case}\label{ssec:spin}

Recall the action of $\phi^*$ on $\L,Q,R,C,D,$ defined in equation~\eqref{eq:phi*ext} and the definition of real elements~\eqref{eq:relt}.
\begin{prop}\label{prop:sd}
Suppose the relative spin structure $\s$ on $L$ is in fact a spin structure, and $\gamma \in \mathcal{I}_Q D$ is real. Then, for all $\alpha = (\alpha_1,\ldots,\alpha_k)$ with $\alpha_j \in C,$ we have
\[
\phi^* \mg_k(\alpha_1,\ldots,\alpha_k) = (-1)^{1+k + s_\tau^{[1]}(\alpha)}\mg_k(\phi^*\alpha_k,\ldots,\phi^*\alpha_1).
\]
\end{prop}
\begin{proof}
Since $\s$ is a spin structure, $w_\s = 0.$ Using Corollary~\ref{sgn2} and the fact that $\gamma$ is real, we calculate
\begin{align*}
\phi^* \mg_k(\alpha_1,\ldots,\alpha_k) &= \sum_{\beta \in \sly,\,l \geq 0}(-1)^{\frac{\mu(\beta)}2}\frac{T^\beta}{l!} \phi^*\qkl^\beta(\alpha_1,\ldots,\alpha_k; \gamma^{\otimes l}) \\
&= \sum_{\beta \in \sly,\,l \geq 0}(-1)^{k+1+l+s_\tau^{[1]}(\alpha)}\frac{T^\beta}{l!} \qkl^\beta(\phi^*\alpha_k,\ldots,\phi^*\alpha_1;(\phi^*\gamma)^{\otimes l})\\
&= \sum_{\beta \in \sly,\,l \geq 0}(-1)^{k+1+s_\tau^{[1]}(\alpha)}\frac{T^\beta}{l!} \qkl^\beta(\phi^*\alpha_k,\ldots,\phi^*\alpha_1;\gamma^{\otimes l}) \\
&= (-1)^{1+k + s_\tau^{[1]}(\alpha)}\mg_k(\phi^*\alpha_k,\ldots,\phi^*\alpha_1).
\end{align*}
\end{proof}

\begin{lm}\label{lm:reob}
Suppose $\s$ is a spin structure and $\gamma \in \mathcal{I}_Q D,\,b \in C,$ are real with $\deg b = 1.$ Then $\mg_k(b^{\otimes k})$ is real.
\end{lm}
\begin{proof}
Since $\deg b = 1,$ it follows that $s_\tau^{[1]}(b,\ldots,b)=0.$ So, by Proposition~\ref{prop:sd}, we have
\[
\phi^* \mg_k(b^{\otimes k}) = (-1)^{1+k + s_\tau^{[1]}(b,\ldots,b)}\mg_k((\phi^*b)^{\otimes k}) = - \mg_k(b^{\otimes k}).
\]
\end{proof}
\begin{rem}
It is instructive to compare Proposition~\ref{prop:sd} and Lemma~\ref{lm:reob} above to Theorem~5.1 and Lemma~6.6 of~\cite{Solomon2018} respectively. In the terminology of~\cite{Solomon2018}, Proposition~\ref{prop:sd} says that the $A_\infty$ algebra $(C,\mg)$ is $\phi^*$ self-dual.
\end{rem}

\begin{lm}\label{lm:reR}
Suppose $n \not \equiv 1 \pmod 4.$ An element $a \in R$ is real if and only if $\deg a \equiv 2$ or $1-n \pmod 4.$
\end{lm}
\begin{proof}
Recall that $\phi^* t_i = -t_i$ if and only if $\deg t_i \equiv 2 \pmod 4$. Moreover, $\phi^* s = -s$ and $\deg s = 1-n.$ In particular, when $n$ is even, $s^2 = 0,$ and when $n$ is odd, since $n \not\equiv 1,$ we have $\deg s = 2.$ The lemma follows.
\end{proof}

\begin{lm}\label{lm:reC}
Suppose $\alpha \in C$ is real with $\deg \alpha = d$ and $n \not \equiv 1 \pmod 4.$ Write
\[
\alpha = \sum_i f_i \alpha_i
\]
where $0\neq f_i \in R$ and $\alpha_i \in A^*(L).$ Then, $|\alpha_i| \equiv d-2$ or $d-1+n \pmod 4.$
\end{lm}
\begin{proof}
Since $\phi^*$ acts trivially on forms on $L$ and $\alpha$ is real, it follows that $f_i$ is real for all~$i.$ Since $\deg \alpha = d,$ it follows from Lemma~\ref{lm:reR} that
\[
\deg \alpha_i = d - \deg f_i \equiv d-2 \text{ or } d - 1 + n \pmod 4.
\]
\end{proof}

\begin{prop}\label{cor:spin_b_exist}
Suppose $\s$ is a spin structure, $n \not \equiv 1 \pmod 4,$ and $H^k(L;\R) = 0$ for $k \equiv 0,n+1\pmod 4, \, k \neq 0,n.$ Then for any real closed $\gamma \in (\mI_QD)_2$ and any $a \in (\mI_R)_{1-n}$, there exists a real bounding chain $b$ for $\mg$ such that
$
\int_L b = a.
$
\end{prop}
\begin{proof}
Fix $a\in (\mI_R)_{1-n}$ and a real $\gamma \in (\mI_QD)_2$. Write $G(a)$ in the form of a list as in~\eqref{eq:list}.

Take $\bar{b}_0\in A^n(L)$ any representative of the Poincar\'e dual of a point, and let $b_{(0)}:=a\cdot \bar{b}_0$.
By Lemma~\ref{lm:init}, the chain $b_{(0)}$ satisfies
\[
\mg(e^{b_{(0)}})\equiv 0=c_{(0)}\cdot 1\pmod{F^{E_0}C},\qquad c_{(0)}=0.
\]
Moreover, $\int_L b_{(0)} = a$, $db_{(0)}=0$, $\deg_{C}b_{(0)}=n+1-n=1$, and Lemma~\ref{lm:reR} implies $b_{(0)}$ is real.

Proceed by induction. Suppose we have a real $b_{(l)}\in C$ with $\deg_Cb_{(l)}=1$, $G(b_{(l)})\subset G(a),$ and
\[
(db_{(l)})_n = 0, \qquad
\int_L b_{(l)} = a, \qquad \mg(e^{b_{(l)}})\equiv c_{(l)}\cdot 1\pmod{F^{E_l}C},\qquad c_{(l)}\in (\mI_R)_2.
\]
Define obstruction chains $o_j$ by~\eqref{eq:oj_dfn}.
By Lemma~\ref{oj_closed}, we have $do_j=0$, and by Lemma~\ref{lm:u_even} we have $o_j\in A^{2-\deg\lambda_j}(L).$ To apply Lemma~\ref{lm:inductive}, it is necessary to choose forms $b_j\in A^{1-\deg\lambda_j}$ such that $(-1)^{\deg\lambda_j}db_j=-o_j$
for $j\in\{\kappa_l+1,\ldots,\kappa_{l+1}\}$ such that $\deg\lambda_j\ne 2.$
If $\lambda_j$ is not real, Lemma~\ref{lm:reob} implies that $o_j = 0$, so we choose $b_j = 0.$ If $\deg \lambda_j = 2-n,$ Lemma~\ref{lm:td} implies $o_j = 0,$ so we choose $b_j = 0.$
If $\lambda_j$ is real, then Lemma~\ref{lm:reR} implies $|o_j| \equiv 0$ or $1+n \pmod 4$. If $2-n < \deg \lambda_j< 2$, then $0 < |o_j| < n$.
So, the cohomological assumption on $L$ implies $[o_j] = 0 \in H^*(L;\R),$ and we choose $b_j$ such that $(-1)^{\deg\lambda_j}db_j = -o_j.$ For other possible values of $\deg \lambda_j,$ degree considerations imply $o_j = 0,$ so we choose $b_j =0.$

Lemma~\ref{lm:inductive} now guarantees that $b_{(l+1)}:=b_{(l)}+\sum_{\substack{\kappa_l+1\le j\le \kappa_{l+1}\\ \deg\lambda_j\ne 2}}\lambda_jb_j$
satisfies
\[
\mg(e^{b_{(l+1)}})\equiv c_{(l+1)}\cdot 1\pmod{F^{E_l}C},\qquad c_{(l+1)}\in (\mI_R)_2.
\]
Since $b_j = 0$ when $\lambda_j$ is not real and when $\deg \lambda_j = 2-n,1-n,$ it follows that $b_{(l+1)}$ is real and satisfies
\[
(db_{(l+1)})_n = (db_{(l)})_n=0, \qquad \int_Lb_{(l+1)}=\int_Lb_{(l)}=a.
\]

Thus, the inductive process gives rise to a convergent sequence $\{b_{(l)}\}_{l=0}^\infty$ where $b_{(l)}$ is real and bounding modulo $F^{E_l}C$.
Taking the limit as $l$ goes to infinity, we obtain
\[
b=\lim_l b_{(l)},\quad\deg_C b=1, \quad\int_L b = a, \quad \mg(e^b)= c\cdot 1,\quad c=\lim_lc_{(l)}\in (\mI_R)_2.
\]
\end{proof}

In order to prove the symplectic deformation invariance axiom~\eqref{ax:def} of Theorem~\ref{axioms}, we will use the following analog of Proposition~\ref{cor:spin_b_exist}, the proof of which is verbatim the same. Recall the notation $R^J, Q^J, \mI_{R^J}, \mI_{Q^J}, C^J, D^J,$ and $\m^{J,\gamma},$ introduced in Section~\ref{sssec:Jnov}.

\begin{prop}\label{cor:spin_b_existJ}
Suppose $\s$ is a spin structure, $n \not \equiv 1 \pmod 4,$ and $H^k(L;\R) = 0$ for $k \equiv 0,n+1\pmod 4, \, k \neq 0,n.$ Then for any real closed $\gamma \in (\mI_{Q^J}D^J)_2$ and any $a \in (\mI_{R^J})_{1-n}$, there exists a real bounding chain $b$ for $\m^{J,\gamma}$ such that
$
\int_L b = a.
$
\end{prop}

Proposition~\ref{cor:spin_b_exist} can be simplified in the case $n=2,3.$
\begin{prop}\label{prop:3d}
If $n =2$ or $3$, and $\gamma \in (\mI_QD)_2$ is real and closed, any form $b \in A^{n}(L)\otimes(\mI_R)_{1-n}$ is a real bounding chain for $\mg$.
\end{prop}
\begin{proof}
Since $\dim L \leq 3,$ it follows that $\s$ is a spin structure. Lemma~\ref{lm:reR} implies that $b$ is real. Since $\deg b = 1,$ it follows from Lemma~\ref{lm:reob} that $\mg(e^b)$ is real and from Proposition~\ref{deg_str_map} that $\deg \mg(e^b) = 2.$
So, Lemma~\ref{lm:reC} implies
\[
\mg(e^{b})=f\in A^0(L)\otimes (\mI_R)_2.
\]
By Propositions~\ref{cl:a_infty_m} and~\ref{cl:unit},
\begin{multline*}
d\mg(e^{b})=-\sum_{(k_1,\beta)\ne (1,\beta_0)}\m^{\gamma,\beta}_{k_1}(b^{\otimes(i-1)}\otimes \mg_{k_2}(b^{\otimes k_2})\otimes b^{\otimes(k_1-i)})=\\
=-\sum_{(k_1,\beta)\ne (1,\beta_0)}\m^{\gamma,\beta}_{k_1}(b^{\otimes(i-1)}\otimes f \otimes b^{\otimes(k_1-i)})
=(-1)^{\deg f}f\cdot(b-b)=0.
\end{multline*}
Therefore,
\[
\mg(e^{b})=c\cdot 1,\qquad c\in (\mI_R)_2.
\]
\end{proof}

The proofs of the following proposition and lemma are similar to those of Proposition~\ref{prop:sd} and Lemma~\ref{lm:reob} respectively.
\begin{prop}\label{prop:sdt}
Suppose the relative spin structure $\s$ on $L$ is in fact a spin structure, and $\gt \in \mathcal{I}_Q \mD$ is real. Then, for all $\at = (\at_1,\ldots,\at_k)$ with $\at_j \in \mC,$ we have
\[
\phi^* \mgt_k(\at_1,\ldots,\at_k) = (-1)^{1+k + s_\tau^{[1]}(\at)}\mgt_k(\phi^*\at_k,\ldots,\phi^*\at_1).
\]
\end{prop}

\begin{lm}\label{lm:reobt}
Suppose $\s$ is a spin structure and $\gt \in \mathcal{I}_Q \mD,\,\bt \in \mC,$ are real with $\deg \bt = 1.$ Then $\mgt_k(\bt^{\otimes k})$ is real.
\end{lm}

\begin{prop}\label{cor:spin_unique}
Suppose $\s$ is a spin structure, $n \not \equiv 1 \pmod 4,$ and $H^k(L;\R) = 0$ for $k \equiv 3,n\pmod 4, \, k \neq 0,n.$
Let $(\gamma,b)$ be a bounding pair with respect to $J$ and let $(\gamma',b')$ be a bounding pair with respect to $J'$, both real, such that $\varrho_\phi([\gamma,b])=\varrho_\phi([\gamma',b']).$
Then $(\gamma,b)\sim(\gamma',b')$.
\end{prop}

\begin{proof}
We construct a pseudoisotopy $(\mgt,\bt)$ from $(\mg,b)$ to $(\m^{\gamma'},b').$
Let $\{J_t\}_{t \in [0,1]}$ be a path from $J$ to $J'$ in $\J_\phi$.
Choose a real $\xi\in \mI_Q D$ with $\deg_D\xi=1$ such that
\[
\gamma'-\gamma=d\xi
\]
and define
\[
\gt:=\gamma+t(\gamma'-\gamma)+dt\wedge\xi\in \mD.
\]
Then,
\begin{gather*}
d\gt=dt\wedge\d_t(\gamma+t(\gamma'-\gamma))-dt\wedge d\xi=0,\\
j_0^*\gt=\gamma,\quad j_1^*\gt=\gamma',\quad i^*\gt=0.
\end{gather*}
Moreover, $\gt$ is real and $\deg_\mD \gt = 2.$

We now move to constructing $\bt$.
Write $G(b,b')$ in the form of a list as in~\eqref{eq:list}.
Since $\int_Lb'=\int_Lb,$ there exists $\eta\in \mI_RC$ such that $|\eta|=n-1$ and $d\eta=(b')_n-(b)_n.$
Write
\[
\bt_{(-1)}:=b+t(b'-b)+dt\wedge\eta\in\mC.
\]
Then
\begin{gather*}
(d\bt_{(-1)})_{n+1}=dt\wedge(b'-b)_n-dt\wedge d\eta=0,\\
j_0^*\bt_{(-1)}=b,\qquad j_1^*\bt_{(-1)}=b'.
\end{gather*}
Let $l\ge 0$. Assume by induction that we have constructed $\bt_{(l-1)}\in \mC$ with $\deg_{\mC}\bt=1,$ $G(\bt_{(l-1)})\subset G(b,b'),$ such that
\[
(d\bt_{(l-1)})_{n+1}=0,\qquad j_0^*\bt_{(l-1)}=b,\qquad j_1^*\bt_{(l-1)}=b',
\]
and if $l\ge 1$, then
\[
\mgt(e^{\bt_{(l-1)}})\equiv \ct_{(l-1)} \cdot 1\pmod{F^{E_{l-1}}\mC}, \qquad \ct_{(l-1)}\in(\mI_R\mR)_2, \qquad d\ct_{(l-1)} = 0.
\]
Define the obstruction chains $\ot_j$ by ~\eqref{eq:ojt_dfn}.
By Lemma~\ref{lm:ut_even} we have $\ot_j\in A^{2-\deg\lambda_j}(I\times L),$ by Lemma~\ref{lm:ojt_closed} we have $d\ot_j=0$, and by Lemma~\ref{lm:ut_relative} we have $\ot_j|_{\d(I\times L)}=0$ whenever $\deg \lambda_j \ne 2$.
To apply Lemma~\ref{lm:ut_exact}, it is necessary to choose forms $\bt_j\in A^{1-\deg\lambda_j}(I\times L)$ such that $(-1)^{\deg\lambda_j}d\bt_j=-\ot_j+\ct_j\, dt, \ct_j \in \R,$ for $j\in\{\kappa_{l-1}+1,\ldots,\kappa_{l}\}$ such that $\deg\lambda_j\ne 2.$
If $\lambda_j$ is not real, Lemma~\ref{lm:reobt} implies that $\ot_j = 0$, so we choose $\bt_j = 0.$ If $\deg \lambda_j = 1-n,$ Lemma~\ref{lm:deg_ut} implies that $\ot_j = 0,$ so again we choose $\bt_j = 0.$
If $\deg \lambda_j = 1,$ Lemma~\ref{lm:ob1} gives $\bt_j$ such that $-d \bt_j = -\ot_j+\ct_j \,dt,\,\ct_j \in \R,$ and $\bt_j|_{\d(I\times L)} = 0$.
If $\lambda_j$ is real, then Lemma~\ref{lm:reR} implies $|\ot_j| \equiv 0$ or $1+n \pmod 4$. If $1-n < \deg \lambda_j< 1$, then $1 < |\ot_j| < n+1$.
So, the cohomological assumption on $L$ and Lemma~\ref{B_l} imply that $[\ot_j] = 0 \in H^*(I\times L,\d(I\times L);\R).$ Thus, we choose $\bt_j$ such that $(-1)^{\deg\lambda_j}d\bt_j = -\ot_j$ and $\bt_j|_{\d(I\times L)} = 0.$ For other possible values of $\deg \lambda_j,$ degree considerations imply $\ot_j = 0,$ so we choose $\bt_j =0.$

Lemma~\ref{lm:ut_exact} now guarantees that
\[
\bt_{(l)}:=\bt_{(l-1)}+\sum_{\substack{\kappa_{l-1}+1\le j\le\kappa_l \\ \deg\lambda_j\ne 2}}\lambda_j\bt_j
\]
satisfies
\[
\mgt(e^{\bt_{(l)}})\equiv \ct_{(l)}\cdot 1\pmod{F^{E_l}\mC},\qquad \ct_{(l)}\in (\mI_R\mR)_2, \qquad d\ct_{(l)} = 0.
\]
Since $\bt_j = 0$ when $\lambda_j$ is not real and when $\deg \lambda_j = 1-n,$ it follows that $\bt_{(l)}$ is real and satisfies
\[
(d\bt_{(l)})_{n+1}= (d\bt_{(l-1)})_{n+1}=0.
\]
Since $\bt_j|_{\d(I\times L)}=0$ for all $j\in\{\kappa_l+1,\ldots,\kappa_{l+1}\}$ such that $\deg\lambda_j\ne 2$, we have
\[
j_0^*\bt_{(l)}=j_0^*\bt_{(l-1)}=b,\qquad j_1^*\bt_{(l)}=j_1^*\bt_{(l-1)} = b'.
\]
Taking the limit $\bt=\lim_l\bt_{(l)},$ we obtain a real bounding chain for $\mgt$ that satisfies $j_0^*\bt=b$ and $j_1^*\bt=b'$. So, $(\mgt,\bt)$ is a pseudoisotopy from $(\mg,b)$ to $(\m^{\gamma'},b').$
\end{proof}

We now move to the proof of the main results of this section.

\begin{proof}[Proof of Theorem~\ref{thm2}]
The cohomological conditions are equivalent to $H^i(L,\R) \simeq H^i(S^n,\R)$ for $i \equiv 0,3,n,n+1 \pmod 4$. Thus,
surjectivity is guaranteed by Proposition~\ref{cor:spin_b_exist} and injectivity is given by Proposition~\ref{cor:spin_unique}.
\end{proof}

\subsection{The case of Georgieva}\label{ssec:pnksomg}

This section explores the classification of bounding pairs in a real setting under different assumptions.
The results are used in Section~\ref{sssec:geosuperpot} to express Georgieva's invariants in terms of the superpotential.

Let $\sly^G = \sly_L^G = H_2(X,L;\Z)/\Im(\Id + \phi_*)$.
Let $\L^{G},Q^G,R^G,C^G,D^G,\mC^G,$ and $\mD^G,$ be defined similarly to $\L,Q,R,C,D,\mC$ and $\mD,$ but with the variables $\Tg,\tg_1,\ldots,\tg_{N^G},\sg,$ in place of $T,t_1,\ldots,t_N,s,$ and with the trivial $\phi^*$ action,
\begin{equation*}
\phi^*\Tg^\beta = \Tg^\beta, \qquad \phi^* \tg_i = \tg_i, \qquad \phi^* \sg = \sg.
\end{equation*}
Since the action of $\phi^*$ on $A^*(L)$ is also trivial, the real condition
$
\phi^* a = -a
$
for $a\in C^G$ implies $a=0$.
Let
\[
\varrho_\phi^G:\{\text{real bounding pairs}\}/\sim\;\;\lrarr\;(\mI_Q H^*(X;Q^G))_2^{-\phi^*}
\]
be given by
\[
\varrho_\phi^G(\gamma,0)=[\gamma].
\]
Note that $\varrho_\phi^G$ is well defined by Lemma~\ref{lm:homotopy}.
Then, we obtain the following variant of Theorem~\ref{thm1}.
\begin{prop}[Classification of bounding pairs -- the case of Georgieva]\label{prop:altrealspin}
Suppose $(X,L,\omega,\phi)$ is a real setting and
$\frac{\mu(\beta)}{2}+w_\s(\beta)\equiv 0 \pmod 2$ for all $\beta \in \sly^G.$
Then $\varrho_\phi^G$ is bijective.
\end{prop}
The proof is given at the end of the section based on the following lemmas.

\begin{lm}\label{prop:sdt_v}\label{lm:even}
Suppose $\frac{\mu(\beta)}{2}+w_\s(\beta)\equiv 0 \pmod 2$ for all $\beta \in \sly^G$, and $\gamma \in \mathcal{I}_{Q^G} D^G$ is real. Then, for all $\alpha = (\alpha_1,\ldots,\alpha_k)$ with $\alpha_j \in C^G,$ we have
\[
\phi^* \mg_k(\alpha_1,\ldots,\alpha_k) = (-1)^{1+k + s_\tau^{[1]}(\alpha)}\mg_k(\alpha_k,\ldots,\alpha_1).
\]
\end{lm}
The proof is similar to Proposition~\ref{prop:sd}.

\begin{cor}\label{lm:even_v}
Suppose $\frac{\mu(\beta)}{2}+w_\s(\beta)\equiv 0 \pmod 2$ for all $\beta \in \sly^G$ and $\gamma \in \mathcal{I}_{Q^G} D^G$ is real. Then $\mg_0 = 0$.
In particular, $(\gamma,0)$ is a bounding pair.
\end{cor}

\begin{lm}\label{prop:sdt_t_v}
Suppose $\frac{\mu(\beta)}{2}+w_\s(\beta)\equiv 0 \pmod 2$ for all $\beta \in \sly^G$ and $\gt \in \mathcal{I}_{Q^G} \mD^G$ is real. Then, for all $\at = (\at_1,\ldots,\at_k)$ with $\at_j \in \mC^G,$ we have
\[
\phi^* \mgt_k(\at_1,\ldots,\at_k) = (-1)^{1+k + s_\tau^{[1]}(\at)}\mgt_k(\at_k,\ldots,\at_1).
\]
\end{lm}
The proof is again similar to Proposition~\ref{prop:sd}.

\begin{cor}\label{lm:reobt_v}
Suppose $\frac{\mu(\beta)}{2}+w_\s(\beta)\equiv 0 \pmod 2$ for all $\beta \in \sly^G$ and $\gt \in \mathcal{I}_{Q^G} \mD^G$ is real. Then $\mgt_0 = 0$.
In particular, $(\gt,0)$ is a bounding pair.
\end{cor}

\begin{lm}\label{cor:spin_unique_v}
Suppose $\frac{\mu(\beta)}{2}+w_\s(\beta)\equiv 0 \pmod 2$ for all $\beta \in \sly^G$.
Let $(\gamma,b)$ be a bounding pair with respect to $J$ and let $(\gamma',b')$ be a bounding pair with respect to $J'$, both real, such that $\varrho_\phi^G([\gamma,0])=\varrho_\phi^G([\gamma',0]).$
Then $(\gamma,0)\sim(\gamma',0)$.
\end{lm}
The proof is a simplified version of Proposition~\ref{cor:spin_unique}, as follows.
\begin{proof}
Let $\{J_t\}_{t \in [0,1]}$ be a path from $J$ to $J'$ in $\J_\phi$.
Let $\xi\in (\mI_{Q^G}D^G)_1$
be real such that
\[
\gamma'-\gamma=d\xi
\]
and define
\[
\gt:=\gamma+t(\gamma'-\gamma)+dt\wedge\xi\in (\mI_{Q^G}\mD^G)_2.
\]
Then $\gt$ is real and satisfies
\begin{gather*}
d\gt=dt\wedge\d_t(\gamma+t(\gamma'-\gamma))-dt\wedge d\xi=0,\\
j_0^*\gt=\gamma,\quad j_1^*\gt=\gamma',\quad i^*\gt=0.
\end{gather*}
Thus, by Corollary~\ref{lm:reobt_v}, $(\mgt,0)$ is a pseudoisotopy from $(\mg,0)$ to $(\m^{\gamma'},0)$.
\end{proof}

\begin{proof}[Proof of Proposition~\ref{prop:altrealspin}]
For $\g\in (\mI_{Q^G}H^*(X;Q^G))_2^{-\phi^*}$, choose a real representative $\gamma$. Then, $\gamma|_L=0$, so $\gamma\in D^G$.
Surjectivity follows from Corollary~\ref{lm:even_v}. Injectivity follows from Lemma~\ref{cor:spin_unique_v}.
\end{proof}

\section{Superpotential}\label{sec:suppot}
Denote
\[
\L_c:=\left\{\sum_iT^{\beta_i}a_i\in \L\;\big|\;\beta_i\in \Im(H_2(X;\Z)\to \sly) \right\}.
\]
Define a map of $\Lambda_c[[t_0,\ldots,t_N]]$-modules
\[
\mathcal{D}:R\lrarr \L_c[[t_0,\ldots,t_N]]
\]
by $\mathcal{D}(\Theta)=\sum(\text{type-$\mathcal{D}$ summand in $\Theta$}).$ In other words,
\[
\mathcal{D}(\Theta)=\sum_{\substack{\lambda=T^\beta\prod_{j=1}^Nt_j^{r_j}\\ \beta\in\Im(H_2(X;\Z)\to \sly)}}[\lambda](\Theta)\cdot \lambda.
\]

Recall the definition of the superpotential:
\[
\Omega(s,t)=\Oh(s,t)-\mathcal{D}(\Oh),\qquad \Oh(s,t)=(-1)^n\big(\sum_{k\ge 0}\frac{1}{k+1}\langle\mg_k(b^k),b\rangle
+\m_{-1}^{\gamma}\big).
\]

\subsection{Invariance under pseudoisotopy}\label{subsec:invar}

\begin{proof}[Proof of Theorem~\ref{thm_inv}]
If $n = 0,$ we have $\sly=\{\beta_0\}$.
The energy zero property, Proposition~\ref{cl:qt_zero}, gives $\Omega(\gamma,b)=0,$ and the theorem is a tautology. So, we may assume $n > 0.$
Let $\gt$ and $\bt$ be as in the definition of gauge equivalence, Definition~\ref{dfn_g_equiv}. Recall from~\cite[Remark 2.4]{ST1} that Stokes' theorem for differential forms with coefficients in a graded ring comes with a sign,
\[
\int_M d\xi = (-1)^{\dim M+|\xi|+1}\int_{\d M} \xi.
\]
Using this, Lemma~\ref{lm:pseudo}, Lemma~\ref{lm:d_ll_gg},
and Proposition~\ref{prop:mgt_a_infty},
\begin{align*}
\Oh_{J'}&(\gamma',b')-\Oh_J(\gamma,b)=
(-1)^n\Big(\sum_{k\ge 0}\frac{1}{k+1}\left(\langle\m^{\gamma'}_k(b'^{\otimes k}),b'\rangle-\langle\mg_k(b^{\otimes k}),b\rangle\right)+ \m^{\gamma'}_{-1}-\mg_{-1}\Big)\\
=&(-1)^{n+n+1}\big(pt_*\sum_{k\ge 0}\frac{1}{k+1}d\ll\mgt_k(\bt^{\otimes k}),\bt\gg+pt_*d\mgt_{-1}\big)\\
=&pt_*\sum_{\substack{k_1+k_2=k+1\\k\ge 0, k_1\ge 1,k_2\ge 0}}\frac{k_1}{k+1}\cdot
\ll\mgt_{k_1}(\bt^{\otimes k_1}), \mgt_{k_2}(\bt^{\otimes k_2})\gg
+pt_*\frac{1}{2}\ll \mgt_0,\mgt_0\gg  -pt_*\widetilde{GW}.
\shortintertext{We may relax the limit of summation to $k_1 \geq 0$ since $k_1$ multiplies the summands.
Interchanging the summation indices $k_1$ and $k_2$ and using the symmetry of the pairing~\eqref{eq:pseudopair}, we continue}
=&pt_*\frac{1}{2}\sum_{\substack{k_1+k_2=k+1\\ k, k_1,k_2\ge 0}}\frac{k_1}{k+1}\cdot
\ll\mgt_{k_1}(\bt^{\otimes k_1}), \mgt_{k_2}(\bt^{\otimes k_2})\gg+\\
&+pt_*\frac{1}{2} \sum_{\substack{k_1+k_2=k+1\\k,k_1,k_2\ge 0}}\frac{k_2}{k+1}\cdot
\ll\mgt_{k_1}(\bt^{\otimes k_1}),\mgt_{k_2}(\bt^{\otimes k_2})\gg
+pt_*\frac{1}{2}\ll \mgt_0,\mgt_0\gg  -pt_*\widetilde{GW}\\
=&pt_*\frac{1}{2}\sum_{k_1,k_2\ge 0}\ll \mgt_{k_1}(\bt^{\otimes k_1}),\mgt_{k_2}(\bt^{\otimes k_2})\gg-
pt_*\widetilde{GW}.\\
\shortintertext{By the Maurer-Cartan equation~\eqref{eq:mct}, we continue}
=&pt_*\frac{1}{2}\ll \ct \cdot 1,\ct \cdot 1\gg
-pt_*\widetilde{GW}.\\
\shortintertext{Using $n > 0$ and the bilinearity of $\ll \cdot, \cdot \gg$ over $\mR$ from Proposition~\ref{cl:t_linear}, we conclude}
=&
- pt_*\widetilde{GW}.
\end{align*}
Since $\mathcal{D}(\widetilde{GW}) = \widetilde{GW},$ it follows that
\[
\Omega_{J'}(\gamma',b')-\Omega_J(\gamma,b)=\Oh_{J'}(\gamma',b')-\Oh_J(\gamma,b)
-\D\left(\Oh_{J'}(\gamma',b')-\Oh_J(\gamma,b)\right)=0.
\]
\end{proof}

\subsection{A coordinate free formulation}\label{ssec:cff}
In this section we modify Definition~\ref{def:ogw} to obtain open Gromov-Witten invariants that are manifestly independent of the choice of a basis. We show the modified definition coincides with the original one. This invariant language is also used extensively in~\cite{ST3}.

Let $W$ be a graded real vector space of dimension $N+1.$ Let $\R[[W]]$ denote the ring of formal functions on the completion of $W$ at the origin and let $m_W \subset \R[[W]]$ denote the unique maximal ideal. More explicitly, let $\{w_i\}_{i =0}^N$ be a homogeneous basis of $W,$ let $\{w_i^*\}_{i =0}^N$ be the dual basis of $W^*$, and take $t_i$ to be of degree $-|w_i|.$ Then we have an isomorphism $\R[[t_0,\ldots,t_N]] \overset{\sim}\to \R[[W]]$ taking $t_i$ to $w_i^*.$ Under this isomorphism, the ideal $\langle t_0,\ldots,t_N\rangle$ corresponds to the ideal $m_W.$ Since each tangent space of $W$ is canonically isomorphic to~$W,$ the $\R[[W]]$ module of formal vector fields on $W$ is canonically isomorphic to $W\otimes \R[[W]]$. Each formal vector field $w \in W \otimes\R[[W]]$ gives rise to a derivation $\partial_w : \R[[W]] \to \R[[W]].$  In coordinates, if $w = \sum_i f_i w_i $ with $f_i \in \R[[W]]$, then $\partial_w = \sum_i f_i \d_i.$ For $m \in \Z,$ let $W[m]$ denote the shift of $W$ by $m.$ That is, $W[m]$ is the graded vector space with $W[m]^i = W^{i+m}.$

Write
\begin{gather*}
R_W:=\Lambda[[s]]\otimes \R[[W[2]]],\\
Q_W:=\L_c \otimes \R[[W[2]]] \leq R_W.
\end{gather*}
Define ideals $\mI_R^W \triangleleft R_W$
and $\mI_Q^W \triangleleft Q_W$ by
\[
\mI_R^W = \L^+R_W + m_WR_W + \left<s\right> R_W,\qquad \mI_Q^W = \L_c^+Q_W + m_W Q_W.
\]
Denote by $\Gamma_W \in W\otimes Q_W$ the vector field corresponding to the parity operator $P:W \to W$ given by $P(w) = (-1)^{\deg w}w.$ In a basis as above, this means
\[
\Gamma_W = \sum_{i =0}^N  (-1)^{\deg w_i} w_i \otimes w_i^* = \sum_{i = 0}^N w_i^* \cdot w_i.
\]
Note that $\deg\Gamma_W = 2.$
Let
\[
C_W:=A^*(L)\otimes R_W,\quad\text{and}\quad D_W:= A^*(X,L)\otimes Q_W.
\]

Consider the isomorphism $\R[[W[2]]] \to \R[[t_0,\ldots,t_N]]$ given by $w_i^* \mapsto t_i.$ Tensoring with the identity map, we obtain isomorphisms
\[
\varphi_R : R_W \to R, \qquad \varphi_Q : Q_W \to Q, \qquad \varphi_C : C_W \to C, \qquad \varphi_D : D_W \to D.
\]
We denote these isomorphisms collectively by $\varphi.$ Consider a real setting as in Section~\ref{sssec:bno}. For $W$ concentrated in even degree, define an action of $\phi^*$ on $W[2]$ by $\phi^*(w) = (-1)^{\deg w/2}w.$ The action of $\phi^*$ on $W[2]$ induces actions on $R_W,Q_W,C_W,$ and $D_W,$ such that $\varphi$ commutes with $\phi^*.$

For $\gamma_W \in \mI_Q^W D_W,$ we define $A_\infty$ operations $\m_k^{\gamma_W} : C_W^{\otimes k} \to C_W$ and $\m_{-1}^{\gamma_W} \in R_W$ just as in Section~\ref{ssec:iainf}.
We define a bounding pair (chain) over $W$ by replacing $R,Q,\mI_R,\mI_Q,C,D,\mg_k,$ in Definition~\ref{dfn_bd_pair} with $R_W,Q_W,\mI_R^W,\mI_Q^W,C_W,D_W,\m^{\gamma_W}_k,$ respectively. Similarly, we define gauge equivalence over $W$ by modifying Definition~\ref{dfn_g_equiv}. We extend the classifying maps $\varrho$ (resp. $\varrho_\phi$) to bounding (resp. real bounding) pairs over $W$ by modifying the definition in Section~\ref{sssec:bno}. Then, $\varphi$ establishes a bijective correspondence between bounding pairs over $W$ and bounding pairs, and the correspondence respects gauge equivalence. Moreover, the classifying maps $\varrho,\varrho_\phi,$ commute with $\varphi.$ Thus, Theorem~\ref{thm1} (resp. Theorem~\ref{thm2}) extends to bounding (resp. real bounding) pairs over $W.$

For $(\gamma_W,b_W)$ a bounding pair over $W$, we define
\[
\Oh^W(\gamma_W,b_W)
=(-1)^{n}\big(\sum_{k\ge0}\frac{1}{(k+1)}\langle\m_k^{\gamma_W} (b_W^{\otimes k}),b_W\rangle
+\m_{-1}^{\gamma_W}\big).
\]
Let $D : R_W \to Q_W$ denote the unique $Q_W$ module homomorphism such that $D|_{Q_W} = \Id,$ $D(\left < s\right>) = 0$ and $D(T^\beta) = 0$ for $\beta \notin \im \varpi.$ Let $q : R_W \to R_W/Q_W$ denote the $Q_W$-module quotient map. Define
\[
\Omega^W(\gamma_W,b_W)\in R_W
\]
to be the unique element such that $q(\Omega^W(\gamma_W,b_W)) =  q(\Oh^W(\gamma_W,b_W))$ and $D(\Omega^W(\gamma_W,b_W)) =~0.$ Then
\begin{equation}\label{eq:vpsp}
\varphi_R(\Oh^W(\gamma_W,b_W)) = \Oh(\varphi_D(\gamma_W),\varphi_C(b_W)), \quad \varphi_R(\Omega^W(\gamma_W,b_W)) = \Omega(\varphi_D(\gamma_W),\varphi_C(b_W)).
\end{equation}
Thus, Theorem~\ref{thm_inv} extends to $\Omega^W.$

Take $W = W_L$ from Section~\ref{sssec:intro_OGW}.
By Theorem~\ref{thm1} (resp. Theorem~\ref{thm2}), choose a bounding pair $(\gamma_W,b_W)$ over $W$ such that
\begin{equation}\label{eq:cbpW}
\varrho([\gamma_W,b_W])=(\g_W,s) \qquad  \text{(resp. $\varrho_\phi([\gamma_W,b_W])=(\g_W,s)).$}
\end{equation}
By Theorem~\ref{thm_inv}, the superpotential $\Omega^W = \Omega^W(\gamma_W,b_W)$ is independent of the choice of $(\gamma_W,b_W).$

To define open Gromov-Witten invariants over $W,$ we proceed as follows.
An element $A\in W$ determines a constant coefficient formal vector field $A\otimes 1\in W[2]\otimes\R[[W[2]]]$. Thus, every $A\in W$ determines a derivation $\d_A: R_W \to R_W$.
Let
\[
z: \L[[W]] \lrarr \L
\]
be the extension of scalars of the homomorphism $\R[[W]] \to \R$ given by evaluation at the origin.
Define
\[
\ogw^W_{\beta,k}: W^{\otimes l} \lrarr \R
\]
by
\[
\ogw^W_{\beta,k}(A_1,\ldots,A_l)
=[T^\beta]\big(z(\d_{A_l}\cdots\d_{A_1}\d_s^k\Omega^{W})\big).
\]

Observe that nowhere in the definition of $\ogw^W_{\beta,k}$ have we chosen a basis of $W.$ On the other hand, recall the invariants
\[
\ogw_{\beta,k}: W^{\otimes l} \to \R
\]
of Definition~\ref{def:ogw}, which a priori depend on the choice of a basis $\{\Gamma_i\}_{i=0}^N$ of $W$.
\begin{lm}
We have
\[
\ogw_{\beta,k} = \ogw^W_{\beta,k}.
\]
\end{lm}

\begin{proof}
Recall the definition of $\Gamma$ from Section~\ref{sssec:intro_OGW}.
Take $w_i = \Gamma_i.$ Then $\varphi_D(\Gamma_W) = \Gamma.$ So, in the case of Theorem~\ref{thm1} (resp. Theorem~\ref{thm2}), it follows from equation~\eqref{eq:cbpW} that
\begin{equation*}
\varrho([\varphi_D(\gamma_W),\varphi_C(b_W)])=(\g,s) \qquad  \text{(resp. $\varrho_\phi([\varphi_D(\gamma_W),\varphi_C(b_W)])=(\g,s))$}
\end{equation*}
as in equation~\eqref{eq:cbp}. So, in the definition of $\ogw_{\beta,k},$ we may take $(\gamma,b) = (\varphi_D(\gamma_W),\varphi_C(b_W)).$ Abbreviating $\Omega = \Omega(\gamma,b),$ it follows from equation~\eqref{eq:vpsp} that
\begin{equation}\label{eq:vpsp2}
\varphi_R(\Omega^W) = \Omega.
\end{equation}
Observe also that $\varphi_R^*\d_{t_i} = \d_{w_i}.$

Since $\ogw^W_{\beta,k}$ and $\ogw_{\beta,k}$ are both linear, it suffices to verify equality on the basis $w_i.$ Indeed, by equation~\eqref{eq:vpsp2} we calculate,
\begin{multline*}
\ogw_{\beta,k}(w_{i_1},\ldots,w_{i_l}):= [T^{\beta}]\d_{t_{i_1}}\cdots\d_{t_{i_l}}\d_s^k\Omega|_{s=0,t_j=0}
=[T^\beta] \varphi_R(\d_{w_{i_1}}\cdots\d_{w_{i_l}}\d_s^k\Omega^W)|_{s=0,t_j=0} = \\
= [T^\beta]z\big((\d_{w_{i_1}}\cdots\d_{w_{i_l}}\d_s^k\Omega^W)\big)
= \ogw_{\beta,k}^W(w_{i_1},\ldots,w_{i_l}).
\end{multline*}
Thus, $\ogw^W_{\beta, k}=\ogw_{\beta,k}$.
\end{proof}

\begin{cor}\label{cor:indep}
The invariants $\ogw_{\beta,k}$ of Definition~\ref{def:ogw} are independent of the choice of a basis.
\end{cor}

\subsection{Properties of bounding chains}\label{ssec:bd_axioms}
In what follows, we restrict our attention to $\gamma\in \mathcal{I}_Q D$ of the form $\gamma=\sum_{j=0}^Nt_j\gamma_j$ with $\gamma_j \in A^*(X,L)$ as in Section~\ref{sssec:intro_OGW}.
Without loss of generality, assume $\gamma_0=1$.
In order to prove the axioms of Theorem~\ref{axioms} we first formulate parallel properties for bounding chains.

Let
$\Ups =C$ or $\Ups = R$. Decompose $\Ups$ as
$\Ups = \Ups'\otimes_{\R} R$ where $\Ups'=A^*(L)$ or $\Ups'=\R$, respectively.
We now define analogs of the Gromov-Witten unit and divisor axioms for elements of $\Upsilon$.
Fix a \sababa{} monoid $G=\{\pm\lambda_j\}_{j=1}^\infty$ indexed as in~\eqref{eq:list}.
Let $b\in \Ups$ such that $G(b)\subset G$. Write $b=\sum_{j=0}^\infty\lambda_j b_j$ with $b_j \in \Ups'$.
For $\lambda=\pm T^\beta s^k\prod_{m=0}^Nt_m^{l_m},$ define
\[
\gamma_i(\lambda):=\int_\beta\gamma_i,\qquad \ell_i(\lambda):=l_i,\qquad \forall i\in\{0,\ldots,N\}.
\]
In particular, $\gamma_i(\lambda) = 0$ if $\deg \gamma_i \neq 2.$ Each $j\in\{0,\ldots,N\}$ defines a map
\[
j:\Z_{\ge 0}\to \Z_{\ge 0}\cup \{ \diamond\}
\]
as follows.
For $r$ such that $\d_{t_j}\lambda_r\in G(b),$ define $j(r)$ by $\d_{t_j}\lambda_r=\ell_j(\lambda_r)\lambda_{j(r)}.$ Otherwise, set $j(r)=\diamond.$ Define also $\lambda_\diamond=1$ and $b_\diamond=0.$

\begin{dfn}\label{dfn:b_axioms}
With the preceding notation, we say $b \in \Upsilon$ satisfies the \textbf{unit} property modulo $F^{E_l}\Ups$ if
\begin{equation}\label{eq:b_unit}
\ell_0(\lambda_j)b_j=0,
\end{equation}
for all $j\le \kappa_l$.
We say $b$ satisfies the \textbf{divisor} property modulo $F^{E_l}\Ups$ if for any $i$ with $|\gamma_i|=2$,
\begin{equation}\label{eq:b_div}
	\ell_i(\lambda_j)b_j=\gamma_i(\lambda_j)\cdot b_{i(j)},
\end{equation}
for all $j\le \kappa_l$.
We say $b$ satisfies the unit (resp. divisor) property if it satisfies the unit (resp. divisor) property modulo $F^{E_l}\Ups$ for all $l$.
\end{dfn}

\begin{rem}
For motivation, consider the following partial reformulation of Definition~\ref{dfn:b_axioms}. The operation $\gamma_i$ defined above extends uniquely to a derivation $\gamma_i : \Upsilon \to \Upsilon$ over $R.$ An element $b \in \Upsilon$ satisfies the unit property if
\[
\partial_{t_0} b = 0,
\]
and it satisfies the divisor property if for any $i$ with $|\gamma_i| = 2,$
\[
(\partial_{t_i} - \gamma_i)b = 0.
\]
Following~\cite[Theorem 3]{ST1}, we can reformulate Proposition~\ref{q_fund} as
\[
[\partial_{t_0} ,\mg_k] = 0,
\]
and we can reformulate Proposition~\ref{cl:q_div} as
\[
[\partial_{t_i} - \gamma_i, \mg_k] = 0,
\]
for $i$ such that $|\gamma_i|= 2.$ Since both $\partial_{t_i}$ and $\gamma_i$ are derivations, if $b \in \mI_R C$ is a bounding chain for $\mg$ such that $b$ satisfies the unit (resp. divisor) property, then the superpotential $\Omega(\gamma,b)$ also satisfies the unit (resp. divisor) property. If $\Omega(\gamma,b)$ satisfies the unit (resp. divisor) property, the associated $\ogw$ invariants satisfy the unit (resp. divisor) axiom of Theorem~\ref{axioms}.

Definition~\ref{dfn:b_axioms} is designed to work smoothly with truncation modulo $F^{E_l}\Upsilon$ and thus with the obstruction theoretic arguments below. The index $\diamond$ allows a uniform formulation of the divisor property for $j \in \Z_{\geq 0}$ such that $\partial_{t_i} \lambda_j \notin G(b).$
\end{rem}

\begin{prop}\label{lm:d_t0b}\label{div}
Assume $H^{i}(L;\R)=0$ for $i\ne 0,n$, and let $b \in \mI_R C$ be a bounding chain for $\mg$ such that
$\int_Lb \in R$ satisfies the unit and divisor properties. Then there exists a bounding chain $b'$, gauge equivalent to $b$, that satisfies the unit and the divisor properties.
\end{prop}

\begin{prop}\label{div_real}
In a real setting, assume $L$ is spin, $\gamma$ is real, and $H^{i}(L;\R)=0$ for $i\equiv 3,n\pmod 4,$ $i\ne 0,n$. Let $b$ be a real bounding chain for $\mg$ such that
$\int_Lb \in R$ satisfies the unit and divisor properties. Then there exists a real bounding chain $b'$, gauge equivalent to $b$, that satisfies the unit and the divisor properties.
\end{prop}

The simultaneous proof of Propositions~\ref{div} and~\ref{div_real} is given at the end of this section. It uses induction on $l$ to construct $b'_{(l)}$, a bounding chain for $\mg$ modulo $F^{E_l}C$, together with $\bt_{(l)}$, a pseudoisotopy  modulo $F^{E_l}\mC$ from $b$ to $b'_{(l)}$. The proof is based on the following lemmas. We refer to the situation where $L$ is spin and belongs to a real setting as the real spin case.

Fix a \sababa{} monoid $G$ indexed as in~\eqref{eq:list}. Let $\gamma\in (\mI_QD)_2$ with $d\gamma=0$. Let $l\ge 0$. Assume we have $b'_{(l-1)}\in C$ with $\deg_{C} b'_{(l-1)}=1$, $G(b'_{(l-1)})\subset G$, and if $l\ge 1,$ then
\begin{equation}\label{eq:bd_mod_l}
\mg(e^{b'_{(l-1)}})\equiv c'_{(l-1)}\cdot 1\pmod{F^{E_{l-1}}C},\qquad c'_{(l-1)}\in(\mI_R)_2.
\end{equation}
Write $b'_{(l-1)}=\sum_{j=0}^{\kappa_{l-1}}\lambda_jb'_j+\u$ with $\u\equiv 0\pmod{F^{E_{l-1}}C}$.
For $j=\kappa_{l-1}+1,\ldots,\kappa_l,$ set
\begin{equation}\label{eq:o'j}
o'_j:=[\lambda_j](\mg(e^{b'_{(l-1)}})).
\end{equation}

\begin{lm}\label{lm:oj_axioms}
Assume $d\u=0$.
If $b'_{(l-1)}$ satisfies the divisor property modulo $F^{E_{l-1}}C$, then
\[
\ell_i(\lambda_j)o'_j=-(-1)^{\deg \lambda_j}\gamma_i(\lambda_j)db'_{i(j)}\qquad \forall j\in\{\kappa_{l-1}+1,\ldots,\kappa_l\},\quad\deg\lambda_j\ne 2.
\]
If $b'_{(l-1)}$ satisfies the unit property modulo $F^{E_{l-1}}C$, then
\[
\ell_0(\lambda_j)o'_j=0,\qquad \forall j\in\{\kappa_{l-1}+1,\ldots,\kappa_l\},\quad \lambda_j\ne t_0.
\]
In particular, $\ell_0(\lambda_j)o'_j=0$ if $\deg\lambda_j\ne 2$.
\end{lm}

\begin{proof}
For $\lambda,\lambda_{i_1},\ldots,\lambda_{i_m}\in R$ such that $\prod_{j=1}^m\lambda_{i_j}=\pm \lambda$, define $s^{\lambda}(\prod_{j=1}^m\lambda_{i_j})$ by $\prod_{j=1}^m\lambda_{i_j} =(-1)^{s^\lambda(\prod_{j=1}^m\lambda_{i_j})}\lambda$.
Write $\prod_{j=m}^1\lambda_{i_j}$ for the product taken in reverse order, that is, $\prod_{j=m}^1\lambda_{i_j}= \lambda_{i_m}\cdots\lambda_{i_1}$.
 Abbreviate
\begin{gather*}
s_1:=s^{\lambda_j}(T^\beta\prod_{m=k}^1\lambda_{z_m} \prod_{h=l}^{1}t_{i_h}),
\quad
s_2:=s^{\lambda_j}(T^\beta\prod_{m=k}^1\lambda_{z_m} \prod_{j=N}^{0}t_j^{r_j}),\\
s_1':=s_1+\sum_{m=1}^k\deg \lambda_{z_m},\qquad s_2':=s_2+\sum_{m=1}^k\deg \lambda_{z_m}.
\end{gather*}
Note that $s_2$ satisfies
$
s_2 =
s^{\lambda_{j}}(T^\beta\prod_{m=k}^1\lambda_{z_m}t_i^{r_i} \prod_{j\ne i}t_j^{r_j})
=
s^{\lambda_{i(j)}}(T^\beta\prod_{m=k}^1\lambda_{z_m}t_i^{r_i-1} \prod_{j\ne i}t_j^{r_j}).
$
By Propositions~\ref{lm:qlinear},~\ref{q_zero}, and~\ref{cl:symmetry}, compute
\begin{equation}\label{eq:decomp}
\begin{split}
\ell_i(\lambda_j&)\cdot o'_j
=
\ell_i(\lambda_j)\cdot\hspace{-3em}\sum_{\substack{k,l,\beta \ne \beta_0 \\T^\beta\prod_{m=k}^1\lambda_{z_m} \prod_{h=l}^{1}t_{i_h} =\\
\quad = (-1)^{s_1} \lambda_j}} \hspace{-2em}\frac{1}{l!}
(-1)^{s_1'}
\qkl^{\beta}(\otimes_{m=1}^kb'_{z_m};\otimes_{h=1}^{l}\gamma_{i_h})
+\ell_i(\lambda_j)[\lambda_j](d\u)\\
=&\hspace{-1.5em}
\sum_{\substack{k,l,\beta\ne \beta_0  \\ \sum_{m=0}^N r_m=l\\ T^\beta\prod_{m=k}^1\lambda_{z_m}\prod_{j=N}^0t_j^{r_j} = \\ \quad=(-1)^{s_2}\lambda_j}} \hspace{-2.5em}	 (-1)^{s_2'} (r_i+\sum_{a=1}^k\ell_i(\lambda_{z_a}))
\cdot\prod_{m=1}^N\frac{1}{r_m!}
\qkl^{\beta}(\otimes_{m=1}^kb'_{z_m};\otimes_{m=0}^{N}\gamma_{m}^{\otimes r_m})\\
=&\sum_{\substack{k,l,\beta\ne \beta_0  \\ \sum_{m=0}^N r_m=l \\T^\beta\prod_{m=k}^1\lambda_{z_m}\prod_{j=N}^{0}t_j^{r_j} =\\
\quad =(-1)^{s_2} \lambda_j}}
 \hspace{-2.5em}
 (-1)^{s_2'}
	 r_i\cdot\prod_{m=1}^N\frac{1}{r_m!}
\qkl^{\beta}(\otimes_{m=1}^kb'_{z_m};\otimes_{m=0}^{N}\gamma_{m}^{\otimes r_m})+\\
&+\hspace{-1.5em}\sum_{\substack{k,l,\beta\ne \beta_0  \\ \sum_{m=0}^N r_m=l \\T^\beta\prod_{m=k}^1\lambda_{z_m}\prod_{j=N}^{0}t_j^{r_j} =\\ \quad =(-1)^{s_2} \lambda_j}}\hspace{-2.5em}
(-1)^{s_2'} \cdot
\sum_{a=1}^k
	\prod_{m=1}^N\frac{1}{r_m!}
\qkl^{\beta}(\otimes_{m=1}^{a-1}b'_{z_m}\otimes \ell_i(\lambda_{z_a})b'_{z_a} \otimes \otimes_{m=a+1}^kb'_{z_m} ;\otimes_{m=0}^{N}\gamma_{m}^{\otimes r_m}).
\end{split}
\end{equation}
For $i$ such that $|\gamma_i|=2,$ apply the divisor property, Proposition~\ref{cl:q_div}, to the first summand  and the assumption on $b_j$ to the second summand to obtain
\begin{align*}
\ell_i(&\lambda_j)\cdot o'_j
=\hspace{-3.5em}
\sum_{\substack{k,l,\beta\ne \beta_0  \\ \sum_{m=1}^N r_m=l \\T^\beta\prod_{m=k}^1\lambda_{z_m}\prod_{j=N}^{0}t_j^{r_j}
 =\\
\quad = (-1)^{s_2}
\lambda_j}}\hspace{-.5em}
(-1)^{s_2'}
\gamma_i(T^\beta)
\frac{1}{(r_i-1)!}\prod_{m\ne i}\frac{1}{r_m!}
\q_{k,l-1}^{\beta}(\otimes_{m=1}^kb'_{z_m}; \gamma_i^{\otimes(r_i-1)}\otimes\otimes_{m\ne i}\gamma_{m}^{\otimes r_m} )+\\
&+
\hspace{-2em}\sum_{\substack{k,l,\beta\ne \beta_0  \\ \sum_{m=1}^N r_m=l \\T^\beta\prod_{m=k}^1\lambda_{z_m}\prod_{j=N}^{0}t_j^{r_j} =\\
\quad=(-1)^{s_2}\lambda_j}}\hspace{-2.5em}\sum_{a=1}^k
(-1)^{s_2'}
\gamma_i(\lambda_{z_a})
	\prod_{m=0}^N\frac{1}{r_m!}
\qkl^{\beta}(\otimes_{m=1}^{a-1}b'_{z_m}\otimes b'_{i(z_a)} \otimes \otimes_{m=a+1}^kb'_{z_m} ;\otimes_{m=0}
^N\gamma_{m}^{\otimes r_m})\\
=&
\sum_{\substack{k,l,\beta\ne \beta_0 \\ \sum_{m=0}^N r_m=l \\T^\beta\prod_{m=k}^1\lambda_{z_m}\prod_{j=N}^{0}t_j^{r_j}
 =(-1)^{s_2}\lambda_{i(j)}}}\hspace{-2.5em}
 (-1)^{s_2'}
 \gamma_i(T^\beta)
\prod_{m=0}^N\frac{1}{r_m!}
\qkl^{\beta}(\otimes_{m=1}^kb'_{z_m};\otimes_{m=0}^{N}\gamma_{m}^{\otimes r_m})+\\
&+
\hspace{-2em}\sum_{\substack{k,l,\beta\ne \beta_0  \\ \sum_{m=0}^N r_m=l \\T^\beta\prod_{m=k}^1\lambda_{z_m}\prod_{j= N}^{0}t_j^{r_j}
 =(-1)^{s_2}\lambda_{i(j)}}}\hspace{-2.5em}
 (-1)^{s_2'}
 \sum_{a=1}^k \gamma_i(\lambda_{z_a})
	\prod_{m=0}^N\frac{1}{r_m!}
\qkl^{\beta}(\otimes_{m=1}^{k}b'_{z_m} ;\otimes_{m=0}^{N}\gamma_{m}^{\otimes r_m})\\
=&
\sum_{\substack{k,l,\beta \ne \beta_0 \\ \sum_{m=0}^N r_m=l \\T^\beta\prod_{m=k}^1\lambda_{z_m}\prod_{j=N}^{0}t_j^{r_j}
 =(-1)^{s_2}\lambda_{i(j)}}}\hspace{-2.5em}
 (-1)^{s_2'}
 \gamma_i(T^\beta\prod_{a=1}^k\lambda_{z_a})
\prod_{m=0}^N\frac{1}{r_m!}
\qkl^{\beta}(\otimes_{m=1}^kb'_{z_m};\otimes_{m=0}^{N}\gamma_{m}^{\otimes r_m})\\
=&
\gamma_i(\lambda_j)\cdot[\lambda_{i(j)}](\mg(e^{b'_{(l-1)}})-\m^{\gamma,\beta_0}(e^{b'_{(l-1)}})).
\end{align*}
By equation~\eqref{eq:bd_mod_l}, Proposition~\ref{q_zero}, and the assumption $2 \neq \deg \lambda_j = \deg \lambda_{i(j)},$ the last expression equals
$-(-1)^{\deg \lambda_j}\gamma_i(\lambda_j)\cdot db_{i(j)},$
as required.

Now consider equation~\eqref{eq:decomp} with $i=0.$
By the induction hypothesis the last sum vanishes. Proposition~\ref{q_fund} then gives
\[
\ell_0(\lambda_j)o'_j
=\q^{\beta_0}_{0,1}(1)\cdot \delta=-1\cdot \delta,
\]
where $\delta=1$ if $\lambda_j=t_0$ and $\delta=0$ otherwise.
In particular, $\ell_0(\lambda_j)o'_j= 0$ unless $\deg\lambda_j=2$.
\end{proof}

\begin{lm}\label{lm:b'j}
Assume $\u \in A^n(L)\otimes R$ and $b'_{(l-1)}$ satisfies the unit and divisor properties modulo $F^{E_{l-1}}C$. In addition, assume $(b'_{(l-1)})_n$ satisfies the unit and divisor properties.
 For $j\in\{\kappa_{l-1}+1,\ldots,\kappa_l\}$ such that $\deg\lambda_j\ne 2,$ suppose $[o'_j]=0$.
Then there exist $b'_j\in A^{1-\deg\lambda_j}(L)$ such that $(-1)^{\deg\lambda_j}db'_j=-o'_j$ and $b'_{(l)}:=b'_{(l-1)}+\sum_{\substack{\kappa_{l-1}+1\le j\le \kappa_l\\ \deg\lambda_j\ne 2}}\lambda_jb'_j$ satisfies the unit and divisor properties modulo $F^{E_l}C$. Moreover, we can take $b'_j = 0$ when $\deg \lambda_j = 1-n$, so $\big(b'_{(l)}\big)_n=\big(b'_{(l-1)}\big)_n$.
In the real spin case, if $b'_{(l-1)}$ is real, then we can take $b'_j = 0$ when $\lambda_j$ is not real,  so $b'_{(l)}$ is real.
\end{lm}

\begin{proof}
Let $j\in\{\kappa_{l-1}+1,\ldots,\kappa_l\}.$  If $\deg \lambda_j > 2$ or $\deg \lambda_j < 2-n,$ then for degree reasons $o'_j = 0,$ so we choose $b'_j = 0.$ The case $2-n \leq \deg \lambda_j < 2$ is treated in several steps as follows.

If $\ell_0(\lambda_j) \neq 0,$ then Lemma~\ref{lm:oj_axioms} gives $o'_j = 0.$ So, we choose $b'_j = 0,$ and thus, $\ell_0(\lambda_j) b'_j= 0.$ For all $i$ with $|\gamma_i| = 2,$ we claim that $\ell_i(\lambda_j)b'_j=\gamma_i(\lambda_j)\cdot b'_{i(j)}.$ Indeed, $\ell_0(\lambda_{i(j)}) \neq 0$ and $\nu(\lambda_{i(j)}) \leq E_{l-1},$ so the assumption that $b'_{(l-1)}$ satisfies the unit and divisor properties modulo $F^{E_{l-1}}C$ implies that $b'_{i(j)} = 0.$

If $\ell_0(\lambda_j) = 0$ but there exists $i_0$ with $|\gamma_{i_0}| = 2$ such that $\ell_{i_0}(\lambda_j)>0$, then Lemma~\ref{lm:oj_axioms} implies
\[
o'_j=-(-1)^{\deg \lambda_j}
\frac{\gamma_{i_0}(\lambda_{j})}{\ell_{i_0}(\lambda_j)}\cdot db'_{i_0(j)}.
\]
Take
\[
b'_j:=
\frac{1}{\ell_{i_0}(\lambda_j)}\gamma_{i_0}(\lambda_{j})
\frac{\gamma_{i_0}(\lambda_{j})}{\ell_{i_0}(\lambda_j)} \cdot
b'_{i_0(j)}.
\]
Clearly, $\ell_0(\lambda_j)b'_j = 0.$ We claim that $\ell_i(\lambda_j)b'_j=\gamma_i(\lambda_j)\cdot b'_{i(j)}$ for all $i$ with $|\gamma_i|=2$. Indeed, for $i = i_0,$ this is immediate. For $i \neq i_0,$ since $\nu(\lambda_{i_0(j)}) \leq E_{l-1}$ and we have assumed $b'_{(l-1)}$ satisfies the divisor property modulo $F^{E_{l-1}}C$, it follows that
\begin{multline*}
\ell_i(\lambda_j)b_j' = \frac{\gamma_{i_0}(\lambda_j) }{\ell_{i_0}(\lambda_j)}\ell_i(\lambda_j) b'_{i_0(j)} = \frac{\gamma_{i_0}(\lambda_j )}{\ell_{i_0}(\lambda_j)}\ell_i(\lambda_{i_0(j)}) b'_{i_0(j)} = \\
= \frac{\gamma_{i_0}(\lambda_j) }{\ell_{i_0}(\lambda_j)}\gamma_i(\lambda_{i_0(j)}) b'_{i(i_0(j))}
= \gamma_i(\lambda_{i_0(j)}) \frac{\gamma_{i_0}(\lambda_{i(j)}) }{\ell_{i_0}(\lambda_{i(j)})} b'_{i_0(i(j))} =
\gamma_{i}(\lambda_{i_0(j)})b'_{i(j)} =
\gamma_{i}(\lambda_{j})b'_{i(j)}.
\end{multline*}
In the real spin case, if $\lambda_j$ is not real, then $b'_j =0.$ Indeed, since $|\gamma_i| = 2,$ it follows that $\lambda_{i(j)}$ is also not real. So, the assumption that $b'_{(l-1)}$ is real implies that $b'_{i(j)} = 0.$

It remains to consider the case that for all $i$ with $|\gamma_i| = 2$ and for $i = 0$ we have $\ell_i(\lambda_j)=0.$ Let $b'_j$ be any form that satisfies $-(-1)^{\deg \lambda_j}db'_j=o'_j.$ Clearly, $\ell_0(\lambda_j)b'_j = 0.$ Also, for all $i$ with $|\gamma_i| = 2,$ we have $i(j) = \diamond,$ so both sides of the equation $\ell_i(\lambda_j)b'_j=\gamma_i(\lambda_j)\cdot b'_{i(j)}$ vanish. In the real spin case, if $\lambda_j$ is not real, Lemma~\ref{lm:reob} guarantees that $o'_j = 0,$ so we choose $b'_j = 0.$

Given the above choices of $b'_j,$ we claim that $b'_{(l)}$ satisfies the unit and divisor properties modulo $F^{E_l}C$ and is real in the real spin case. Indeed, it suffices to prove this claim for $\big(b'_{(l)}\big)_m = 0$ with $m = 0,\ldots,n.$ Since $b'_j = 0$ for all $j\in\{\kappa_{l-1}+1,\ldots,\kappa_l\}$ such that $\deg \lambda_j = 1-n,$ it follows that $\big(b'_{(l)}\big)_n=\big(b'_{(l-1)}\big)_n.$ So $\big(b'_{(l)}\big)_n$ satisfies the unit and divisor properties and in the real spin case is real. On the other hand, for $m \neq n$ since $(\u)_m = 0,$ we have
$\big(b'_{(l)}\big)_m  = \big(\sum_{\substack{0\le j\le \kappa_l\\ \deg\lambda_j\ne 2}}\lambda_jb'_j\big)_m.$
For each $j$ with $2-n \leq \deg \lambda_j < 2,$ we have shown that $\ell_0(\lambda_j)b'_j = 0$ and $\ell_i(\lambda_j)b'_j=\gamma_i(\lambda_j)\cdot b'_{i(j)}$ for all $i$ such that $|\gamma_i| = 2.$ Thus $\big(b'_{(l)}\big)_m$ satisfies the unit and divisor properties. In the real spin case, we have also shown that $b'_j = 0$ whenever $\lambda_j$ is not real. So, $\big(b'_{(l)}\big)_m$ is real.
\end{proof}

Continue with $b'_{(l-1)}$ and $o'_j$ as above. Suppose we are given $b'_j$ such that $(-1)^{\deg\lambda_j}db'_j=-o'_j$ for $j\in\{\kappa_{l-1}+1,\ldots,\kappa_l\}$ with $\deg\lambda_j\ne 2$, and $b'_j = 0$ for $j$ with $\deg \lambda_j = 1-n.$ In the real spin case, assume $b'_j=0$ for all $j$ such that $\lambda_j$ is not real.
Let~$b$ be a bounding chain for $\mg$ with $G(b) \subset G$. Recall that we write $\pi:I\times X\to X$ for the projection, and let $\tilde \gamma = \pi^*\gamma \in (\mI_Q\mD)_2$. Finally, suppose we are given $\bt_{(l-1)}\in \mC$ such that $\deg_{\mC}\bt_{(l-1)}=1, G(\bt_{(l-1)}) \subset G,$ and
\begin{equation*}
(d\bt_{(l-1)})_{n+1} = 0, \qquad j_0^*\bt_{(l-1)}=b,\qquad j_1^*\bt_{(l-1)}= b'_{(l-1)}.
\end{equation*}
If $l \geq 1,$ assume
\begin{equation*}
\mgt(e^{\bt_{(l-1)}})\equiv \tilde{c}_{(l-1)}\cdot 1\pmod{F^{E_{l-1}}\mC},
\qquad \tilde{c}_{(l-1)}\in(\mI_R\mR)_2,\quad d\ct =0.
\end{equation*}
Let $\ot_j$ be the obstruction chains for $\bt_{(l-1)}$ defined by~\eqref{eq:ojt_dfn}.
In the real spin case, suppose also that $b,b'_{(l-1)},$ and $\bt_{(l-1)}$ are real.

\begin{lm}\label{lm:bd_cond}
Assume $H^{i}(L;\R)=0$ for $i\ne 0,n$. Alternatively, in the real spin case, assume $H^{i}(L;\R)=0$ for $i\equiv 3,n\pmod 4,$ $i\ne 0,n$.
Then, for $j\in\{\kappa_{l-1}+1,\ldots,\kappa_l\}$ with $\deg\lambda_j\ne 2$, there exist
$\bt_j\in A^{1-\deg\lambda_j}(I\times L)$
such that
\[
(-1)^{\deg\lambda_j}d\bt_j=-\ot_j + \ct_j dt,\qquad \ct_j \in \R, \qquad j_0^*\bt_j=0,\qquad j_1^*\bt_j=b'_j.
\]
Moreover, we can take $\bt_j = 0$ when $\deg \lambda_j = 1-n.$
In the real spin case, we can take $\bt_j=0$ for all $j$ such that $\lambda_j$ is not real.
\end{lm}

\begin{proof}
For $j\in\{\kappa_{l-1}+1,\ldots,\kappa_l\}$ with $\deg\lambda_j\ne 2,$ define $O_j\in A^{2-\deg\lambda_j}(I\times L)$ by
\[
O_j:=to'_j-(-1)^{\deg\lambda_j}dt\wedge b'_j.
\]
Then
\[
dO_j=dt\wedge o'_j+tdo'_j+(-1)^{\deg \lambda_j}dt\wedge db'_j=0,
\]
and
\begin{equation}\label{eq:j0O}
j_0^*O_j=0,\qquad j_1^*O_j=o'_j.
\end{equation}
For $j$ such that $\deg \lambda_j = 1-n,$ we claim that $O_j = 0.$ Indeed, $o'_j = 0$ for degree reasons and $b'_j = 0$ by assumption. In the real spin case, for $j$ such that $\lambda_j$ is not real, we claim that $O_j = 0.$ Indeed, $b'_j = 0$ by assumption, and $o'_j = 0$ by Lemma~\ref{lm:reob} together with the assumption that $b'_{(l-1)}$ is real.
Set
\[
\oh_j:=\ot_j-O_j.
\]
By Lemma~\ref{lm:ojt_closed},
\[
d\oh_j=d\ot_j-dO_j=0.
\]
Lemma~\ref{rem:pseudo}, Lemma~\ref{lm:ut_relative} and equation~\eqref{eq:j0O} imply
\[
\oh_j|_{\d(I\times L)}=0.
\]
In the real spin case, if $\lambda_j$ is not real, we claim that
$\hat{o}_j=0.$ Indeed, in this case
we showed above that $O_j = 0,$ and Lemma~\ref{lm:reobt} together with the assumption that $\bt_{(l-1)}$ is real gives $\ot_j = 0$.
So, we take $\hat{b}_j=0$. If $\lambda_j$ is real, then Lemma~\ref{lm:reR} implies $|\hat{o}_j| \equiv 0$ or $1+n \pmod 4$.
If $\deg \lambda_j = 1-n,$
we claim that $\oh_j = 0.$ Indeed, in this case
we showed above that $O_j = 0,$ and Lemma~\ref{lm:deg_ut} implies that $\ot_j = 0.$ So,
we choose $\bh_j = 0.$ If $\deg \lambda_j = 1,$ Lemma~\ref{lm:ob1} gives $\bh_j$ such that $-d \bh_j = -\oh_j+\ct_j \,dt,\,\ct_j \in \R,$ and $\bh_j|_{\d(I\times L)} = 0$. If $1-n < \deg\lambda_j < 1$, then $1 < |\ot_j| < n+1$.
So, the cohomological assumption on $L$ and Lemma~\ref{B_l} guarantee the existence of $\bh_j\in A^{1-\deg\lambda_j}(I\times L)$
such that $(-1)^{\deg \lambda_j}d\bh_j=-\oh_j$ and $\bh_j|_{\d(I\times L)}=0.$ Take
$
\bt_j:=\bh_j+tb'_j.
$
Then
\begin{gather*}
j_0^*\bt_j=0,\qquad j_1^*\bt_j=b'_j,\\
(-1)^{\deg \lambda_j}d\bt_j
=-\oh_j + \ct_j dt -to'_j+(-1)^{\deg \lambda_j}dt\wedge b'_j
=-\oh_j-O_j +\ct_j dt
=-\ot_j+\ct_j dt.
\end{gather*}
For $j$ such that $\deg \lambda_j = 1-n,$ we chose $\bh_j = 0$ and assumed $b'_j = 0,$ so $\bt_j = 0.$ In the real spin case, for $j$ such that $\lambda_j$ is not real, we chose $\bh_j = 0$ and assumed $b'_j = 0$, so again, $\bt_j = 0.$
\end{proof}

\begin{proof}[Proof of Propositions~\ref{div}-~\ref{div_real}]
Take $\gt:=\pi^*\gamma$. Write $G(b)$ in the form of a list as in~\eqref{eq:list}.
Let $\bar{b}\in A^n(L)$ be any representative of the Poincar\'e dual of a point in $L$, and write $a=\int_Lb.$ Set
\[
b'_{(-1)}:=b'_{-1}:=a\cdot \bar{b}.
\]
Then  $\deg_Cb'_{(-1)}=1$, and $db'_{(-1)}=0$. It follows from the assumption on $\int_L b$ that $b'_{(-1)}$ satisfies the unit and divisor properties and that in the real spin case $b'_{(-1)}$ is real.

Choose $\eta\in \mI_R\cdot (A^{n-1}(L)\otimes R)$ with $G(\eta)\subset G(b)$, real in the real spin case, such that
\[
d\eta=b'_{(-1)}-(b)_n
\]
and define
\[
\bt_{(-1)}:=b+t(b'_{(-1)}-b)+dt\wedge\eta.
\]
Then
\begin{gather*}
(d\bt_{(-1)})_{n+1}=dt\wedge(b'_{(-1)}-b)_n-dt\wedge d\eta=0,\\
j_0^*\bt_{(-1)}=b,\qquad j_1^*\bt_{(-1)}=b'_{(-1)},
\end{gather*}
and in the real spin case $\bt_{(-1)}$ is real.

Proceed by induction. Let $l \geq 0.$ Assume we have constructed $b'_{(l-1)}\in C$, real in the real spin case, satisfying the unit and divisor properties modulo $F^{E_{l-1}}C$,
such that
\[
\deg_Cb'_{(l-1)}=1,\qquad G(b'_{(l-1)}) \subset G(b), \qquad \big( b'_{(l-1)} \big)_n=(b'_{(-1)})_n,
\]
and if $l\geq 1,$
\[
\mg(e^{b'_{(l-1)}})\equiv c'_{(l-1)}\cdot 1\pmod{F^{E_{l-1}}C}, \qquad c'_{(l-1)}\in(\mI_R)_2.
\]
Additionally, assume we have constructed $\bt_{(l-1)}$, real in the real spin case, with $\deg_C\bt_{(l-1)}=1,$ $G(\bt_{(l-1)}) \subset G(b),$ such that
\[
(d\bt_{(l-1)})_{n+1}=0,\qquad j_0^*\bt_{(l-1)}=b,\qquad j_1^*\bt_{(l-1)} = b'_{(l-1)},
\]
and if $l\geq 1,$
\[
\mgt(e^{\bt_{(l-1)}})\equiv \ct_{(l-1)}\cdot 1 \pmod{F^{E_{l-1}}\mC}, \qquad \ct_{(l-1)}\in(\mI_R\mR)_2,\quad d\ct=0.
\]

For $j\in\{\kappa_{l-1}+1,\ldots,\kappa_l\}$, define the obstruction chains $o'_j$ and $\ot_j$ as in~\eqref{eq:o'j} and~\eqref{eq:ojt_dfn} respectively.
For $j$ such that $\deg\lambda_j\ne 2,$ apply Lemma~\ref{lm:ojt_closed}, Lemma~\ref{lm:ut_relative}, and Lemma~\ref{rem:pseudo}, to deduce
\[
d\ot_j=0,\qquad j_0^*\ot_j=0,\qquad j_1^*\ot_j=o'_j.
\]
So, Lemma~\ref{lm:homotopy} implies $[o'_j]=0.$
By Lemma~\ref{lm:b'j} there exist $b'_j\in  A^{1-\deg\lambda_j}(L)$ such that
\[
b'_{(l)}:=b'_{(l-1)}+\sum_{\substack{\kappa_{l-1}+1\le j\le \kappa_l\\ \deg\lambda_j\ne 2}}\lambda_jb'_j
\]
satisfies the unit and divisor properties modulo $F^{E_l}C$,
is real in the real spin case,
and
\[
(-1)^{\deg\lambda_j}db'_j=-o'_j,\qquad \big(b'_{(l)}\big)_n=(b'_{(-1)})_n,\qquad  G(b'_{(l)})\subset G(b).
\]
By Lemma~\ref{lm:inductive}, we have
\[
\mg(e^{b'_{(l)}})\equiv c'_{(l)}\cdot 1\pmod{F^{E_l}C},\quad c'_{(l)}\in (\mI_R)_2.
\]
By Lemma~\ref{lm:bd_cond}, there exist $\bt_j\in A^{1-\deg\lambda_j}(I\times L)$ such that
\[
(-1)^{\deg\lambda_j}d\bt_j=-\ot_j + \ct_j dt,\qquad j_0^*\bt_j=0,\qquad j_1^*\bt_j=b'_j.
\]
Moreover, $\bt_j = 0$ when $\deg \lambda_j = 1-n$, and in the real spin case, $\bt_j=0$ unless $\lambda_j$ is real.
Thus, $\bt_{(l)}:=\bt_{(l-1)}+\sum_{\substack{\kappa_{l-1}+1\le j\le \kappa_l\\ \deg\lambda_j\ne 2}}\lambda_j\bt_j$ satisfies
\[
(d\bt_{(l)})_{n+1}= (d\bt_{(l-1)})_{n+1}=0,\qquad j_0^*\bt_{(l)}=b,\qquad j_1^*\bt_{(l)}=b'_{(l)}, \qquad  G(\bt_{(l)})\subset G(b),
\]
is real in the real spin case,
and by Lemma~\ref{lm:ut_exact},
\[
\mgt(e^{\bt_{(l)}})\equiv \ct_{(l)}\cdot 1\pmod{F^{E_l}\mC},\quad \ct_{(l)}\in(\mI_R\mR)_2,\quad d\ct=0.
\]

The limit chain $b'=\lim_lb'_{(l)}$ satisfies the unit and divisor properties and is real in the real spin case. Furthermore, $(\mgt,\bt=\lim_l\bt_{(l)})$ is a pseudoisotopy from $(\gamma,b)$ to $(\gamma,b')$.
\end{proof}

\subsection{Axioms of OGW}\label{ssec:axioms}
Recall that
\[
\ogw_{\beta,k}(\g_{i_1},\ldots,\g_{i_l})=[T^{\beta}](
\d_{t_{i_1}}\cdots\d_{t_{i_l}}\d_s^k\Omega|_{s=t_j=0}).
\]
In order to prove Theorem~\ref{axioms}, we need the following lemmas.

\begin{lm}\label{lm:d_tjD}
For all $j\in\{0,\ldots,N\},$
\[
\d_{t_j}\circ\mathcal{D}=\mathcal{D}\circ\d_{t_j}.
\]
For any $\Theta\in R$,
\[
\D([\lambda](\Theta)\cdot \lambda)= [\lambda](\D(\Theta))\cdot \lambda.
\]
\end{lm}
\begin{proof}
The first statement holds because
\begin{align*}
\d_{t_j}\D(\Theta)=\sum_{\substack{\lambda=T^\beta\prod_{j=1}^Nt_j^{r_j}\\ \beta\in\Im(H_2(X;\Z)\to \sly)}} [\lambda](\Theta)\cdot \d_{t_j}\lambda
=\D(\d_{t_j}\Theta).
\end{align*}
The second holds because, if $\lambda=T^\beta\prod_{j=1}^Nt_j^{r_j}$ for $\beta\in\Im(H_2(X;\Z)\to \sly)$, then
\begin{align*}
\mathcal{D}([\lambda](\Theta)\cdot \lambda)
=[\lambda](\Theta)\cdot \lambda
=[\lambda](\D(\Theta))\cdot \lambda .
\end{align*}
Otherwise both sides of the equation vanish.
\end{proof}

\begin{lm}\label{homogeneity}
The superpotential is homogeneous of degree $3-n$.
\end{lm}
\begin{proof}
First consider summands of the form $\frac{1}{k+1}\langle\m_k^\gamma(b^{\otimes k}),b\rangle$. It follows from Proposition~\ref{deg_str_map} that $\deg_C\m_k^\gamma = 2-k$. Thus, since $\deg_C b=1$, we have
\[
\deg_R \langle\m_k^\gamma(b^{\otimes k}),b\rangle =
\deg_C\m_k^\gamma(b^{\otimes k})+\deg_C b-\dim L
=(2-k+k)+1-n=3-n.
\]
Now consider summands that come in $\m_{-1}^\gamma$.
Recall that $\deg_C\gamma=2$. Therefore,
\begin{align*}
\deg_R T^{\beta}\int_{\M_{0,l}(\beta)}\wedge_{j=1}^l evi_j^*\gamma=&
\mu(\beta)+2l-\dim\M_{0,l}(\beta)\\
=&\mu(\beta)+2l-(n-3+\mu(\beta)+2l) \\
=& 3-n.
\end{align*}
\end{proof}

The following lemma plays the role of the energy zero axiom at the level of bounding chains and their associated potentials $c$.
\begin{lm}\label{lm:db0}
Let $(\gamma,b)$ be a bounding pair with $\mg(e^b)=c\cdot 1$. Then $d([T^{\beta_0}](b))=0$ and $[T^{\beta_0}](c)=0.$
\end{lm}

\begin{proof}
On the one hand, keeping in mind that $\deg_Cb=1$ and $\gamma|_L = 0$,
Proposition~\ref{q_zero} gives
\begin{multline*}
d([T^{\beta_0}](b))=[T^{\beta_0}](- \gamma|_L+db-b\wedge b)=\\
=[T^{\beta_0}](\q^{\beta_0}_{0,1}(\gamma)+ \q^{\beta_0}_{1,0}(b)+\q^{\beta_0}_{2,0}(b,b))
=[T^{\beta_0}](\mg(e^b))
=[T^{\beta_0}](c\cdot 1).
\end{multline*}
On the other hand, $db\in A^{\ge 1}(L)\otimes R$ while $c\cdot 1\in A^0(L)\otimes R.$  Therefore, $d[T^{\beta_0}](b)=0$ and $[T^{\beta_0}](c)=0.$
\end{proof}

The following lemma is important for the proof of deformation invariance axiom.
Recall the notation $\sly^J$ and $\L^J=\L^J_\omega$ from Section~\ref{sssec:Jnov}.
\begin{lm}\label{lm:smNov}
Let $J$ be an almost complex structure on $X$ and let $S_L \subset H_2(X,L;\Z)$ be a subgroup contained in the kernel of the Maslov index $\mu : H_2(X,L;\Z) \to \Z.$
Then the Novikov ring $\Lambda^J$ is the same for any symplectic form $\omega$ with respect to which $L$ is Lagrangian and $J$ is tame and such that $S_L$ is contained in the kernel of $\omega : H_2(X,L;\Z) \to \R$.
\end{lm}
\begin{proof}
Let $\omega$ be a symplectic form on $X$ with respect to which $L$ is Lagrangian and $J$ is tame and such that $S_L$ is contained in the kernel of $\omega : H_2(X,L;\Z) \to \R$. Since $J,L,$ and $S_L,$ are fixed, the monoid $\sly^J$ is fixed. Since $S_L$ is contained in the kernel of $\omega : H_2(X,L;\Z) \to \Z,$ there is an induced monoid homomorphism $\omega : \sly^J \to \R.$ The Novikov ring $\Lambda^J_\omega$ is a subring of a completion of the group ring of $\sly^J.$ A priori, both the completion and the subring depend on $\omega,$ but we show they do not.
Let $\lambda = \sum_{i=0}^\infty a_iT^{\beta_i}\in \L^J_\omega$ with $\beta_i \in \sly^J$ distinct. We show $\lambda \in \Lambda^J_{\omega'}$ for any $\omega'$ with respect to which $L$ is Lagrangian and $J$ is tame and such that $S_L$ is contained in the kernel of $\omega' : H_2(X,L;\Z) \to \R$. That is, we show that $\omega'(\beta_i)\ge 0$ for all $i$, and $\lim_{i\to \infty}\omega'(\beta_i)=\infty.$ Indeed, by the argument of Section~\ref{ssec:add_not}, we can write a countable list $B_j \in \sly^J$ for $j = 1,2,\ldots$ of all the classes that can be represented by a non-constant $J$-holomophic disk. By the definition of $\sly^J$, we can write $\beta_i = \sum_{j = 1}^{N_i} x_{ij}B_j$ with $x_{ij} \in \Z_{\geq 0}.$ Since $\omega'(B_j) \geq 0,$ it follows that $\omega'(\beta_i) \geq 0.$ To verify the second condition, suppose by way of contradiction that there exists $E\in \R_{>0}$ such that after passing to a subsequence we have $\omega'(\beta_i) < E$ for all $i.$
Then, for all $i,j,$ such that $x_{ij} \neq 0,$ we have $\omega'(B_j) < E.$ It follows from Gromov compactness that we can take $N_i$ independent of $i.$ Moreover, there is a constant $\hbar > 0$ such that $\omega'(B_j) > \hbar.$ So, $x_{ij} < E/\hbar.$
Consequently, the $\beta_i$ belong to a finite set, which is a contradiction.
\end{proof}

We are now ready to prove the axioms of the OGW invariants, as stated in the introduction.
\begin{proof}[Proof of Theorem~\ref{axioms}]\leavevmode
It is enough to prove the statements for the constraints $A_1,\ldots,A_l,$ chosen from among the basis elements $\g_0,\ldots,\g_N.$ Assume without loss of generality that $\g_0=1\in\Hh^{0}(X,L;Q)$ (resp. $1 \in \Hh^{0}_\phi(X,L;Q)$ in the real case). Recall that in the definition of $\ogw$, we use a bounding chain $b$ that satisfies $\int_L b=s.$ In particular, $\int_L b$ satisfies the unit and divisor properties.
\begin{enumerate}
	\item
By Lemma~\ref{homogeneity} the superpotential $\Omega$ is homogeneous of degree $3-n$.
	Taking the derivative $\d_{t_{i_1}}\cdots\d_{t_{i_l}}\d_s^k\Omega$ reduces the degree by $k(1-n)+\sum_{j=1}^l(2-|\gamma_{i_j}|),$ and taking out $T^{\beta}$ further reduces degree by $\mu(\beta).$ The only monomials that don't vanish after substituting $s=t_j=0$ are those of degree zero. So, in order for $\ogw_{\beta,k}(\gamma_{i_1},\ldots,\gamma_{i_l})$ to be nonzero it is necessary that
	\begin{align*}
	0=&3-n-\big(k(1-n)+\sum_{j=1}^l(2-|\gamma_{i_j}|)+\mu(\beta)\big)\\
	&=3-n-\mu(\beta)-k-2l+kn+\sum_{j=1}^l|\gamma_{i_j}|
	,
	\end{align*}
	as required.
    \item
This is an immediate consequence of the definition since partial derivatives are graded commutative.
	\item
Under the hypothesis of Theorem~\ref{thm1} (resp. Theorem~\ref{thm2}), Proposition~\ref{lm:d_t0b} (resp. Proposition~\ref{div_real})  says we may assume $b$ satisfies the unit property.
	So, by Propositions~\ref{cl:symmetry} and~\ref{q_fund},
	\begin{align*}
(-1)^n	\d_{t_0}\Oh=&\sum_{k,l}\frac{1}{l!(k+1)}\langle\d_{t_0}\qkl(b^{\otimes k};\gamma^{\otimes l}), b\rangle+\d_{t_0}\mg_{-1}\\
	=&\sum_{k,l}\frac{1}{(l-1)!(k+1)}\langle\qkl(b^{\otimes k};1\otimes \gamma^{\otimes l-1}),b\rangle+0\\
	=&-\langle T^{\beta_0}\cdot1,b\rangle =(-1)^{n+1}T^{\beta_0}\int_L b=(-1)^{n+1}T^{\beta_0}s.
	\end{align*}
Lemma~\ref{lm:d_tjD} implies that $\d_{t_0}\mathcal{D}(\Oh)=\mathcal{D}(\d_{t_0}\Oh)=\mathcal{D}(T^{\beta_0}s)=0,$ so
 \[
 \d_{t_0}\Omega=\d_{t_0}\Oh=-T^{\beta_0}s.
 \]
	Therefore, for any multi-index $J$, we have $\d_J\d_{t_0}\Omega|_{s=t_j=0}\ne 0$ only if $J=\{s\}$, and in this case
	\[
	\d_J\d_{t_0}\Omega|_{s=t_j=0}=-T^{\beta_0},
	\]
	as required.
	\item
Abbreviate $\zeta=[T^{\beta_0}](b).$ Observe that
\[
\int_L\zeta=[T^{\beta_0}]\left(\int_L b\right) =s.
\]
So, by Proposition~\ref{q_zero} and Lemma~\ref{lm:db0},
	\begin{align*}
(-1)^n
[T^{\beta_0}](\Oh)=&
	\sum_{k,l}\frac{1}{l!(k+1)}\langle\qkl^{\beta_0}(\zeta^{\otimes k};\gamma^{\otimes l}),\zeta\rangle\\ =&\frac{1}{2}\langle\q^{\beta_0}_{1,0}(\zeta),\zeta\rangle+ \langle\q^{\beta_0}_{0,1}(t_0\cdot 1),\zeta\rangle+\frac{1}{3}\langle\q_{2,0}^{\beta_0}(\zeta,\zeta),\zeta\rangle\\
	=&\frac{1}{2}\langle d\zeta,\zeta\rangle-\langle t_0\cdot 1,\zeta\rangle-\frac{1}{3}\langle \zeta\wedge \zeta,\zeta\rangle\\
	=&0+(-1)^{n+1}t_0s-0=(-1)^{n+1}t_0s.
	\end{align*}
Lemma~\ref{lm:d_tjD} implies
$
[T^{\beta_0}](\mathcal{D}(\Oh))=0,
$	
so
\[
[T^{\beta_0}](\Omega)=[T^{\beta_0}](\Oh)=-t_0s.
 \]
 Therefore, $[T^{\beta_0}](\d_J\Omega|_{s=t_j=0})\ne 0$ only if $J=\{t_0,s\}$, and in this case the resulting invariant takes the value $-1$.
	\item
 	Without loss of generality, assume $A_l=\gamma_i$ with $|\gamma_i|=2$.
Under the hypothesis of Theorem~\ref{thm1} (resp. Theorem~\ref{thm2}),
 Proposition~\ref{div} (resp. Proposition~\ref{div_real}) says we may assume $b$ satisfies the divisor property.
	By a calculation similar to~\eqref{eq:decomp} with
\[
s_2:=s^{\lambda_j}(T^\beta\prod_{m=k+1}^1\lambda_{z_m} \prod_{j=N}^{0}t_j^{r_j}),
\qquad
s_2':=s_2+\sum_{m=1}^{k+1}\deg \lambda_{z_m},
\]
we find that for each $j\in Z_{\ge 0},$
	\begin{align*}
\ell_i(\lambda_j)&\cdot[\lambda_j](\Oh)=
\hspace{-3em}\sum_{\substack{k,l,\beta \\ \sum_{m=0}^N r_m=l \\ T^\beta\prod_{m=k+1}^1\lambda_{z_m}\prod_{j=N}^{0}t_j^{r_j} =(-1)^{s_2}\lambda_j}}
\hspace{-2em}(-1)^{s_2'} r_i
	 \prod_{m=1}^N\frac{1}{r_m!}\frac{1}{k+1}\langle
\qkl^{\beta}(\otimes_{m=1}^kb_{z_m};\otimes_{m=0}^{N}\gamma_{m}^{\otimes r_m}),
b_{z_{k+1}}\rangle+\\
&+
\hspace{-2em}\sum_{\substack{k,l,\beta \\ \sum_{m=0}^N r_m=l \\ T^\beta\prod_{m=k+1}^1\lambda_{z_m}\prod_{j=N}^{0}t_j^{r_j} =(-1)^{s_2}\lambda_j}}\hspace{-3em}	 (-1)^{s_2'}\prod_{m=1}^N\frac{1}{r_m!}\sum_{a=1}^{k}\frac{1}{k+1}\cdot\\
&\hspace{10em}\langle
\qkl^{\beta}(\otimes_{m=1}^{a-1}b_{z_m}\otimes \ell_i(\lambda_{z_a})b_{z_{a}}\otimes\otimes_{m=a+1}^kb_{z_m}; \otimes_{m=0}^{N}\gamma_{m}^{\otimes r_m}), b_{z_{k+1}} \rangle+\\
&+
\hspace{-2em}\sum_{\substack{k,l,\beta \\ \sum_{m=0}^N r_m=l \\ T^\beta\prod_{m=k+1}^1\lambda_{z_m}\prod_{j=N}^{0}t_j^{r_j} =(-1)^{s_2}\lambda_j}}\hspace{-2em}
	(-1)^{s_2'} \prod_{m=1}^N\frac{1}{r_m!}\frac{1}{k+1}\langle
\qkl^{\beta}(\otimes_{m=1}^kb_{z_m}; \otimes_{m=0}^{N}\gamma_{m}^{\otimes r_m}), \ell_i(\lambda_{z_{k+1}})b_{z_{k+1}} \rangle\\
&+
\hspace{-2em}\sum_{\substack{l,\beta \\ \sum_{m=0}^N r_m=l \\T^\beta\prod_{j=N}^{0}t_j^{r_j} =(-1)^{s_2}\lambda_j}}\!\!\!\!
(-1)^{s_2'}	r_i\prod_{m=1}^N\frac{1}{r_m!}
\q_{-1,l}^\beta(\otimes_{m=0}^{N}\gamma_{m}^{\otimes r_m})\\
\end{align*}
By Proposition~\ref{cl:q_div} and the divisor property for $b,$ we continue
\begin{align*}
=&\sum_{\substack{k,l,\beta \\ \sum_{m=0}^N r_m=l\\ T^\beta\prod_{m=k+1}^1\lambda_{z_m}\prod_{j=N}^{0}t_j^{r_j} =\\
\quad =(-1)^{s_2}\lambda_j}}\hspace{-2em}
  (-1)^{s_2'}\cdot\\
&\hspace{5em}\cdot
 	\gamma_i(T^\beta)
    \frac{1}{(r_i-1)!}\prod_{m\ne i}\frac{1}{r_m!}\frac{1}{k+1}\langle
    \qkl^{\beta}(\otimes_{m=1}^kb_{z_m};\gamma_i^{\otimes (r_i-1)}\otimes_{m\ne i}\gamma_{m}^{\otimes r_m}),
    b_{z_{k+1}}\rangle+\\
&+
\hspace{-2em}\sum_{\substack{k,l,\beta \\ \sum_{m=0}^N r_m=l\\ T^\beta\prod_{m=k+1}^1\lambda_{z_m}\prod_{j=N}^{0}t_j^{r_j}=\\ =(-1)^{s_2}\lambda_j}}\hspace{-2em}
(-1)^{s_2'}
\sum_{a=1}^{k}\gamma_i(\lambda_{z_a})
    	 \prod_{m=0}^N\frac{1}{r_m!}
    \frac{1}{k+1}\cdot\\
    &\hspace{12em}\cdot\langle
    \qkl^{\beta}(\otimes_{m=1}^{a-1}b_{z_m}\otimes b_{i(z_{a})}\otimes\otimes_{m=a+1}^kb_{z_m}; \otimes_{m=0}^{N}\gamma_{m}^{\otimes r_m}), b_{z_{k+1}} \rangle+\\
&+
\hspace{-2em}\sum_{\substack{k,l,\beta \\ \sum_{m=0}^N r_m=l\\ T^\beta\prod_{m=k+1}^1\lambda_{z_m}\prod_{j=N}^{0}t_j^{r_j}=\\ =(-1)^{s_2}\lambda_j}}\hspace{-2em}
(-1)^{s_2'}\gamma_i(\lambda_{z_{k+1}})
    	 \prod_{m=0}^N\frac{1}{r_m!}
    \frac{1}{k+1}\langle
\qkl^{\beta}(\otimes_{m=1}^kb_{z_m}; \otimes_{m=0}^{N}\gamma_{m}^{\otimes r_m}), b_{i(z_{k+1})} \rangle+\\
&+
\sum_{\substack{l,\beta \\ \sum_{m=0}^N r_m=l \\T^\beta\prod_{j=N}^{0}t_j^{r_j}=\\ =(-1)^{s_2}\lambda_j}}\!\!\!\!
(-1)^{s_2'}\gamma_i(T^\beta)\frac{1}{(r_i-1)!}\prod_{m\ne i}\frac{1}{r_m!}
    \q_{-1,l}^\beta(\gamma_i^{\otimes (r_i-1)}\otimes_{m\ne i}\gamma_{m}^{\otimes r_m})\\
=&
\gamma_i(\lambda_j)\sum_{\substack{k,l,\beta \\ \sum_{m=0}^N r_m=l\\ T^\beta\prod_{m=k+1}^1\lambda_{z_m}\prod_{j=N}^{0}t_j^{r_j}=\\ =(-1)^{s_2}\lambda_{i(j)}}}\!\!\!\!
(-1)^{s_2'}	\prod_{m=0}^N\frac{1}{r_m!}\frac{1}{k+1}\langle
    \qkl^{\beta}(\otimes_{m=1}^kb_{z_m};\otimes_{m=0}^{N}\gamma_{m}^{\otimes r_m}),
b_{z_{k+1}}\rangle+\\
&+
\gamma_i(\lambda_j)\sum_{\substack{l,\beta \\ \sum_{m=0}^N r_m=l \\T^\beta\prod_{j=N}^{0}t_j^{r_j} =(-1)^{s_2}\lambda_{i(j)}}}\!\!\!\!
(-1)^{s_2'}
\prod_{m=0}^N\frac{1}{r_m!}
\q_{-1,l}^\beta(\otimes_{m=0}^{N}\gamma_{m}^{\otimes r_m})\\
=& \gamma_i(\lambda_j)\cdot [\lambda_{i(j)}](\Oh).
	\end{align*}
Using the fact that $\d_{t_i}\lambda_j=\ell_i(\lambda_j)\lambda_{i(j)}$, Lemma~\ref{lm:d_tjD}, and the preceding computation for $\Oh$,  we have
\begin{align*}
\ell_i(\lambda_j)\lambda_{i(j)}\cdot[\lambda_j](\D(\Oh))=&
\d_{t_i}\left(\lambda_j\cdot[\lambda_j](\D(\Oh))\right)\\
=&\d_{t_i}\D(\lambda_j\cdot[\lambda_j](\Oh))\\
=&\D\left(\d_{t_i}(\lambda_j\cdot[\lambda_j](\Oh))\right)\\
=&\D\left(\ell_i(\lambda_j)\lambda_{i(j)}\cdot[\lambda_j](\Oh)\right)\\
=&\D\left(\lambda_{i(j)}\gamma_i(\lambda_j)\cdot[\lambda_{i(j)}](\Oh)\right)\\
=&\gamma_i(\lambda_j)\lambda_{i(j)}\cdot[\lambda_{i(j)}]\D(\Oh).
\end{align*}
It follows that
\[
\ell_i(\lambda_j)\cdot[\lambda_j](\Omega)=\gamma_i(\lambda_j)\cdot [\lambda_{i(j)}](\Omega).
\]
Thus, writing $\lambda_j=T^\beta s^k \prod_{m=0}^Nt_m^{r_m}$,
\begin{align*}
\ogw_{\beta,k}(\otimes_{m=0}^N\g_m^{\otimes r_m}) = &
k!\prod_{m=0}^Nr_m!\cdot[\lambda_m](\Omega)\\
=&k!(r_i-1)!\prod_{m\ne i}r_m!\cdot\ell_i(\lambda_j)\cdot[\lambda_j](\Omega)\\8
=& \gamma_i(\lambda_j)\cdot k!(r_i-1)!\prod_{m\ne i}r_m!\cdot[\lambda_{i(j)}](\Omega)\\
=&\int_\beta\gamma_i \cdot \ogw_{\beta,k}(\otimes_{m\ne i}=\g_m^{\otimes r_m}\otimes \g_i^{\otimes r_i-1}).
\end{align*}

\item
Consider the space of symplectic structures on $X$ with the topology induced from the $C^0$ topology on $2$-forms on $X.$ It suffices to show that around any symplectic structure $\omega_0$ on $X$ there is a neighborhood $U$ in the space of symplectic structures where the invariants $\ogw_{\beta,k}$ stay constant. Let $J$ be an $\omega_0$ tame almost complex structure on~$X.$
Recall the notation of Section~\ref{sssec:Jnov}.
Take $U$ small enough that $J$ is tame with respect to all $\omega \in U.$
By Lemma~\ref{lm:smNov}, the Novikov ring $\L^J$ is the same for all $\omega\in U$.
Under the hypothesis of Theorem~\ref{thm1} (resp. Theorem~\ref{thm2}), use Proposition~\ref{prop:existJ} (resp. Proposition~\ref{cor:spin_b_existJ}) to choose
$(\gamma,b)\in (\mI_{Q^J}D^J)_2\oplus(\mI_{R^J})_{1-n}$ such that $b$ is a bounding chain for $\m^{J,\gamma}$ and $([\gamma],\int_L b)=(\Gamma,s)$.
The definition of a bounding chain depends on the symplectic form $\omega$ only through $J$ and the Novikov ring $\Lambda^{J}$. So, $b$ is a bounding chain for $\m^{J,\gamma}$ for all $\omega \in U.$ It follows that $b$ is a bounding chain for $\m^\gamma$ as used in the definition of the invariants $\ogw_{\beta,k}$ for all $\omega\in U$. Finally, the superpotential $\Omega$ depends on $\omega$ only through $J.$ So, the superpotential $\Omega(\gamma,b)$ and consequently the invariants $\ogw_{\beta,k}$ remain constant for $\omega \in U.$
\end{enumerate}
\end{proof}

\subsection{Relaxed relative de Rham complex}\label{ssec:relax}
Define subcomplexes $\widehat A^*(X,L) \subset A^*(X)$ and $\widehat A^*(I\times X, I \times L) \subset A^*(I\times X),$ by
\begin{gather*}
\widehat{A}^*(X,L):=\left\{\eta\in A^*(X)\;\bigg|\,\int_Li^*\eta=0\right\},
\\
\widehat{A}^*(I\times X,I\times L):=\left\{\etat\in A^*(I\times X)\;\bigg|\,p_*(\Id\times i)^*\etat=0\right\}.
\end{gather*}
Abbreviate $\Hhh^*(X,L;\R):= H^*(\Ah^*(X,L),d)$. The inclusion of the underlying complexes induces a map $h: \Hh^*(X,L;\R) \to \Hhh^*(X,L;\R)$. In a real setting, let $\Hhh_\phi^{even}(X,L;\R)$ denote the direct sum over $k$ of the $(-1)^{k}$-eigenspace of $\phi^*$ acting on $\Hhh^{2k}(X,L;\R).$ Let $h_\phi : \Hh_\phi^{even}(X,L;\R) \to \Hhh_\phi^{even}(X,L;\R)$ denote the induced map.
\begin{lm}\label{lm:Hrel}\leavevmode
\begin{enumerate}
\item
    If $H^*(L;\R)\simeq H^*(S^n)$, then $h : \Hh^*(X,L;\R) \overset{\sim}{\lrarr} \Hhh^*(X,L;\R)$.
\item
    In a real setting, if $H^j(L;\R)\simeq H^j(S^n;\R)$ for $j \equiv 0,3,n,n+1 \pmod 4$, then $h_\phi : \Hh_\phi^{even}(X,L;\R) \overset{\sim}{\lrarr} \Hhh_\phi^{even}(X,L;\R)$.
\end{enumerate}
\end{lm}
\begin{proof}
Define the graded vector spaces $\Hh^*(L)$ and $\R[-n]$ by
\[
\Hh^j(L)=
\begin{cases}
0, & j=0,\\
H^j(L), & j>0,
\end{cases}
\qquad
\R[-n]^j=
\begin{cases}
0, & j\ne n,\\
\R, & j=n.
\end{cases}
\]
We have a commuting diagram of long exact sequences
\begin{equation}\label{eq:les}
\xymatrix{
\cdots\ar[r]& \Hh^{j-1}(L)\ar[r]\ar[d]^{\int_L} & \Hh^j(X,L)\ar[r]\ar[d]^h & H^j(X)\ar[r]\ar[d]^{\sim} & \Hh^j(L)\ar[r]\ar[d]^{\int_L} & \Hh^{j+1}(X,L)\ar[r]\ar[d]^h & \cdots\\
\cdots\ar[r]& \R[-n]^{j-1}\ar[r] & \Hhh^j(X,L)\ar[r] & H^j(X)\ar[r] & \R[-n]^j\ar[r] & \Hhh^{j+1}(X,L)\ar[r] & \cdots
}
\end{equation}
\begin{enumerate}
    \item
By assumption, the map $\int_L: \Hh^j(L)\to \R[-n]$ is an isomorphism. By the five lemma, $h$ is an isomorphism.
    \item
Note that $\Hh^*(L)$ and $\R[-n]$ are $\phi^*$ invariant. In particular, the eigenspaces of $(-1)$ under the $\phi^*$ action for both groups is $\{0\}$.

Let $j\equiv 2\pmod 4$ and consider the eigenspaces of $(-1)$ in diagram~\eqref{eq:les}. We get
\[
\xymatrix{
0\ar[r]\ar[d]& \Hh^j_\phi(X,L)\ar[r]\ar[d]^{h^j}& H^j_\phi(X) \ar[r]\ar[d]^\sim & 0\ar[d]\\
0\ar[r]& \Hhh^j_\phi(X,L)\ar[r]& H^j_\phi(X) \ar[r] & 0 .
}
\]
In particular, this implies $\Hh^j_\phi(X,L)\simeq \Hhh^j_\phi(X,L)$.

Let $j\equiv 0\pmod 4$ and consider the eigenspaces of $(+1)$ in diagram~\eqref{eq:les}. By assumption on $H^*(L)$, the cohomology $\Hh^j(L)=0$ unless $j=n$. In either case, the map $\int_L:\Hh^j(L)\to \R[-n]^j$ is an isomorphism. Similarly, $\Hh^{j-1}(L)=0$ unless $j=n+1,$ and $\int_L:\Hh^{j-1}(L)\to \R[-n]^{j-1}$ is an isomorphism.
Thus, we get
\[
\xymatrix{
H^{j-1}(X)^{\phi^*}\ar[r]\ar[d]^\sim & \Hh^{j-1}(L)\ar[r]\ar[d]^\sim& \Hh^j_\phi(X,L)\ar[r]\ar[d]^{h^j}& H^j_\phi(X)\ar[r]\ar[d]^\sim & \Hh^j(L) \ar[d]^\sim\\
H^{j-1}(X)^{\phi^*}\ar[r] & \R[-n]^{j-1}\ar[r]& \Hhh^j_\phi(X,L)\ar[r]& H^j_\phi(X)\ar[r]& \R[-n]^j.
}
\]
By the five lemma, this implies $\Hh^j_\phi(X,L)\simeq \Hhh^j_\phi(X,L)$.
\end{enumerate}
\end{proof}

In~\cite[Section 4.2.13]{ST1} we remark that
for $\gamma\in (\mI_Q\widehat{A}^*(X,L))_2$ with $d\gamma=0$, the triple
$(\{\mg_k\}_{k\ge 0},\langle\,,\,\rangle,1)$ is a cyclic unital $A_\infty$ structure on $C$.
Moreover, given closed $\gamma,\gamma'\in (\mI_Q\widehat{A}^*(X,L))_2$ with $[\gamma]=[\gamma']\in \Hhh^*(X,L)$, there exists a cyclic unital pseudoisotopy $\mgt$ from $\mg$ to $\mgp$ with $\gt\in \widehat{A}^*(I\times X,I\times L).$

The results in this paper similarly extend to the case when $\gamma$ is taken in $\widehat{A}^*(X,L;Q)$ rather than $A^*(X,L;Q)$. Specifically, the proof of Theorem~\ref{thm_inv} holds verbatim.
Theorems~\ref{thm1} and~\ref{thm2} hold as well. One replaces $\Hh^*$ with $\Hhh^*$ in the ranges of $\varrho$ resp. $\varrho_\phi.$ However, Lemma~\ref{lm:Hrel} implies that under the hypothesis of Theorem~\ref{thm1} (resp. Theorem~\ref{thm2}) the range of $\varrho$ (resp. $\varrho_\phi$) is, in fact, unchanged up to canonical isomorphism.  The only difference in the proofs of Theorems~\ref{thm1} and~\ref{thm2} is in Lemmas~\ref{lm:td} and~\ref{lm:deg_ut}. In the case of Lemma~\ref{lm:td}, the conclusion becomes that $o_j$ is exact. Indeed, $\int_L(\gamma)_n=\int_L\gamma=0$ implies that $[(i^*\gamma)_n]=0\in H^n(L)\otimes R$. Therefore,
\[
(\mg(e^{b_{(l)}}))_{n} = (\q_{0,1}^{\beta_0}(\gamma)+\q_{1,0}^{\beta_0}(b_{(l)}) +\q_{2,0}^{\beta_0}(b_{(l)},b_{(l)}))_{n}
=(-i^*\gamma+db_{(l)})_{n}
\]
is exact.
The same argument applies in the case of Lemma~\ref{lm:deg_ut} with the conclusion that $\ot_j$ is exact.
Likewise, the proof of Lemma~\ref{lm_rho} carries through and $\varrho$ (resp. $\varrho_\phi$) is well defined.

Finally, to justify the relaxed version of Theorem~\ref{axioms}, let $\g$ be as in Section~\ref{sssec:intro_OGW} and consider a relaxed bounding pair $(\gamma,b)\in \mI_Q \widehat A(X,L;Q)\oplus \mI_D C$ with $([\gamma],b)=(\g,s)$. By Theorem~\ref{thm1} (resp. Theorem~\ref{thm2}), there exists a bounding pair $(\gamma',b')\in \mI_Q D \oplus \mI_R C$ with $\varrho(\gamma',b')=(\g,s)$ (resp. $\varrho_\phi(\gamma',b')=(\g,s)$). By Theorem~\ref{axioms}, the invariants arising from $\Omega(\gamma',b')$ satisfy the open Gromov-Witten axioms.  On the other hand, by the relaxed version of Theorem~\ref{thm1} (resp. Theorem~\ref{thm2}), we have $(\gamma,b)\sim (\gamma',b')$ in the relaxed sense. So, the relaxed version of Theorem~\ref{thm_inv} gives $\Omega(\gamma,b)=\Omega(\gamma',b')$.

\subsection{Comparison with Welschinger invariants}\label{ssec:comparison}
\begin{proof}[Proof of Theorem~\ref{Welsch}]
Assume without loss of generality that $[\gamma_N]=A.$ Abbreviate $k=k_{d,l}.$ By Proposition~\ref{prop:3d} we may choose $b=s\cdot\bb$ where $\bb$ is a representative of the Poincar\'e dual of a point.
Recall that if $n=2$, then $k\le 1$, so $s^k \neq 0.$ By definition,
\begin{align*}
\ogw_{\beta,k}(\gamma_N^{\otimes l})=&
k!l!\cdot[T^\beta s^k t_N^l](\Omega)\\
=&k!l!\left(
\frac{1}{l! k}\langle\q^\beta_{k-1,l}(\bb^{\otimes (k-1)};\gamma_N^{\otimes l}),\bb\rangle\right)\\
=&-(k-1)!\left(\int_{\M_{k,l}(\beta)} \wedge_{j=1}^l(evi_j)^*\gamma_N\wedge \wedge_{j=0}^{k-1} (evb_j)^*\bb\right).
\end{align*}

Denote by $\M_{k,l}^S(\beta)$ the moduli space of genus zero $J$-holomorphic open stable maps $u:(\Sigma,\d \Sigma) \to (X,L)$ of degree $\beta \in \sly$ with one boundary component, $k$ unordered boundary points, and $l$ interior marked points.
It comes with evaluation maps
$evb_j^\beta:\M_{k,l}^S(\beta)\to L$, $j=1,\ldots,k$, and $evi_j^\beta:\M_{k,l}^S(\beta)\to X$, $j=1,\ldots,l$.
The space $\M_{k,l}^S(\beta)$ carries a natural orientation induced by the spin structure on $L$ as in~\cite[Chapter 8]{FOOO} and~\cite{SolomonThesis}. The diffeomorphism of $\M_{k,l}^S(\beta)$ corresponding to relabeling boundary marked points by a permutation $\sigma \in S_k$ preserves or reverses orientation depending on $sgn(\sigma).$ So, in both cases above,
\[
\ogw_{\beta,k}(\gamma_N^{\otimes l})=-\int_{\M_{k,l}^S(\beta)} \wedge_{j=1}^l(evi_j)^*\gamma_N\wedge \wedge_{j=1}^{k}(evb_j)^*\bb.
\]
Thus, the theorem follows from \cite[Theorem 1.8]{SolomonThesis}:
The factor of $2^{1-l}$ arises from two independent sources.
First, as explained in \cite[Theorem 1.8]{SolomonThesis}, each real holomorphic sphere corresponds to two holomorphic disks. This gives a factor of $2.$ Second, each interior constraint $\gamma_N$ is Poincar\'e dual to the homology class of point. On the other hand, Welschinger~\cite{Welschinger3,Welschinger2} considers constraints that are pairs of conjugate points, and thus represent twice the homology class of a point. This gives an additional factor of $2^{-l}.$
\end{proof}

\subsection{The case of Georgieva}\label{sssec:geosuperpot}

\subsubsection{Superpotential invariants}\label{sssec:geosi}

Suppose the hypothesis of Proposition~\ref{prop:altrealspin} is satisfied. Let $W_G = H^*(X;\R)^{-\phi^*}$.
Fix
$\Gg_0,\ldots,\Gg_{N^G},$ a basis of $W_G$, set $\deg \tg_j=2-|\Gg_j|$, and take
\[
\Gg:=\sum_{j=0}^{\;N^G}\tg_j\Gg_j.
\]
By Proposition~\ref{prop:altrealspin}, choose a real bounding pair $(\gag,0)$ such that
\begin{equation}\label{eq:cbpg}
\varrho_\phi^G([\gag,0])=\Gg.
\end{equation}
By Theorem~\ref{thm_inv}, the superpotential $\Omega = \Omega(\gag,0)$ is independent of the choice of $(\gag,0).$
The associated open Gromov-Witten invariants
\[
\ogw_{\beta,k}^{G,L} : W_G^{\otimes l} \lrarr \R
\]
are defined by setting
\[
\ogw_{\beta,k}^{G,L}(\Gg_{i_1},\ldots,\Gg_{i_l}):= \text{ the coefficient of }\Tg^{\beta} \text{ in }
\d_{\tg_{i_1}}\cdots\d_{\tg_{i_l}}\d_\sg^k\Omega|_{\sg=0,\tg_j=0}
\]
and extending linearly to general input.
Observe that for $k>0$ the invariant $\ogw_{\beta,k}^{G,L}$ vanishes since $\d_{\sg}\Omega=0$.

\subsubsection{Comparison with Georgieva's invariants}
\begin{prop}\label{thm:penkasuperpot}
Suppose $(X,\omega,\phi)$ is admissible in the sense of~\cite{Georgieva}, choose a $\phi$-orienting structure, and let $\s_L$ be the associated relative spin structure for the component $L \subset fix(\phi)$. Assume that $\frac{\mu(\beta)}{2} + w_{\s_L}(\chi_L(\beta)) \equiv 0 \pmod 2$ for all components $L \subset fix(\phi)$ and all $\beta\in \sly^G_L$. For $d \in \mathcal{A}$ and $A_j \in W_{G}$, we have
\[
\sum_{\substack{\chi_L(\beta) = d, \\ L \subset fix(\phi)}}\ogw^{G,L}_{\beta,0}(A_1,\ldots,A_l)=
 2^{1-l}
\ogw^{\text{Georgieva}}_{d,0,l}(A_1,\ldots,A_l).
\]
\end{prop}

\begin{proof}
Without loss of generality, assume $A_j = \Gg_{i_j}$ for $j = 1,\ldots,l.$ Let $\gag_j\in A^*(X;Q^G)^{-\phi^*}$ be such that $\gag = \sum_{j = 0}^N \tg_j\gag_j$ and thus, $[\gag_j]=\Gg_j$. By definition,
\[
\ogw^{G,L}_{\beta,0}(\otimes_{j=1}^l\Gg_{i_j})
= [\Tg^{\beta}] (\d_{\tg_{i_1}}\cdots\d_{\tg_{i_l}}\Omega(\gamma,0))|_{\sg=0,\tg_j=0}.
\]
Since $d\in \mathcal{A}$, for each $\beta$ such that $\chi_L(\beta) = d$, we have $\beta\not\in \Im(H_2(X;\Z)\to\sly)$ and therefore $\Tg^\beta \tg_{i_1}\cdots \tg_{i_l}$ is not of type $\D.$ Thus,
\begin{equation*}
\ogw^{G,L}_{\beta,0}(\otimes_{j=1}^l\Gg_{i_j})=
\q_{-1,l}^\beta(\otimes_{j=1}^l\gag_{i_j})
=\int_{\M_{0,l}(\beta)}\wedge_{j=1}^levi_j^*\gag_{i_j}.
\end{equation*}
Summing, we obtain,
\begin{equation}\label{eq:Gsum}
\sum_{\substack{\chi_L(\beta) = d, \\ L \subset fix(\phi)}}\ogw^{G,L}_{\beta,0}(\otimes_{j=1}^l\Gg_{i_j}) = \sum_{\substack{\chi_L(\beta) = d, \\ L \subset fix(\phi)}}\int_{\M_{0,l}(\beta)}\wedge_{j=1}^levi_j^*\gag_{i_j}.
\end{equation}
It follows from~\cite[Lemma 7.3]{Georgieva} that the orientation of the moduli spaces used in defining $\ogw_{d,0,l}^{Georgieva}$ is the one arising from the relative spin structures $\s_L.$ Hence, the right hand side of~\eqref{eq:Gsum} is readily seen to agree with the definition of~\cite{Georgieva}, up to a factor of $2^{1-l}.$ The factor $2^{1-l}$ comes from the decorations $\pm$ introduced in~\cite[Section 3]{Georgieva} for all interior marked points with the exception of the first.
\end{proof}

\subsubsection{Comparison of superpotential invariants}
\leavevmode
Fix a component $L\subset \fix(\phi)$. Let $\s$ be an arbitrary relative spin structure on $L.$
Take $\sly_L$ and $W_L$ as usual according to whether $L$ satisfies the hypothesis of Theorem~\ref{thm1} or Theorem~\ref{thm2}. Let $q : \sly_L \to \sly^G_L$ denote the natural map. Let $W_\phi^L$ denote $H^*(X;\R)^{-\phi^*}$  (resp. $H^{even}_\phi(X;\R)^{-\phi^*}$) in the case that $L$ satisfies the hypothesis of Theorem~\ref{thm1} (resp. Theorem~\ref{thm2}).
Similarly to Section~\ref{sssec:comps}, identify $W_\phi^L$ with a subspace of $W_L.$
\begin{prop}\label{pr:superpots}
Suppose $L$ and $\s$ satisfy the hypotheses of Proposition~\ref{prop:altrealspin} and either Theorem~\ref{thm1} or Theorem~\ref{thm2}. For $\hat\beta \in \sly^G_L$ and $A_1,\ldots,A_l \in W_\phi^L,$ we have
\begin{equation}\label{eq:superpots}
\sum_{q(\beta) = \hat \beta}\ogw_{\beta,0}(A_1,\ldots,A_l) = \ogw_{\hat \beta,0}^G(A_1,\ldots,A_l).
\end{equation}
\end{prop}
\begin{proof}
Let $\g_0,\ldots,\g_N,\g,$ and $(\gamma,b)$ be as in the beginning of Section~\ref{sssec:intro_OGW}.
Without loss of generality, suppose $\g_0,\ldots,\g_M,$ form a basis for $W_\phi^L$,
and assume $A_j=\g_{i_j}$ for $j=1,\ldots,l.$
Let
\[
\g' = \sum_{j = 0}^M t_j\g_j.
\]
Let $\rho_R : R \to R$ denote the $\L$-algebra homomorphism given by
\begin{equation*}
\rho_R(s) = 0, \qquad \qquad
\rho_R(t_j) =
\begin{cases}
t_j, & j = 0,\ldots,M,\\
0, & j = M+1,\ldots,N.
\end{cases}
\end{equation*}
Let $\rho_Q:Q \to Q$ denote the homomorphism obtained by restricting $\rho_R$ to $Q \subset R$ and denote by
\[
\rho_C : C \to C, \qquad  \rho_D : D \to D, \qquad \rho_H : \Hh^{*}(X,L;Q) \to \Hh^{*}(X,L;Q),
\]
the induced module homomorphisms. These homomorphisms commute with the action of $\phi^*.$ Thus, the map $(\gamma,b) \mapsto (\rho_D(\gamma),\rho_C(b))$ preserves bounding pairs as well as the real condition. Also, $\rho_H(\g) = \g'.$
It follows from equation~\eqref{eq:cbp} that in the case that $L$ satisfies the hypothesis of Theorem~\ref{thm1} (resp. Theorem~\ref{thm2}),
\[
\varrho([\rho_D(\gamma),\rho_C(b)])=(\g',0) \qquad  \text{(resp. $\varrho_\phi([\rho_D(\gamma),\rho_C(b)])=(\g',0)).$}
\]

On the other hand, let $\Gg_0,\ldots,\Gg_{N^G},\Gg,$ and $(\gag,0)$ be as in Section~\ref{sssec:geosi}. Assume without loss of generality that $\Gg_0,\ldots,\Gg_M,$ coincide with $\g_0,\ldots,\g_M,$ and thus form a basis of $W_\phi.$
Also, since the sum on the left hand side of equation~\eqref{eq:superpots} is independent of $\sly,$
we assume without loss of generality that $\sly = \sly^G,$ so $q : \sly \to \sly^G$ is the identity map.
Let $\rho_R^G : R^G \to R$ denote the homomorphism given by
\[
\rho_R^G(\Tg^{\beta}) = T^{\beta}, \qquad \rho_R^G(\sg) = 0, \qquad \rho_R^G(\tg_j) =
\begin{cases}
t_j, & j = 0,\ldots,M,\\
0, & j = M+1,\ldots,N^G.
\end{cases}
\]
Let $\rho_Q^G:Q^G \to Q$ denote the homomorphism obtained by restricting $\rho_R$ to $Q^G \subset R^G$ and denote by
\[
\rho_D^G : D^G \to D, \qquad \rho_H^G : H^{*}(X;Q^G)^{-\phi^*} \to \Hh^{*}(X,L;Q)^{-\phi^*},
\]
the induced module homomorphisms. These homomorphisms commute with the action of~$\phi^*.$ Thus, the map $(\gag,0) \mapsto (\rho_D^G(\gag),0)$ preserves bounding pairs as well as the real condition. Also, $\rho_H^G(\Gg) = \g'.$
It follows from equation~\eqref{eq:cbpg} that in the case that $L$ satisfies the hypothesis of Theorem~\ref{thm1} (resp. Theorem~\ref{thm2}),
\[
\varrho([\rho_D^G(\gag),0])=(\g',0) \qquad  \text{(resp. $\varrho_\phi([\rho_D^G(\gag),0])=(\g',0)\;).$}
\]
Thus, Theorem~\ref{thm1} (resp. Theorem~\ref{thm2}) implies that
\[
(\rho_C(\gamma),\rho_D(b)) \sim (\rho_D^G(\gag),0),
\]
and Theorem~\ref{thm_inv} gives
\[
\rho_R(\Omega(\gamma,b)) = \Omega(\rho_D(\gamma),\rho_C(b)) = \Omega(\rho_D^G(\gag),0) = \rho_R^G(\Omega(\gag,0)).
\]
The desired result follows by taking derivatives on both sides of the preceding equation.
\end{proof}

\begin{proof}[Proof of Theorem~\ref{thm:penka}]
Combine Propositions~\ref{thm:penkasuperpot} and~\ref{pr:superpots}.
\end{proof}

\bibliography{../../bibliography_exp}
\bibliographystyle{../../amsabbrvcnobysame}

\end{document}